\documentclass{amsart}
\usepackage{graphicx}
\usepackage{subcaption}
\usepackage{subcaption}
\usepackage{amsmath}
\usepackage{amssymb}
\usepackage{ upgreek }
\usepackage{mathtools}
\usepackage{centernot}
\usepackage{stmaryrd}
\usepackage{microtype}
\usepackage{esvect}
\usepackage{amsthm}
\usepackage{upgreek}
\usepackage{comment}
\usepackage{amsfonts}
\usepackage{physics}
\usepackage{tikz-cd}
\usepackage[utf8]{inputenc}
\usepackage[english]{babel}
\usepackage{graphicx}
\usepackage[margin=1in]{geometry}
\usepackage{mathtools}
\usepackage{latexsym}
\usepackage{dsfont}
\usepackage{enumerate}
\usepackage[hidelinks]{hyperref}
\usepackage{multirow}

\graphicspath{{./Figures/}}

\newcommand{\Int}[1]{%
  {\kern0pt#1}^{\mathrm{o}}%
}

\newcommand{\bp}{\begin{pmatrix}}
\newcommand{\ep}{\end{pmatrix}}
\newcommand{\arrow}{\longrightarrow}

\newcommand{\sauceremph}[1]{\textbf{#1}}

\newcommand{\volthick}[2]{\vol^{#1}\prn{#2}}
\newcommand{\volsolid}[2]{\vol^{#1}\prn{#2}}
\newcommand{\volsphere}[2]{\vol^{#1}\prn{#2}}
\newcommand{\Dthick}[2]{D^{#1}\prn{#2}}
\newcommand{\Dsolid}[2]{D^{#1}\prn{#2}}
\newcommand{\Dsphere}[2]{D^{#1}\prn{#2}}

\newcommand{\oct}{\text{oct}}
\newcommand{\tet}{\text{tet}}

\usepackage{xargs}                  
\usepackage{xcolor}
\usepackage[colorinlistoftodos,prependcaption,textsize=tiny]{todonotes}
\usepackage{genericknotstyle}

\def\Wknot{W}
\def\Mknot{U}
\newcommand{\Wknoti}[1]{W_{#1}}
\newcommand{\Mknoti}[1]{U_{#1}}

\def\off{\Delta}

\newtheorem{construction}[theorem]{Construction}

\theoremstyle{remark}
\newtheorem{remark}[theorem]{Remark}
\newtheorem{warning}[theorem]{Warning}

\title{Lower Bounds on Volumes of Hyperbolic 3-Manifolds via Decomposition}
\author[C. Adams]{Colin Adams}
\address{Department of Mathematics and Statistics, 18 Hoxsey St., Williams College, USA}
\email{cadams@williams.edu}
\author[M. Capovilla-Searle]{Michele Capovilla-Searle}
\address{Department of Mathematics, University of Iowa, 14 MacLean Hall, Iowa City, Iowa 52242-1419}
\email{mcapovillasearle@uiowa.edu}
\author[D. Li]{Darin Li}
\address{4592 Terra Pl., San Jose, CA95121}
\email{darinli@gmail.com}
\author[L. Li]{Lily Qiao Li}
\address{Department of Mathematics, University of California, Berkeley, 970 Evans Hall, Berkeley, CA 94720-3840}
\email{lilyli@berkeley.edu}
\author[J. McErlean]{Jacob McErlean}
\address{Department of Mathematics, Duke University, 120 Science Dr, Durham, NC 27708-0320}
\email{jacob.mcerlean@duke.edu}
\author[A. Simons]{Alexander Simons}
\address{6903 Rosemont Drive, McLean VA 22101}
\email{ads4@williams.edu}
\author[N. Stewart]{Natalie Stewart}
\address{Mathematics Department, Harvard University, 1 Oxford St, Cambridge, MA 02138}
\email{nataliestewart@math.harvard.edu}
\author[X. Wang]{Xiwen Wang}
\address{Room 501, No. 3, Lane 555, Guangyan Road, Shanghai, China 200072}
\email{xiwenw2@gmail.com}

\begin{document}

\begin{abstract}In a variety of settings we provide a method for decomposing a 3-manifold $M$ into pieces. When the pieces have the appropriate type of hyperbolicity, then the manifold $M$ is hyperbolic and its volume is bounded below by the sum of the appropriately defined hyperbolic volumes of the pieces. A variety of examples of appropriately hyperbolic pieces and volumes are provided, with many examples from link complements in the 3-sphere.
\end{abstract}

\maketitle

\section{Introduction}\label{introduction section}

Throughout this paper, $M$ is either a compact 3-manifold or a compact manifold with a link removed.
We assume we are in the piecewise-linear category, and all manifolds may have boundary unless specified to be closed. When a manifold is said to be hyperbolic, then there are no spherical or projective plane boundaries, and there is a hyperbolic metric on the manifold obtained by removing torus boundaries and Klein bottle boundaries, with finite hyperbolic volume. 
When there are additional higher genus boundaries, for hyperbolicity, we further assume that they are realized as totally geodesic surfaces in the hyperbolic metric on the manifold. This  is called tg-hyperbolicity. We do not assume that surfaces in $M$ are connected.

When a 3-manifold possesses a hyperbolic metric, we can use the geometric invariants from hyperbolicity to understand the 3-manifold. In the case that the 3-manifold has a hyperbolic metric of finite hyperbolic volume, the Mostow-Prasad Rigidity Theorem says that this metric, and hence its volume, is a topological invariant.
A great amount of research has been devoted to approximating the volume of particular families of hyperbolic 3-manifolds. Much of this work has focused on link complements in the 3-sphere $S^3$.

For instance, in \cite{Lackenby}, it was proved that if $L$ is an alternating link in $S^3$ that is hyperbolic, then \[v_{\oct} (t(D) -2)/2 < \vol(S^3\setminus L )< 10 v_{\tet} (t(D) -1)\] where $t(D)$ is the number of twist regions in a twist-reduced alternating diagram of $L$, $v_{\oct}$ is the volume of an ideal regular octahedron, approximately 3.6638, and $v_{\tet}$ is the volume of an ideal regular tetrahedron, approximately 1.0149. Additional lower bounds for hyperbolic alternating link volumes appear in \cite{CKP2}.

If we restrict to  hyperbolic Montesinos links,  lower bounds for volumes that depend on the combinatorics of the link diagram are provided in \cite{FP} and \cite{FKP}. Other categories of links with lower bounds include prime link diagrams with twist-reduced diagrams of twist number at least two and all twist regions containing at least seven crossings \cite{FKP2}, homogeneously adequate links \cite{FKP3}, and positive braids \cite{FKP, Giambrone}.

However, in all cases, obtaining  accurate lower bounds on volume has proved elusive. In this paper, given a link in $S^3$ that can be appropriately decomposed along surfaces into tangles, we define certain types of hyperbolicity for the tangles which, when satisfied, provide an associated volume for each tangle.
Using results from \cite{Agol-Storm-Thurston}, as generalized in \cite{CFW}, we can then show that links obtained from the tangles are themselves hyperbolic. 
Further,  we obtain surprisingly accurate lower bounds on the volumes of  the links from the sum of the volumes of the constituent tangles. However, unlike the previous lower bounds that depend only on the combinatorics of the diagrams, this  method does assume that the volumes of the constituent tangles are known. We provide examples of the various volumes of tangles in the last section. 

These results for links in $S^3$ are particular cases of a much more general theorem that shows that if there is an appropriate decomposition of a 3-manifold into pieces, each of which is hyperbolic in the appropriate way, the manifold is hyperbolic with volume at least as large as the sum of the associated volumes of the pieces, versions of which we prove in Sections 2 and 3. In the rest of the introduction, we focus on the application to links in $S^3$.

For our purposes, a tangle is  a properly embedded disjoint union of $m$ arcs and potentially some additional circles in a ball.
These are sometimes called $2m$-tangles in the literature.

The first situation applies to a link $L$ in the open thickened torus $T \times (0,1)$, which we can think of as the complement of the Hopf link in $S^3$. 
Suppose $L$ decomposes into $n$ tangles $(\cT_i)_{i=1}^n$, arranged in a cycle that follows the core of $T \times (0,1)$ as pictured in Figure \ref{Torus cycle figure}(a). There may be an arbitrary nonzero number of connecting strands between any two adjacent tangles.

\begin{figure}
    \centering
    \includegraphics[width=.6\textwidth]{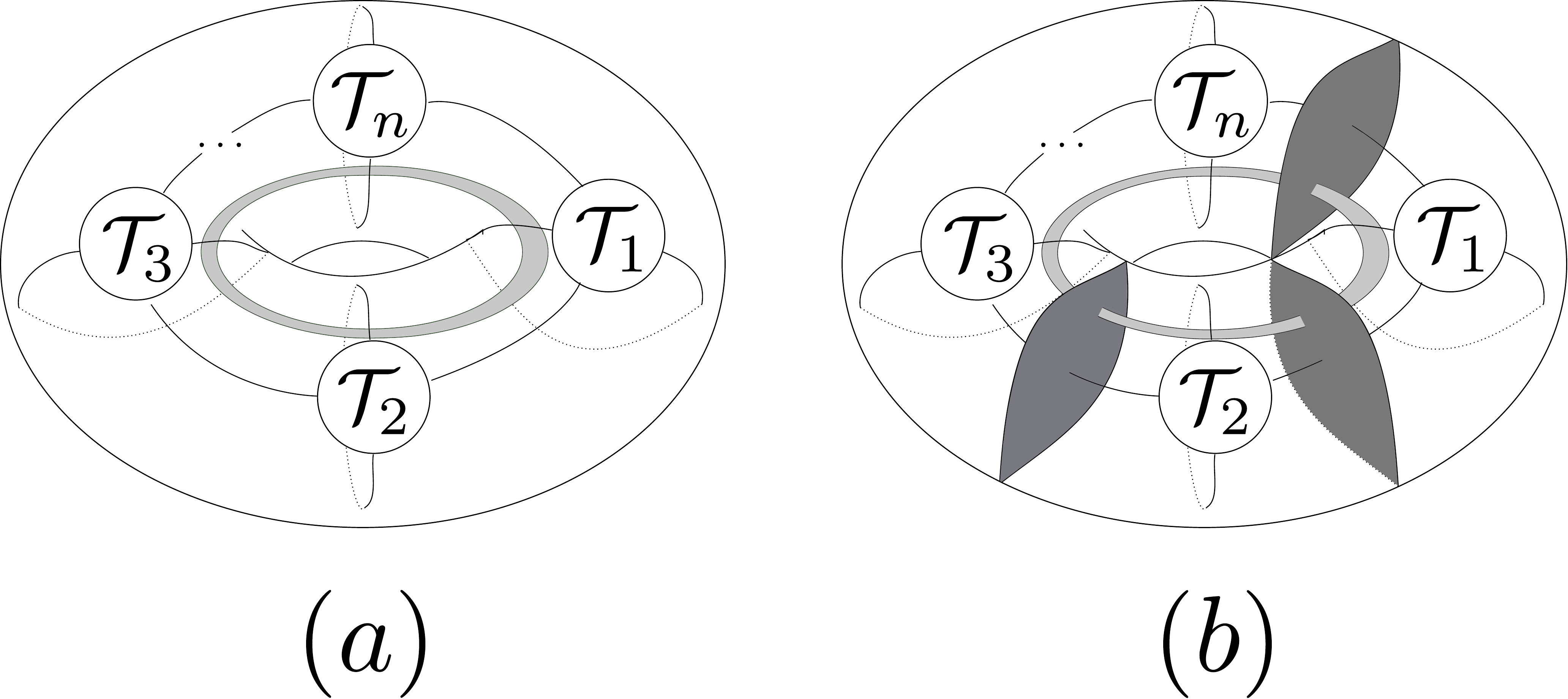}
    \caption{Part (a) shows a link decomposing into a cycle of tangles in $T \times (0,1)$.
    Each node is replaced with a tangle.
    Part (b) shows surfaces that decompose the link into tangles.
    }
    \label{Torus cycle figure}
\end{figure}

We may construct a surface $\Sigma \subset M \setminus L$ by taking the disjoint union of a collection of multiply-punctured annuli, each corresponding to a meridianal annulus between adjacent tangles as shown in Figure \ref{Torus cycle figure}(b). 
Cutting the manifold open along $\Sigma$ yields the disjoint union of the individual tangles in thickened cylinders ending at the ``boundary'' meridianal annuli.
Since there is a homeomorphism of $T \times (0,1)$ taking meridians to $(p,q)$-curves, we could have just as easily cut $L$ along a collection of annuli which bound $(p,q)$ curves on each component of $T \times \partial I$. 

We refer to such a tangle as a \emph{thickened-cylinder tangle}, and we associate with it a link in a thickened torus called its \emph{double} or \emph{2-replicant} $\Dthick{2}{\cT_i}$, formed by doubling across the image of $\Sigma$.
If this link is hyperbolic in $T \times (0,1)$, we say that $\cT_i$ is \emph{2-hyperbolic}, and define the 2-volume $\volsolid{2}{\cT_i}$ to be half of the volume of this double.

We prove that if   $L \subset T \times (0,1)$ is a link in a thickened torus such that there exist surfaces separating $L$ into thickened-cylinder tangles $\cT_i$ and each $\cT_i$ is 2-hyperbolic,
    then $L$ is hyperbolic, and its volume satisfies
    \[
        \vol(L) \geq \sum_i \volsolid{2}{\cT_i}.
    \]

If a given thickened-cylinder tangle can be further decomposed into a cycle, we may ask about hyperbolicity of the individual tangles in cubes, called \emph{cubical tangles}.
An example of this is illustrated in Figure \ref{thickened torus example}.

A cubical tangle $\cT_{ij}$ doubles to a thickened-cylinder tangle.
We say that a cubical tangle is \emph{$(2,2)$-hyperbolic} if the double of its double $\Dthick{(2,2)}{\cT_{ij}}$ is a hyperbolic link, in which case we define the \emph{$(2,2)$-volume}
\[
    \volthick{(2,2)}{\cT_{ij}} := \frac{1}{4} \vol\prn{\Dthick{(2,2)}{\cT_{ij}}}
\]

In fact, given a decomposition of a fixed thickened-cylinder tangle $\cT_i$ into cubical tangles $\cT_{ij}$, the double of the thickened-cylinder tangle can be realized as a cycle of thickened-cylinder tangles, each of which is doubled from some $\cT_{ij}$.
Then, if  $\cT_i$ is  a thickened-cylinder tangle, and there are  surfaces separating $\cT_i$ into a cycle $\cT_{ij}$ of cubical tangles and each $\cT_{ij}$ is $(2,2)$-hyperbolic,
    then, $\cT_i$ is $2$-hyperbolic, with volume satisfying
    \[
        \volthick{2}{\cT_i} \geq \sum_j \volthick{(2,2)}{\cT_{ij}}.
    \]
    Moreover, if a link $L$ in a thickened torus decomposes into cubical tangles such that each $\cT_{ij}$ is $(2,2)$-hyperbolic, then $L$ is hyperbolic with volume satisfying
    \[
        \vol(L) \geq \sum_{ij} \volthick{(2,2)}{\cT_{ij}}.
    \]

In a future paper \cite{selectalternating}, we will give a large class of $2$-hyperbolic thickened-cylinder tangles, called \emph{select alternating thickened cylinder tangles} and a class of $(2,2)$-hyperbolic cubical tangles called \emph{select alternating cubical tangles}.
As a corollary to the second of those, it is shown that all rational tangles in a cube  with one endpoint on each edge of the square, other than the integer tangles 0, $\pm 1$, and $\infty$,  are (2,2)-hyperbolic cubical tangles.

\begin{figure}[htbp]
    \centering
    \includegraphics[scale=.15]{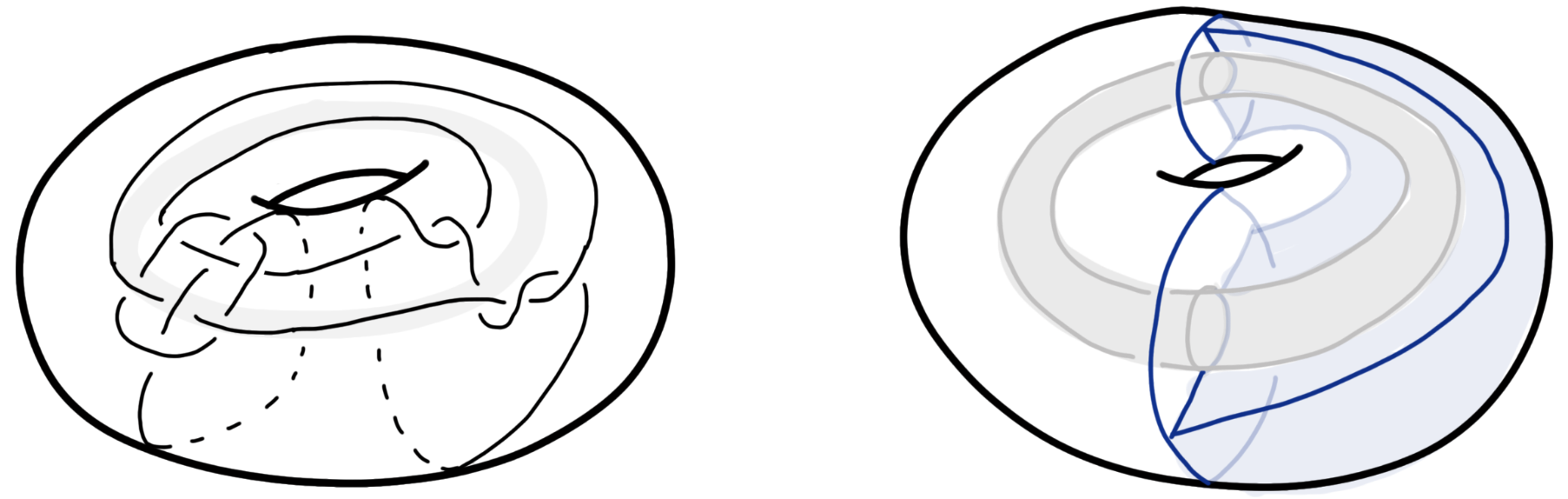}
    \caption{A link $L$ in $T \times I$ and a decomposition of $T \times I$ into a thickened cylinder  and two cubes corresponding to the tangles. One of the cubes is shaded.}
    \label{thickened torus example}
\end{figure}

\begin{example} Consider the link $L$ in $T \times (0,1)$ shown in Figure \ref{thickened torus example}. We can separate $L$ along a surface into a thickened-cylinder tangle $\cT_1$ on the left, and two cubical tangles $\cT_{21}$ and $\cT_{22}$ on the right. By \cite{selectalternating}, we know $\cT_1$ is 2-hyperbolic, while $\cT_{21}$ and $\cT_{22}$ are both $(2,2)$-hyperbolic. Thus, by Theorem \ref{Torus theorem} and Theorem \ref{Square tangles theorem}, $\mbox{vol}^2_{T \times I}(\cT_1) + \mbox{vol}^{(2,2)}_{T \times I}(\cT_{21}) + \mbox{vol}^{(2,2)}_{T \times I}(\cT_{22})$ gives a lower bound of 34.7259 on the volume of $L$. The actual volume is 37.8023.
\end{example}

The next result applies to a link $L$ in $S^3$ constructed as a cycle of $2n$ tangles, which we call a {\it bracelet link}. We define the length of the bracelet link to be $2n$.
For the case $n=2$, this is pictured in Figure \ref{braceletlink}. Each tangle is connected to the subsequent tangle in the cycle by two or more strands. For example,  a Montesinos link is an example of a bracelet link, where the individual tangles are rational tangles and there are exactly two strands between adjacent tangles (absorbing the additional twist allowed into one of the tangles). If the number of tangles used in the construction of the Montesinos link is odd, we can treat one adjacent pair of tangles as a single tangle so that the number of tangles is even.

\begin{figure}[htbp]
    \centering
    \includegraphics[scale=1.2]{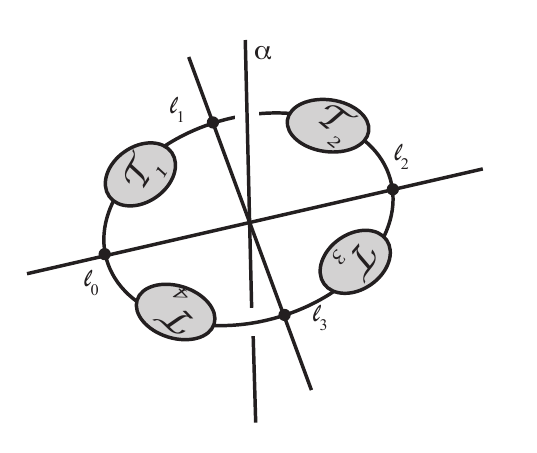}
    \caption{A bracelet link of length 4.}
    \label{braceletlink}
\end{figure}

 Given a bracelet link, we cut the 3-sphere into ``flying" \emph{saucers} $S_1, \dots, S_{2n}$ obtained by taking $n$ spheres all sharing the circle $\alpha$ in $S^3$ corresponding to the vertical axis at the center of the link and each cutting through the strands connecting two adjacent tangles. The result is $2n$ tangles, each in a  unique saucer, the boundary of which is two bowed disks sharing the central circle. The endpoints of the tangle lie on the two disks. We call the $i$th saucer $S_i$ and its tangle $\cT_i$, as in Figure \ref{saucer}. For convenience, in certain situations we let $\cT_i$ represent both the tangle and its saucer. We call the two punctured disks that make up the boundary of the saucer the \emph{faces} of the saucer. 
  Let $\ell_i$ be the number of stands connecting $\cT_i$ and $\cT_{i+1}$, and 
 $\ell_{0}$ the number of strands connecting $\cT_{2n}$ with $\cT_1$. Note that $\ell_i \geq 2$. 
 
 \begin{figure}[htbp]
    \centering
    \includegraphics[scale=.7]{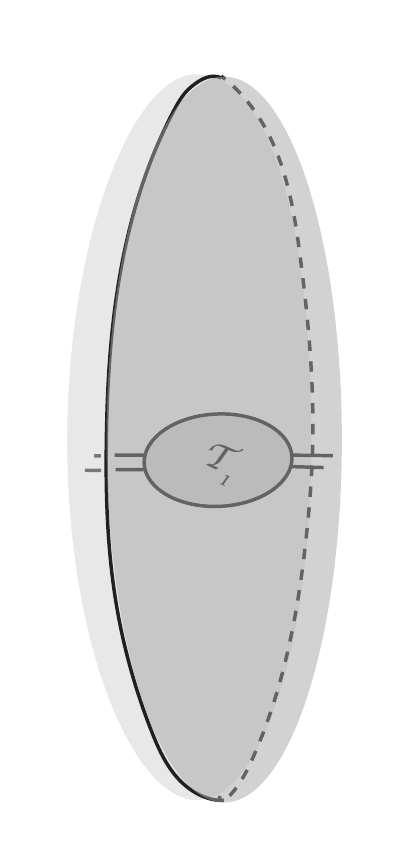}
    \caption{A saucer and tangle.}
    \label{saucer}
\end{figure}

We define the \emph{$2n$-replicant} $\Dsphere{2n}{\cT_i}$ of $\cT_i$ to be the bracelet link formed by the cycle $(\cT_i,\cT_i^R,\cT_i, \dots,\cT_i^R)$ of $2n$ tangles where $\cT_i^R$ is the ``reflected tangle'' obtained by reflecting the saucer for $\cT_i$ through the face containing the endpoints to one side of $\cT_i$ and perpendicular to the strands.
So $\cT_i$ has $\ell_{i-1}$ endpoints on the first face and $\ell_i$ on the second, whereas $\cT_i^R$ has $\ell_i$ on the first face and $\ell_{i-1}$ on the second. See Figure \ref{braceletreflection}.

\begin{figure}[htbp]
    \centering
    \includegraphics[scale=1.2]{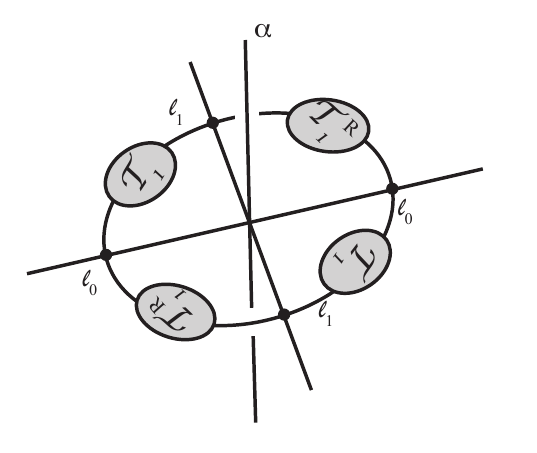}
    \caption{The $2n$-replica $\Dsphere{2n}{\cT_i}$ of $\cT_i$.}
    \label{braceletreflection}
\end{figure}

 We say that $\cT_i$ is \emph{$2n$-hyperbolic} if $\Dsphere{2n}{\cT_i}$ is hyperbolic, in which case we define
\[
    \volsphere{2n}{\cT_i} := \frac{\vol\prn{\Dsphere{2n}{\cT_i}}}{2n}.
\] 

We prove that if $L$ is a bracelet link made of a cycle $(\cT_i)_{i=1}^{2n}$ of $2n$ saucer tangles such that each $\cT_i$ is $2n$-hyperbolic, then $L$ is hyperbolic and the volumes satisfy
    \[
        \vol(L) \geq \sum_i \volsphere{2n}{\cT_i}.
    \]

\begin{figure}[htbp]
    \centering
    \includegraphics[scale=.5]{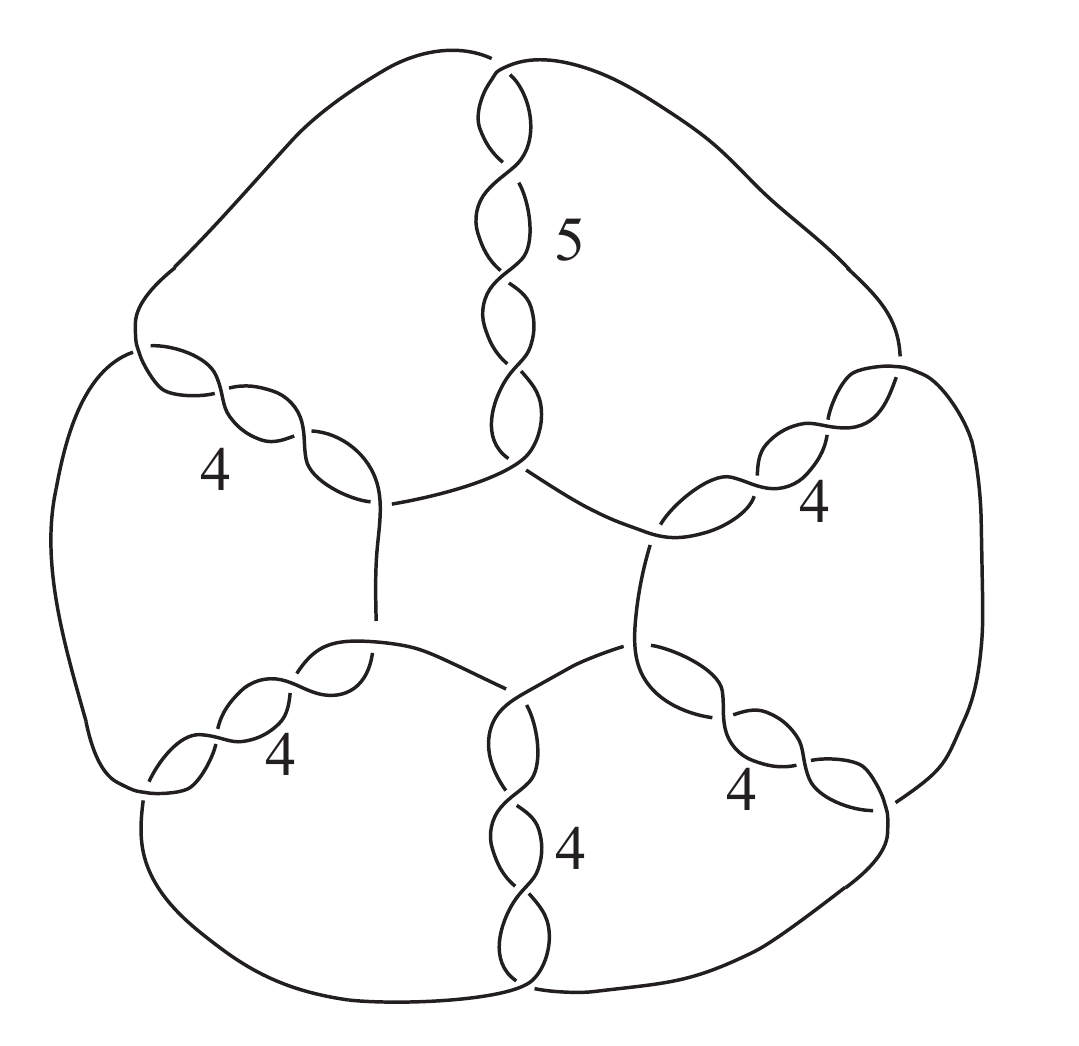}
    \caption{A bracelet link.}
    \label{pretzellink}
\end{figure}

\begin{example}
To see an explicit example, consider the length six bracelet link projection of Figure \ref{pretzellink}. The SnapPy computer program \cite{SnapPy} yields a volume of approximately 32.9819. 
The link is prime, alternating and twist-reduced with twist number 6, so Lackenby's lower  bound on volume applies and yields $2v_{\oct} \approx 7.3276.$ The lower bound that comes from \cite{FP} and \cite{FKP} is $3v_{\oct} \approx 10.9914.$ Other known bounds are lower than these. To apply our theorem, we split the link into six saucer tangles along the union of three spheres all intersecting at the central circle, five of the tangles of which are vertical integer 4 tangles and one of which is a vertical integer 5 tangle. Summing $\volsphere{6}{\cT_i}$ for $i = 1, \dots, 6$, we obtain a lower bound of 32.7858.  

The theorem also implies that if we alter this bracelet link by replacing some of the tangles by their reflections, so that the resulting link is non-alternating, it will still be hyperbolic and the same lower bound on volume applies. All of the volumes for this collection of links are lower than the volume of the alternating bracelet link with which we started and the least volume is realized for the links where half of the tangles have been reflected. The corresponding volume for these is approximately 32.7926. Thus, in this case, the lower bound is within .01 of the actual volume.
\end{example}

 
 In certain situations, we  can further decompose the saucer tangles. A tangle in a tetrahedron such that all endpoints are on the cross-sectional square as in Figure \ref{tetrahedraltangle}, with at least one endpoint on each edge of the square and none at the corners, is called a {\it tetrahedral tangle.}

 A saucer tangle can sometimes be decomposed into tetrahedral tangles. In particular, as in Figure \ref{toruslattice}, certain links form an $(m,n)$-rectangular lattice of tangles on an unknotted torus $T$. In this case, the 3-sphere can be decomposed into $mn$ tetrahedra, each containing a tetrahedral tangle. For each tetrahedron, one edge lies on the core curve of the solid torus to the inside of $T$ and is shared by $n$ other tetrahedra, and one edge lies on the core curve of the outside solid torus,  and is shared by $m$ tetrahedra. The remaining four edges of each tetrahedron are each  shared by four others. In particular, on the standardly embedded torus, there are $m$ columns and $n$ rows of tangles.

 \begin{figure}[htbp]
    \centering
    \includegraphics[scale=.7]{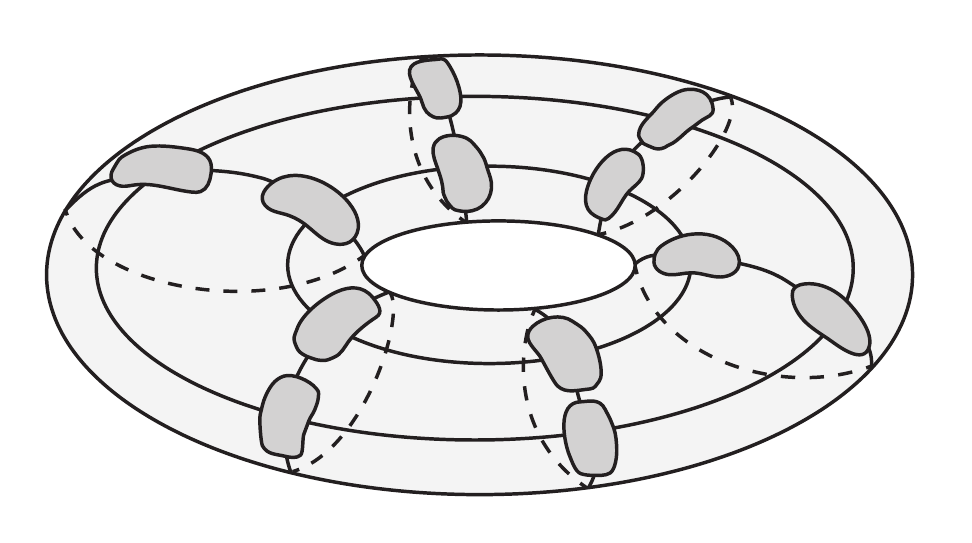}
    \caption{A link given as a torus lattice in $S^3$.}
    \label{toruslattice}
\end{figure}

\begin{figure}[htbp]
  \begin{subfigure}[b]{0.2\textwidth}
    \includegraphics[width=\textwidth]{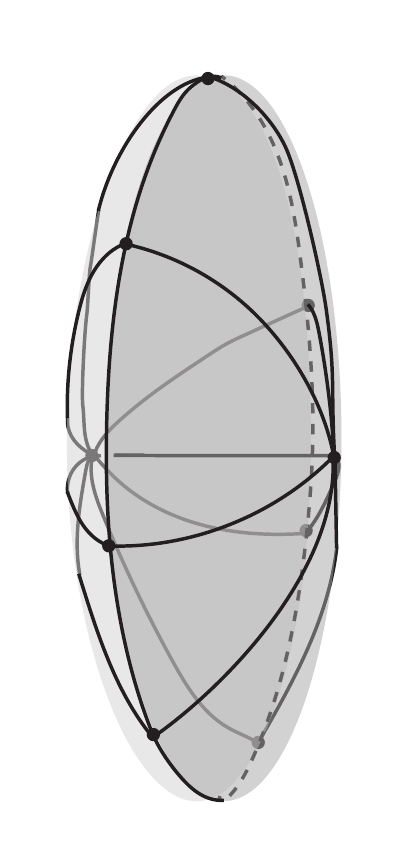}
    \caption{Decomposing a saucer into tetrahedra.}
    \label{fig:f1}
  \end{subfigure}
  \begin{subfigure}[b]{0.2\textwidth}
    \includegraphics[width=\textwidth]{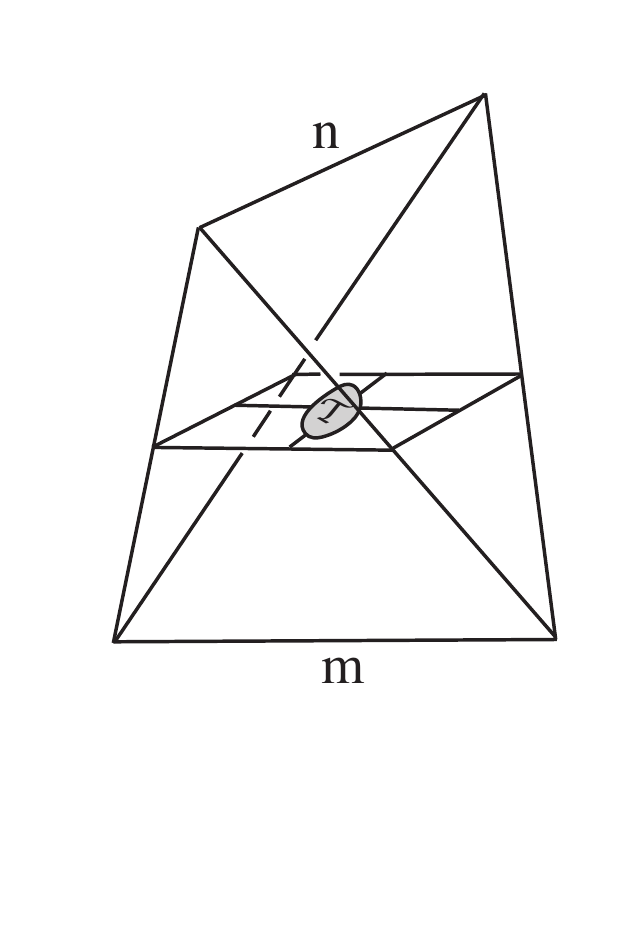}
    \caption{A tetrahedral tangle.}
    \label{tetrahedraltangle}
  \end{subfigure}
  \caption{Decomposing saucer tangles into tetrahedral tangles.}
  \label{suacertet2}
\end{figure}


Given a tetrahedral tangle $\cT_i$, we can form the $(2m, 2n)$-replicant, denoted $D^{(2m, 2n)}(\cT_i)$ by creating a $2m \times 2n$ torus lattice link with copies of $\cT_i$ replacing each vertex, such that two connected tangles are reflections of one another.

We say that a tetrahedral tangle $\cT_i$ is \emph{$(2m, 2n)$-hyperbolic} if $\Dsphere{(2m,2n)}{\cT_i}$ is hyperbolic, in which case we define
\[
    \volsphere{(2m,2n)}{\cT_i} := \frac{\vol\prn{\Dsphere{2n}{\cT_i}}}{4mn}.
\]

We show that if  $L$ is a $2m \times 2n$ torus lattice link made of a collection $\{\cT_{i,j}\}_{i=1}^{2n}$ of $4mn$ tetrahedral tangles such that each $\cT_{i,j}$ is $(2m,2n)$-hyperbolic, then, $L$ is hyperbolic and the volumes satisfy
    \[
        \vol(L) \geq \sum_{i,j} \volsphere{(2m,2n)}{\cT_{i,j}}.\]

In fact, it is not necessary that every saucer tangle decompose into tetrahedral tangles. If even a subset of them do, we obtain a similar result. 

\begin{example} \label{tettangleexample} The tetrahedral tangle $\cT$ in its projection to the cross-sectional square in Figure \ref{tetrahedralbigon}(a) is $(2m,2n)$-hyperbolic for all $m,n \geq 1$; this can be checked for $m = n = 1$ explicitly, and extended to arbitrary $(2m,2n)$ by Theorem \ref{Saucer hyperbolicity theorem}.
Thus any $(2m, 2n)$-torus lattice link in $S^3$ constructed from copies of just this tangle is hyperbolic and has volume at least $4mn \volsphere{(2m,2n)}{\cT}$. For example, since $\volsphere{(2,2)}{\cT} = 3.1322306$, any $(2,2)$-torus lattice link in $S^3$ constructed by placing this tangle or its reflections over a vertical or horizontal line into the four slots in the (2,2)-lattice has volume at least 12.5289226. The lower bound on volume is realized by the link in Figure \ref{tetrahedralbigon}(b) obtained by placing the tangles in the slots so each is a reflection of the others across the separating surfaces, making those surfaces totally geodesic. This is in fact the $(2,2)$-replicant of the tangle.

However, caution must be exercised. If the tangle is rotated $90^\circ$ before being placed in the same tetrahedron or if its crossings are switched, it is no longer $(2,2)$-hyperbolic. And if that tangle is included as an option in the (2,2)-lattice, the link may no longer be hyperbolic and even if it is, the volume of the link is no longer necessarily bounded below by 12.5289226. 

If on the other hand, we considered this tangle as a tangle in a cube, it  and its reflections and all of their rotations are (2,2)-hyperbolic in the thickened torus.
\end{example}

\begin{figure}[htbp]
    \centering
    \includegraphics[scale=.7]{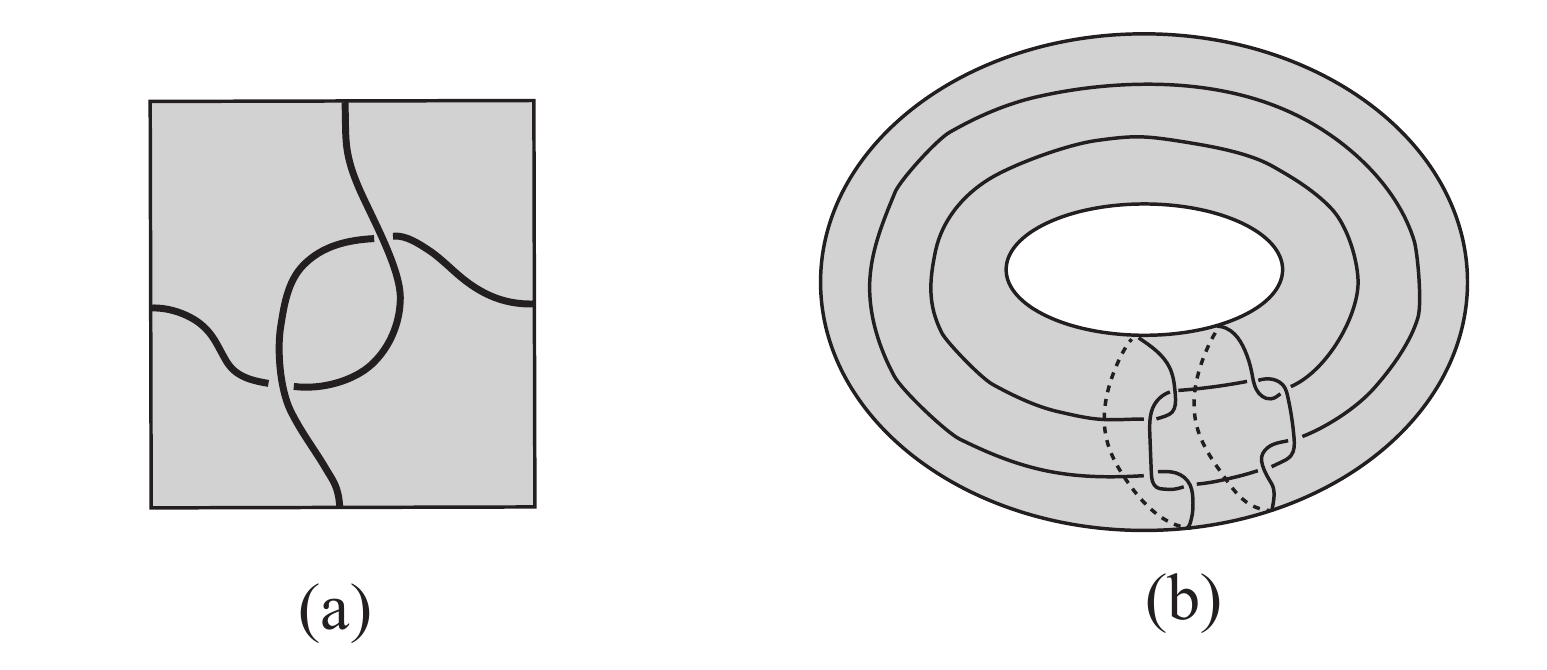}
    \caption{A 2-tangle and a (2,2)-lattice link generated by copies of it.}
    \label{tetrahedralbigon}
\end{figure}

Similar results for tangles that we have described in the thickened torus and in the 3-sphere also apply to links in an open solid torus, which we can think of as the complement of a single unknotted component, as well as to links in $S^2 \times S^1$.

The proofs of the main result and generalizations rely on the following theorems.
The first is well known \cite[Thm 1.10.15]{Klingenberg}.
\begin{theorem}\label{fixed surface lemma}
    Let $M$ be a Riemannian manifold and let $f:M \rightarrow M$ be a nontrivial isometry.
    Then, the subset $\Fix(f)$ is a union of embedded totally geodesic submanifolds.
\end{theorem}

The second is  due to Agol, Storm, and Thurston in  \cite[Thm. 9.1]{Agol-Storm-Thurston}, as  generalized by Calegari, Freedman and Walker in \cite[Thm. 5.5]{CFW}.
\begin{theorem}\label{Cutting and pasting theorem}
    Let $\overline{M}$ be a compact manifold with interior $M$ a hyperbolic 3-manifold of finite-volume.
    Let $\Sigma$ be a properly embedded two-sided totally geodesic surface, let $\psi:\Sigma \rightarrow \Sigma$ be a diffeomorphism, and let $M'$ be the manifold formed by surgering along $\Sigma$ and gluing the resulting pieces together via $\psi$.
        Then, $M'$ is a hyperbolic 3-manifold of finite volume, satisfying
        \[
            \vol(M') \geq \vol(M).
        \]
        Equality is attained if and only if $\Sigma$ is totally geodesic in $M'$.
\end{theorem}

We describe the idea of the proof of our main result for bracelet links. We begin with a copy of the $2n$-replicant of each tangle making up the bracelet link. Each of the decomposition surfaces in these link complements are totally geodesic by Theorem \ref{fixed surface lemma}. Then by repeated cutting and regluing, always along totally geodesic surfaces, we obtain $2n$ copies of the original link complement. Then Theorem \ref{Cutting and pasting theorem} implies that the resulting set of copies must have total volume greater than the sum of the volume of the replicants, and therefore the original link complement has volume greater than the sum of the volumes of the individual tangles. This turns a geometric question into a combinatorial question. We answer that question in a very general setting in Section 2, where the main theorem is stated and proved. 

In Section 3, we generalize the main theorem in various ways.
In Section 4, we apply these results to prove the theorems listed above for the various link complements in $S^3$. 
In Section 5,  we consider the determination of appropriate hyperbolicity for various kinds of tangles. 

In particular, we prove that if  $\cT$ is a $2n$-hyperbolic saucer tangle for some $n \geq 1$,
    then, $\cT$ is $2m$-hyperbolic for all $m \geq n$.
Thus, once we know a given tangle is $2n$-hyperbolic, and so it can be used to prove hyperbolicity and lower bound volume of a  $2n$-bracelet link, it can also be used in any bracelet link of greater length.
A corollary to this result shows that if a tetrahedral tangle is $(2m,2n)$-hyperbolic, then it is $(2r,2s)$-hyperbolic for all $r \geq m$ and $s \geq n$.

In Section 6, we consider explicit examples of saucer tangles. We use results of \cite{Volz}, \cite{Wu}  and Theorem \ref{Saucer hyperbolicity theorem} to  determine $2n$-hyperbolicity for all arborescent tangles.
We go on to prove certain solid-cylinder tangles are 2-hyperbolic by relating them to $2m$-hyperbolic saucer tangles having alternating projections.

Both the  cyclic graph on $2n$ vertices and the torus lattice graph on $mn$ vertices are examples of graphs in $S^3$ such that the replacement of the vertices with tangles and the edges with strands connecting the tangles yields a link for which a version of Theorem \ref{Main theorem separating} holds. That is to say, if all of the constituent tangles of a given link $L$ are hyperbolic in the appropriate way, then they each have an associated volume and the link is hyperbolic with volume at least as large as the sum of the volumes of the tangles.  We call such a graph a \emph{tangle reflection graph}. In Section 8, we extend beyond cyclic graphs and torus lattice graphs to a variety of other tangle reflection graphs.

In Section 9, we provide tables of various types of volumes for some tangles. 

\section{The main theorem for a separating starburst}\label{Main theorem separating section}
\subsection{Definitions and statement of the theorem}

In this section we prove that if given an appropriately defined collection of surfaces in a 3-manifold that together intersect in a 1-manifold, such that the pieces they decompose the manifold into are appropriately hyperbolic, then the manifold is hyperbolic with volume at least as large as the sum of the appropriately defined volumes of the pieces.

In the case that the manifold $N$ has boundary components of genus greater than one, we double the manifold across these boundaries. The resulting manifold $N'$ has only torus boundary, and the collection of surfaces doubles to a collection of surfaces in the new manifold. Then we can work with that manifold. The manifold $N$ is tg-hyperbolic if and only if the manifold $N'$ is hyperbolic. Thus, from now on we will work with manifolds with no higher genus boundary, keeping in mind the results for those manifolds imply results for those that have such boundary components.

We motivate the more general results by returning to  the example of a bracelet link $L \subset S^3$ that decomposes into a cycle of $2m$ saucer tangles $(\cT_i)_{i=1}^{2m}.$
There is a collection $\bS = (\Sigma_i)_{i=1}^m$ of surfaces in $M = S^3 \setminus L$ that separates $M$ into $2m$ pieces that may each be viewed as the complement of a particular tangle $\cT_i$ in a saucer.
Local to the locus where these surfaces intersect, these resemble the product of a circle with Figure \ref{newstarburstfigure}(a).
We cut  along $\bS$ and analyze the resulting pieces, generalizing to the following situation.
\begin{definition}
    Let $M$ be a compact orientable 3-manifold, possibly with torus boundaries, and let $\bS = \cbr{\Sigma_1,\dots,\Sigma_m}$ be a collection of $m$ two-sided  surfaces properly embedded in $M$.
    We denote by $M \dbs \bS$ the manifold $M \dbs (\Sigma_1 \cup \cdots \cup \Sigma_m)$ obtained by removing the interior of a regular neighborhood of the union of the surfaces. 
    
    We assume that any point in the intersection set of two of the surfaces is in fact in the intersection set of all of the surfaces, which is to say, the intersection set $E = \Sigma_i \cap \Sigma_j = \bigcap _{k = 1}^m \Sigma_k$ for any $i \neq j$. We further assume that $E$ is a compact 1-manifold properly embedded in $M$.
    
    We suppose that  there is a regular neighborhood $U$ of $E$ such that the induced pair $(U,U \cap \cup_{k=1}^m \Sigma_k)$ is diffeomorphic to $(D \times E, S_m \times E)$, where $D$ is an open disk and $S_m$ is the union of $m$ open line segments $L_i$ in $D$, in cyclic order, each passing directly through a single intersection point at the center,  as in Figure \ref{newstarburstfigure}(a). 
    We further require that the diffeomorphism takes $U \cap \Sigma_i$ to $L_i \times E$ in cyclic order. 
    Suppose further that $M \dbs \bS$ can be expressed as the disjoint union of $2m$ 3-manifolds, denoted $M_1, \dots M_{2m}$, each intersecting the boundaries of the regular neighborhoods of two of the surfaces $\Sigma_j$ with indices differing by one modulo $2m$, as in Figure \ref{newstarburstfigure}(b).
     
    Then we say that $(M, \bS)$ is an {\bf $m$-sheeted separating starburst}.
\end{definition}

For convenience we allow the case of $E = \emptyset$ in the definition. 

\begin{figure}
    \centering
    \includegraphics[scale=0.6]{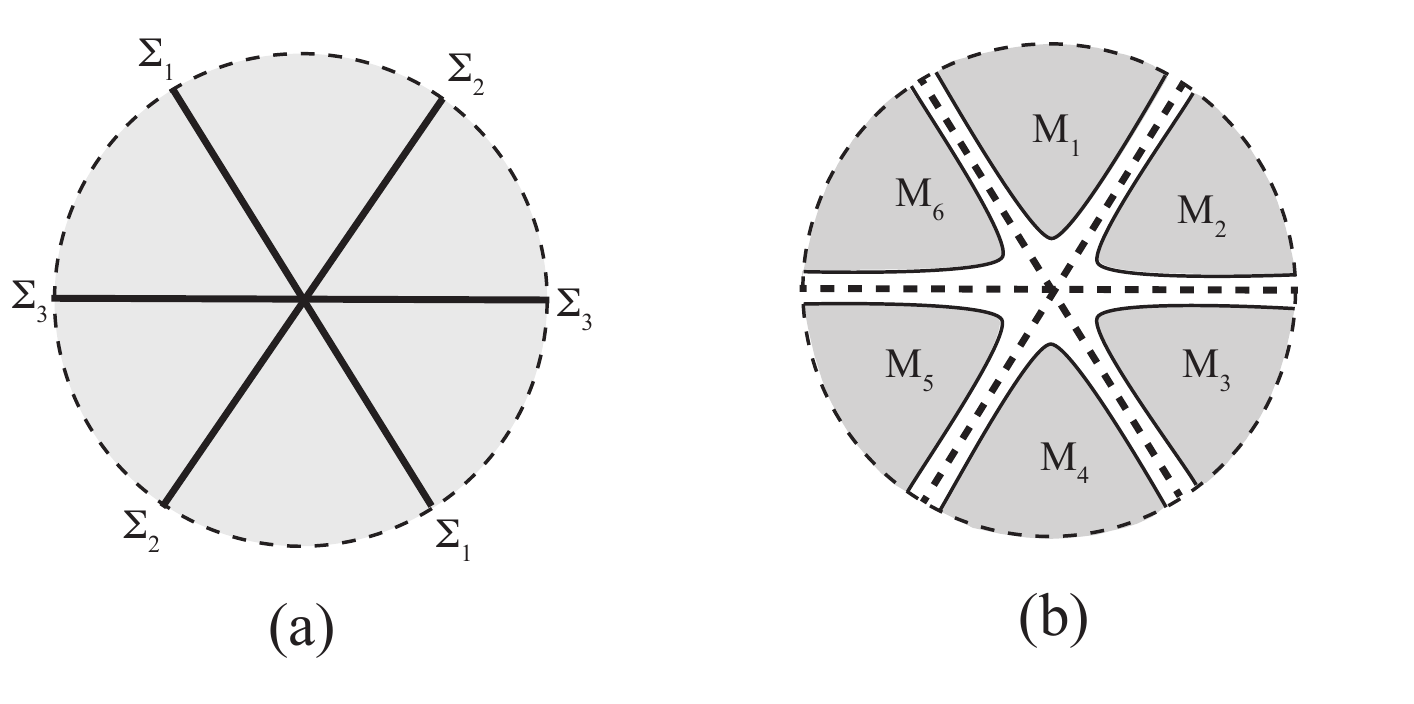}
    \caption{The starburst $S_m$ for $m = 3$ appears in (a). The cross section of the neighborhood $U$, after the removal of the regular neighborhood of the union of the surfaces appears in (b). 
    }
    \label{newstarburstfigure}
\end{figure}

In the case of a bracelet link $N = S^3 \setminus L$, the set $\bS$ consists of $m$ punctured spheres, all intersecting in a single circle. 
The manifold $N \dbs \bS$ is a disjoint union of $2m$ connected 3-manifolds $N_i$, each  given by the complement of a tangle in a saucer. 

We seek to associate with each tangle complement $N_i$ a notion of hyperbolicity, and we do so by bootstrapping the definition of hyperbolicity defined for links.
As defined in the introduction, in the case of a saucer tangle, there is a well-defined notion of a \emph{$2m$-replicant} formed by gluing together a cycle of $2m$ copies of $\cT_i$ along compatible endpoints.
We say that $\cT_i$ is \emph{$2m$-hyperbolic} if its $2m$-replicant is a hyperbolic link.

This gluing occurs by specifying surfaces $V^1_i,V^2_i \subset \partial N_i$ such that  $V^1_i \cap V^2_i = E$, given by the \emph{counter-clockwise pointing} and \emph{clockwise pointing} sides of the subsurface of $\partial N_i$ corresponding with $\bS$, and gluing copies of $V^j_i$ in different copies of $N_i$ together in a cycle. 

In analogy with the case of tangles, we may define the $2m$-replicant of this data as follows.

\begin{definition}
    A \emph{piece} is a pair $(P,(V^1,V^2))$ where $P$ is a 3-manifold with distinguished surfaces $V^1,V^2 \subset \partial P$ such that $V^1 \cap V^2$ is a 1-manifold (possibly disconnected and possibly with boundary).
    We may define the $2m$-replicant of a piece $D^{2m}(P,\prn{V^1,V^2})$ by reflecting $2m$ times over $V^1$ or $V^2$ as follows.
    \begin{equation}\label{2n replication equation}
        D^{2m}(P,\prn{V^1,V^2}) := \frac{\coprod_{i=1}^{2m} P}{(V^1,2i) \sim (V^1,2i+1) \hspace{10pt} \text{ and } \hspace{10pt} (V^2,2i-1) \sim (V^2,2i)}
    \end{equation}
    Here, arithmetic is performed modulo $2m$.
    We say that $(P,\prn{V^1,V^2})$ is \emph{$2m$-hyperbolic} if $D^{2m}(P,\prn{V^1,V^2})$ is hyperbolic, in which case we define
    \[
        \vol^{2m}(P,\prn{V^1,V^2}) := \frac{1}{2m} \vol(D^{2m}(P,\prn{V^1,V^2})).
    \]
    In the case that the surfaces $(V^i)$ are clear, we will often simply refer to such a piece as $P$, and write $D^{2m}(P)$ and $\vol^{2m}(P)$ for the replicant and volume. 
\end{definition}
\begin{remark}
    The order of $V^1$ and $V^2$ can be interchanged without changing the diffeomorphism class of the replicant.
    However, we use the syntax of ordered pairs for reasons that will be clear in the later \emph{$\ell$-piece} definition.
\end{remark}

We let $H$ represent the surface that is the closure of the complement of $V^1 \cup V^2$ in $\partial P$. Then $H$ corresponds to boundary components in $D^{2m}(P,\prn{V^1,V^2})$. When  $D^{2m}(P, (V^1, V^2))$ has  boundary components of genus one, keep in mind that we remove that boundary component when determining the hyperbolicity of $D^{2m}(P, (V^1, V^2))$.

For an $m$-sheeted separating starburst $(M, \bS)$, the manifold $M \dbs \bS$ is a disjoint union of $2m$ pieces, each with two possibly disconnected surfaces  on its boundary corresponding to two consecutive surfaces in the starburst. We denote the $i$th piece and its two distinguished surfaces on its boundary by $(M_i,(V^1_i, V^2_i))$ where gluing $V^2_i$ in the boundary of the $M_i$ to $V^1_{i+1}$ in the boundary of $M_i$, for all $i$ cyclically returns us to the manifold $M$. 

We will prove the following result, which demonstrates that one can use $2m$-hyperbolicity of suitably defined pieces of a manifold to prove hyperbolicity and also to obtain a lower bound on the volume of the manifold. This version of the theorem applies to bracelet links in the 3-sphere.

\begin{theorem}\label{Main theorem separating}
    Suppose $(M,\bS)$ is an $m$-sheeted separating starburst in a 3-manifold $M$ with all boundary components tori, and $(M_i,(V^1_i,V^2_i))_{i=1}^{2m}$ are the associated pieces. If 
    $E$ does not intersect $\partial M$ and if all $M_i$ are $2m$-hyperbolic, then $M$ is hyperbolic, with volume satisfying
    \[
        \vol(M) \geq \sum_{i=1}^{2m} \vol^{2m}(M_i).
    \]  
    Equality is attained if and only if each surface $\Sigma_i \subset M$ is totally geodesic.
\end{theorem}

\begin{warning}
    The notation $V^j_i$ is suppressed in the volume bound because the surfaces are chosen canonically, but in general, differing choice of surfaces $V^j_i$ may yield different geometry, which may not obey the volume bound.
\end{warning}

In the case that $E = \emptyset$, so the surfaces $\{\Sigma_i\}$ do not intersect one another, the result is an immediate corollary of Theorem \ref{Cutting and pasting theorem}. For each $i$, we form $D^{2m}\prn{M_i}$, and then, cutting and pasting along the $2m$ totally geodesic surfaces in each of the $D^{2m}\prn{M_i}$, we can form $2m$ copies of $M$, with total volume at least as large as the total volume of the $2m$-replicants. Dividing by $2m$ yields the result. 
This is the case for thickened-cylinder tangles in the thickened torus, for instance. 

We henceforth assume $E \neq \emptyset$.
In this case, we prove Theorem \ref{Main theorem separating} in Subsection \ref{Separating subsection}.
Then, we consider a generalization of this to \emph{separating starbursts in a piece} in Subsection \ref{Separating inductive subsection}.

A further generalization holds in the case that $M \dbs \bS$ does not necessarily decompose as the disjoint union of $\cbr{M_i}$, 
$E$ may intersect $\partial M$, and $M$ is not necessarily orientable.
These will be considered in Section \ref{Main theorem nonseparating section}.

\subsection{Proof of the main theorem for a separating starburst in a manifold}\label{Separating subsection}
For the duration of this subsection, we fix the index $m$, starburst $(M,\bS),$ and resulting pieces $(M_i,(V^1_i,V^2_i))$ as in Theorem \ref{Main theorem separating}, and we assume that each $(M_i,(V^1_i,V^2_i))$ is $2m$-hyperbolic.
It will help to refer to Example \ref{Algebra example} while following the proof.
We suppress the notation $V^j_i$ most of the time. 

Note that both $M$ and the $2m$-replicants $D^{2m}(M_i,(V^1_i,V^2_i))$ are formed by gluing together $2m$ elements of $\cbr{M_i}_{i=1}^{2m}$ in a cycle such that all such gluings are given by canonical isometries $V^j_i \simeq V^j_i$ or $V^2_i \simeq V^1_{i+1}$, with arithmetic performed modulo $2m$.

Furthermore, $D^{2m}(M_i) = D^{2m}(M_i,(V^1_i,V^2_i))$ possesses $m$ totally geodesic surfaces. Choosing one of them,  $F$, the manifold $D^{2m}(M_i) \dbs F$ is composed of two connected components, where each connected component is formed by gluing together $m$ many copies of $M_i$.
Forming a manifold that is composed of a cycle of pieces, where $m$-many copies of $M_i$ are connected, followed by $m$-many copies of $M_j$, for any $1 \leq i,j \leq 2m$, Theorem \ref{Cutting and pasting theorem} implies that it is hyperbolic with volume at least $\frac{1}{2} \prn{\vol^{2m}(M_i) + \vol^{2m}(M_j)}$.

Our strategy will use this logic to confer properties between different manifolds formed by cycles of $2m$ elements of $\cbr{M_i}$.
We use vector spaces to represent the possible configurations, which we now define.
\begin{construction}\label{Cyclic word construction}
    Let $\widetilde W$ be the set of cyclic words of length $2m$ from the alphabet $\cbr{T_1, T_1^R,\dots,T_{2m},T_{2m}^R}$ such that the words are expressed as a product
    \[
        t_1 \cdots t_{2m}
    \]
    where for all $i$, the product $t_it_{i+1}$ is a word from the following set
    \[
        \cbr{T_jT_{j+1}}_{j=1}^m \cup \cbr{T_j^R T_j}_{j=1}^m \cup \cbr{T_j T_j^R}_{j=1}^m \cup \cbr{T_{j+1}^R T_j^R}_{j=1}^m
    \]
    with arithmetic done modulo $2m$; 
    these correspond precisely with the manifolds that can be formed formally as cycles of $2m$ elements of $\cbr{M_i}$ along the canonical diffeomorphisms described in the preceding discussion.
    We accordingly call the letters $T_i^R$ \emph{reflected}. 
    
    There is a map from $\widetilde W$ into the cyclic words on alphabet $\cbr{T_i}$ made by replacing $T_i^R$ with $T_i$.
    This map is not injective in the case $2m = 2$;
    in this case, the elements $T_1T_2$ and $T_1^RT_2^R$ constitute the only pair of distinct cyclic words that map to the same word without reflections.
    
    We argue that the map is injective when $2m \geq 4$.
    Indeed, given the image of a valid word $w$ containing multiple letters, whether a letter is reflected is determined by its exponent and whether the following letter has one-larger or one-smaller index than the preceding letter.
    Since these are distinct indices mod $2m$, no two distinct words have the same image containing multiple letters.
    Noting that $(T_iT_i^R)^m = (T_i^RT_i)^m$ as cyclic words, this yields injectivity.
    We refer to the image of this map as $\Wknot$.
    
    For an arbitrary word $w = t_1\cdots t_{2m} \in \Wknot$, define $w^R := t_{2m}\cdots t_1$.
    Define the subset $\Wknoti i \subset \Wknot$ of \emph{palindromic words on $i$ letters} to be those words expressed as a product $t_1^{a_1} \cdots t_i^{a_i}t_i^{a_i} \cdots t_1^{a_1}$, where each $a_i \geq 0$.
    This defines a filtration on $\Wknot$:
    \[
        \Wknoti 1 \subset \cdots \subset \Wknoti m \subset \Wknot.
    \]
    Define the $\QQ$-vector space $\Mknoti i$ generated by $\Wknoti i$, and define $\Mknot$ to be generated by $\Wknot$.
\end{construction}

As in the above discussion, the set $\widetilde W$ corresponds naturally with a collection of 3-manifolds.
In fact, $T_1T_2$ and $T_1^RT_2^R$ correspond with the same 3-manifold up to homeomorphism by reflecting one to obtain the other.

Hence for all $w \in W$, we may unambiguously define the 3-manifold $\cS(w)$ to be the 3-manifold formed by gluing together a cycle of $2m$ many pieces from $\cbr{\prn{M_i,\bS}}$ such that the associated indices correspond with the indices of the letters in $w$.
Further, we may realize the free commutative monoid $\ZZ_{\geq 0}[\Wknot] \subset \Mknot$ via 3-manifolds by defining
\[
    \cS(w_1 + \cdots + w_k) := \cS(w_1) \sqcup \cdots \sqcup \cS(w_k).
\]

Using this, we may define volume. By taking hyperbolic volumes (with non-hyperbolic manifolds having volume 0), there is a map 
\[
    \vol:\Wknot \rightarrow \RR_{\geq 0}
    \hspace{50pt}
    \vol(w) := \vol(\cS(w)).
\]  
This extends linearly to a map $\vol:\Mknot \rightarrow \RR$.
Further, whenever $\cS(w_1)$ and $\cS(w_2)$ are hyperbolic, we have $\vol(w_1 + w_2) = \vol(\cS(w_1) \sqcup \cS(w_2)) = \vol(w_1) + \vol(w_2)$.

By the preceding discussion, we have $\cS(T_1T_2\cdots T_{2m}) = M$, so that
\begin{equation}\label{Volume of manifold}
    \vol(T_1T_2 \cdots T_{2m}) = \vol(M).
\end{equation}

Additionally, we have $\cS(T_i^{2m}) = D^{2m}(M_i,(V^1_i,V^2_i))$ for all $i$.
Hence the linear transformation $\vol:\Mknoti 1 \rightarrow \RR$ is simply the map 
\begin{equation}\label{Evaluation morphism}
    \vol : \Mknoti 1 \rightarrow \RR \hspace{50pt} \text{ defined on basis by } \hspace{50pt} T_i^{2m} \mapsto 2m \cdot \vol^{2m}(M_i) 
\end{equation}

We prove our theorem by using Theorem \ref{fixed surface lemma} and Theorem \ref{Cutting and pasting theorem} to relate the volumes of different elements of $\Mknot$.
These relations will be realized via the following quotient map
\[
    \pi: \Mknot \twoheadrightarrow V := \Mknot / \Span\cbr{2a_1a_2 - a_1a_1^R - a_2a_2^R \mid \operatorname{len}(a_1) = m}
\]
which equates a cyclic word with the average of the ``doubled'' halves made when splitting the word into sub-words of length $m$. 
We will see in the following proposition that these relations reduce cyclic words to sums of words made up of a single letter.
\begin{proposition}\label{Reducability proposition}
    The restriction $\pi|_{\Mknoti 1}:\Mknoti i \rightarrow V$ is an isomorphism. 
    In particular, we have
    \[
        2m \cdot \vol(w) \geq \sum_{i=1}^{2m} q_i(w) \cdot \vol(T_i^{2m}).    
    \]
    for all $w \in W$, where $q_i(w)$ denotes the number of appearances of the letter $T_i$ in $w$.
\end{proposition}

\begin{figure}[htpb]
    \centering
    \includegraphics[width=.8\textwidth]{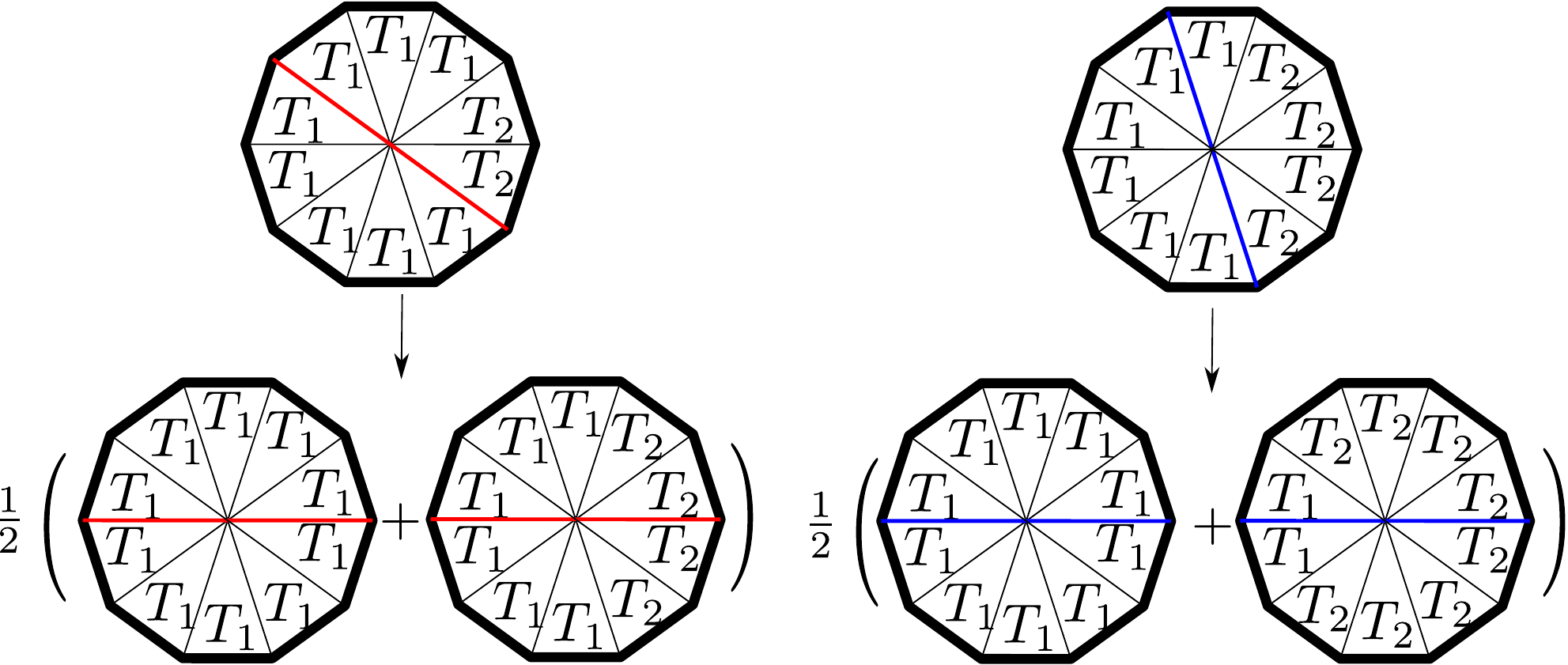}
    \caption{Some examples of relations laid out by $\pi$ in the case $m = 5$.the roles of $T_1$ and $T_2$ can be reversed.
    }
    \label{Pizza figure}
\end{figure}

First we work through an example.
\begin{example}\label{Algebra example}
    Fix $2m = 10$.
    Let $v_i \in \Wknoti 2$ be the monotonic cyclic word $T_1^i T_2^{m-i} T_2^{m-i} T_1^i$ and define $x_j := T_j^{2m}$.
    Then, by Figure \ref{Pizza figure} there is a cycle of words $(v_4,v_3,v_1,v_2)$ such that each element of this cycle is expressed as an $\QQ$-linear combination in $V$ of the next in the cycle and an element of $\Wknoti 1$.
    In particular, repeatedly applying the relations depicted, we obtain
    \[
        v_4 = \frac{1}{2}(x_1 + v_3) = \frac{1}{2}(x_1 + \frac{1}{2} (x_1 + v_1)) = \frac{1}{2}(x_1 + \frac{1}{2} (x_1 + \frac{1}{2}(x_2 + v_2))) = \frac{1}{2}(x_1 + \frac{1}{2} (x_1 + \frac{1}{2}(x_2 + \frac{1}{2}(x_2 + v_4))))  
    \] 
    This in turn is equivalent to the linear relation $v_4 = \frac{4}{5}x_1 + \frac{1}{5}x_2$, so $\pi(v_4)$ is generated by $\pi(\Wknoti 1)$.
    
    In particular, each of the moves outlined in Figure \ref{Pizza figure} are implemented by Theorem \ref{Cutting and pasting theorem}, so that we have 
    \[
        \vol(v_4) \geq \vol\prn{\frac{1}{2}(x_1 + \frac{1}{2} (x_1 + \frac{1}{2}(x_2 + \frac{1}{2}(x_2 + v_4))))  }
    \] 
    and hence $10\vol(v_4) \geq 8\vol(T_1^{10}) + 2\vol(T_2^{10})$, verifying Proposition \ref{Reducability proposition} in this case.
\end{example}

The following lemma establishes that the map $\pi$ corresponds with relations on volume. 
\begin{lemma}\label{Volume relation lemma separating}
    Suppose that $w = a_1a_2 \in \Wknot$ where $\operatorname{len}(a_1) = m$, and define $w_j := a_ja_j^R$.
    Then,
    \begin{equation}\label{Volume relation equation}
        2\vol(w) \geq \vol\prn{w_1 + w_2}.
    \end{equation}
    Further, we have $2q_i(w) = q_i(w_1) + q_i(w_2)$ for all $i$, where $q_i(w)$ denotes the number of appearances of the letter $T_i$ in $w$.
    \end{lemma}
\begin{proof}
    We prove Equation \eqref{Volume relation equation} purely topologically, and note that the second statement is clear.
    
    Since $w_j = a_ja_j^R$, there is a diffeomorphism of $\cS(w_j)$ which sends the piece corresponding with a letter in $a_j$ to the corresponding letter in $a_j^R$;
    this diffeomorphism fixes a surface $F_j \subset \cS(w_j)$ such that $\cS(w_j) \dbs F_j$ is the disjoint union of two copies of a manifold formed by gluing pieces of $\cbr{M_i}$ together corresponding with the word $a_j$.
    Furthermore, we may glue together one such manifold for each $w_1$ and $w_2$ to form $\cS(w)$;
    this outlines a totally geodesic surface $F_1 \sqcup F_2 \subset \cS(w_1 + w_2)$ such that cutting and gluing along $F_1 \sqcup F_2$ yields $\cS(2w)$.
    Hence equation \eqref{Volume relation equation} follows from Theorem \ref{Cutting and pasting theorem}.
\end{proof}

Proposition \ref{Reducability proposition} will be proved primarily via the following lemma.
\begin{lemma}\label{Pizza inductive lemma}
    Fix $i \geq 2$.
    For every $w \in \Wknoti i$, there is a relation of the form $2w = w' + w''$ where $w' \in \Wknoti i$ and $w'' \in \Wknoti{i-1}$.
\end{lemma}
\begin{proof}[Proof of Lemma \ref{Pizza inductive lemma}] Let $2n$ be the length of $w$.
    Write 
    \begin{equation}\label{The cycle notation}
        w = t_1^{a_1} \cdots t_i^{a_i}t_i^{a_i} \cdots t_1^{a_1} = t_1^{2a_1}t_2^{a_2} \cdots t_i^{a_i}t_i^{a_i} \cdots t_2^{a_2}.
    \end{equation}
    We may replace $w$ with $w^R$, so suppose without loss of generality that $a_1 \leq a_i$.
    
    Group the leftmost $n$ terms in the right-hand side of \eqref{The cycle notation} into a word $x$, and write $w = xy$.
    Note that, since $a_1 \leq a_i$, we have
    \[
        x = t_1^{2a_1} \cdots t_{i-1}^{a_{i-1}} t_{i}^{b}
    \]
    for some $b \leq a_i$.
    Hence $xx^R \in \Wknoti i$.
    Further, note that we have $2a_1 \leq n$, so that
    \[
        y = t_i^{2a_i - b} \cdots t_2^{a_2}
    \]
    implying that $yy^R \in \Wknoti{i-1}$.
    Setting $w' := xx^R$ and $w'' := yy^R$ yields the desired relation.
\end{proof}

\begin{proof}[Proof of Proposition \ref{Reducability proposition}]
    Note that there are no nontrivial relations imposed on $\Mknoti 1$;
    hence $\pi|_{\Mknoti 1}$ is injective.
    It suffices to prove that $\pi(\Wknoti 1)$ generates $V$.
    
    First we will prove inductively that $\pi(\Mknoti 1)$ contains $\pi(\Wknoti i)$ for each $i$, with base case $i = 1$ trivially verified.
    For the case $i \geq 2$ and element $w_1 \in \Wknoti i$, repeated application of Lemma \ref{Pizza inductive lemma} yields a sequence
    \[
        (w_1,w_2,w_3,\dots)
    \]
    where $w_j \in \Wknoti i$ for all $j$ and there exists a relation of the form $\pi(w_j) = \frac{\pi(w) + \pi(w_{j+1})}{2}$ where $w \in W_{i-1} \subset \pi(\Mknoti 1)$.
    Since $W_i$ is finite, there exist two indices $j < k$ such that $w_j = c \cdot w_k$ for some $c > 1$;
    then, there is some $u' \in \Mknoti{i-1}$ such that $\pi(w_j) = \frac{\pi(u')}{c - 1}$ by construction.
    The inductive hypothesis then guarantees that $\pi(w_j) \in \pi(\Mknoti 1)$, completing the inductive step.
    
    In particular, the above construction yields the inequality $\vol \pi(w_j) \geq \frac{\vol \pi(u')}{c-1}$, and induction yields that $\vol \pi(w_j)$ is at least the volume of an element of $\Mknoti 1$ with coordinates summing to 1, and hence $\vol \pi(w_j) \geq 2m \cdot \sum_i q_i(w_j) \vol(T_i^{2m})$, completing the inductive step.
    
    Using this, for an arbitrary word $w \in \Wknot$, we may pick an arbitrary relation $\pi(w) = \pi(w_1) + \pi(w_2)$, and note that $w_1,w_2 \in \Wknoti m$.
    By the above statement, $\pi(w) \in \pi(\Mknoti m) = \pi(\Mknoti 1)$ and $2\vol(w) \geq \vol(w_1) + \vol(w_2) \geq 2m \sum_i q_i(w) \vol(T_i^{2m})$, as desired.
\end{proof}

\begin{proof}[Proof of Theorem \ref{Main theorem separating}]
    For hyperbolicity and the volume bound, we simply combine Theorem \ref{Cutting and pasting theorem}, Proposition \ref{Reducability proposition} and Equation \eqref{Evaluation morphism}.
    Together these yield hyperbolicity and
    \begin{align*}
        2m \cdot \vol(M) 
        &= 2m \cdot \vol(T_1 \cdots T_{2m})\\
        &\geq \sum_{i=1}^{2m} \vol^{2m}(M_i).
    \end{align*}

    We now characterize equality of this bound.
    Suppose there is some $\Sigma \in \bS$ that is not totally geodesic.
    Then Lemma \ref{Volume relation lemma separating} and Theorem \ref{Cutting and pasting theorem} together imply that
    \begin{align*}
          2m \cdot \vol(M) 
          &= 2m \cdot \vol(w)\\
          &> m \cdot \prn{\vol(w_1) + \vol(w_2)}\\ &\geq \sum_i \vol^{2m}(M_i)
    \end{align*}
    where $\Sigma$ splits $w$ into $w_1w_2$.
    
    Suppose conversely that each $\Sigma \in \bS$ is totally geodesic.
    By repeatedly cutting along totally geodesic surfaces, then gluing along isometric totally geodesic surfaces, we may directly construct $\coprod_{i=1}^{2m} D^{2m}(M_i)$ from $\coprod_{i=1}^{2m} M$.
    By Theorem \ref{Cutting and pasting theorem}, this implies that
    \[
        \vol(M) = \sum_i \vol^{2m}(M^i)
    \]
    as desired.
\end{proof}

\subsection{Proof of the main theorem for a separating starburst in an \texorpdfstring{$\ell$}{l}-piece}\label{Separating inductive subsection}

The example from Section \ref{introduction section}, where a given saucer tangle can be decomposed into tetrahedral tangles as in Figure \ref{suacertet2} motivates the following definition.

\begin{definition}
    Let $P$ be a 3-manifold with distinguished surfaces $V^1,V^2,\dots,V^{2\ell} \subset \partial P$ such that $V^{i} \cap V^{j}$ is a 1-manifold (possibly with boundary, or possibly empty) for each distinct $1 \leq i,j \leq 2\ell$.
    We call such data an \emph{$\ell$-piece}, and we call a 3-manifold with no distinguished surfaces a \emph{0-piece}.
\end{definition}
For an $\ell$-piece, define $R = \overline{\partial P \setminus \cup V_i}$. Given an $\ell$-piece $(P,(V_i)_{i=1}^{2\ell}))$, some index $1 \leq k \leq \ell$, and some positive integer $m$, we may form the 3-manifold $D^{2m}(P,(V^{2k - 1},V^{2k}))$.
For any $j \neq 2k-1,2k$, the points $D^{2m}(V^j,(V^{2k - 1},V^{2k}))$ form a surface contained in $\partial D^{2m}(P,(V^{2k - 1},V^{2k}))$, and the intersections of these surfaces remain 1-manifolds.
Hence the data $\prn{D^{2m}(P,(V^{2k - 1},V^{2k})),\prn{D^{2m}(V^{j},(V^{2k - 1},V^{2k}))}_{j \neq 2k-1,2k}}$ constitute an $(\ell-1)$-piece.

That is, one can $2m$-replicate over one of the pairs of surfaces in an $\ell$-piece, which yields an $(\ell-1)$-piece.
Iteration of this construction yields the notion of an $\bm$-replicant, which we define now

\begin{definition}    
    We define the $\bm = (2m_1,\dots,2m_\ell)$-replicant of an $\ell$-piece $D^{\bm}(P,\prn{V^1,V^2,\dots,V^{2\ell}})$ inductively as 
    \[
        D^{(2m_1,\dots,2m_\ell)}(P,\prn{V^j}_{j=1}^{2\ell}) := D^{(2m_2,\dots,2m_\ell)}\prn{D^{2m}(P,(V^{2k - 1},V^{2k})),\prn{D^{2m}(V^{j},(V^{2k - 1},V^{2k}))}_{j=3}^{2\ell}}
    \]
    that is, the $(2m_1,\dots,2m_\ell)$-replicant is the $(2m_2,\dots,2m_{\ell})$-replicant of the $(\ell-1)$-piece resulting from doubling everything across $V^{1}$ and $V^{2}$.
    
    We say that $(P,\prn{V^j})$ is \emph{$\bm$-hyperbolic} if $D^{\bm}(P,\prn{V^j})$ is hyperbolic, in which case we define
    \[
        \vol^{\bm}(P,\prn{V^j}) := \frac{1}{\prod_{j=1}^\ell 2m_j}\vol(D^{\bm}(P,\prn{V^j}).
    \]
    
    Similar to before, if the tuple $\prn{V^j}$ is clear, we often simply refer to the piece as $P$, and refer to its $\bm$-replicant as $D^\bm(P)$ and its $\bm$-volume as $\vol^\bm(P)$.
\end{definition}
The above definition appears to involve a choice of permutation on $\ell$ letters defining the order over which the pairs are ``doubled over.''
However, the following result shows that every such choice yields the same manifold.

We say that an \emph{isomorphism of pieces} is a diffeomorphism of their underlying manifolds which restricts to a diffeomorphism of each surface.
\begin{proposition}\label{Replication commuting proposition}
    Let $(P,(V^j))$ be an $\ell$-piece.
    For any $1 \leq k \leq \ell$, define the notation
    \[
        D^{2m}_k(P,(V^j)) := \prn{D^{2m}(P,(V^{2k - 1},V^{2k})),\prn{D^{2m}(V^{j},(V^{2k - 1},V^{2k}))}_{j \neq 2k-1,2k}}.
    \]
    Then, for any distinct $1 \leq k,k' \leq \ell$ and any positive integers $m$ and $m'$, there is an isomorphism of pieces
    \begin{equation}\label{Double replication equation}
        D^{2m}_{\widetilde k} D^{2m'}_{k'} (P,(V^j)) \simeq D^{2m'}_{\widetilde k'} D^{2m}_{k} (P,(V^j))
    \end{equation}
    where $\widetilde k = k + 1$ if $k > k'$ and $k$ otherwise and $\widetilde k' = k' + 1$ if $ k' > k$ and $k'$ otherwise.
    In particular, for any $\bm = (2m_1,\dots,2m_\ell)$ and $1 \leq k \leq 2\ell$, we have
    \begin{equation}\label{Replication commuting equation}
    D^{\bm}\prn{P,(V^j)_{j=1}^{2\ell}} = D^{(2m_1,\dots,\widehat{2m_k},\dots,2m_\ell)}\prn{D^{2m_k}(P,(V^{2k - 1},V^{2k})),\prn{D^{2m_k}(V^{j},(V^{2k - 1},V^{2k}))}_{j \neq 2k-1,2k}}
    \end{equation}
    where $\widehat .$ denotes exclusion from a list.
\end{proposition}
\begin{proof}
    Note that both sides of \eqref{Double replication equation} are expressed as a quotient of the disjoint union of $4mm'$ copies of $P$.
    For the left hand side, we first quotient together subspaces of $2m'$ copies of $P$ along $V^{2k' - 1}$ and $V^{2k'}$, then we quotient together such subspaces along the images of $V^{2k-1}$ and $V^{2k}$.
    The right hand side is analogous, with $(m',k')$ and $(m,k)$ interchanged.
    
    This may be expressed in both cases as a quotient of the disjoint union of $4mm'$ copies of $P$, arranged in a ``grid,'' where a point is either identified with nothing, identified with a point on the east or west, or identified with a point on the north or south (with combinations of the last two allowed).
    These relations are the same in both cases, so the resulting manifolds are diffeomorphic;
    furthermore, this diffeomorphism clearly carries the piece surfaces to each other, so the resulting pieces are isomorphic.
    
    We prove Equation \eqref{Replication commuting equation} via a generalization. Each permutation $\pi \in S_\ell$ defines a manifold $D^{\bm}_\pi(P,(V^j))$, where one replicates over the $\pi(1)$ surface first, the $\pi(2)$ second etc. 
    Equation \eqref{Double replication equation} yields a diffeomorphism $D^{\bm}_\pi(P,(V^j)) \simeq D^{\bm}_{\tau \pi}(P,(V^j))$ for any simple transposition $\tau$.
    Since simple transpositions generate $S_\ell$, this implies that $D^{\bm}_\pi(P,(V^j)) \simeq D^{\bm}_{\rho}(P,(V^j))$, for any $\pi,\rho \in S_\ell$, which generalizes Equation \eqref{Replication commuting equation}.
\end{proof}

By similar logic to the $\ell = 1$ case, an $m$-sheeted separating starburst in an $(\ell-1)$-piece $\prn{M,\prn{V^j}_{j=1}^{2(\ell-1)}}$ yields a cycle of $2m$ many $\ell$-pieces $\prn{M_i,\prn{V_i^j}_{j=1}^{2\ell}}$. We define $H$ to be the closure of $\partial M \setminus \cup_{j = 1}^{2(\ell-1)} V^j$. We denote the closure of the complement of $V^1_i \cup V^2_i$ in $\partial M_i$ by $H_i$. Note that $H_i \subset \partial M$.


We generalize Theorem \ref{Main theorem separating} directly to $\ell$-pieces in the following theorem.
\begin{theorem} \label{Main theorem separating inductive}
    Let $(M,\bS)$ be an $m$-sheeted separating starburst with all boundary components tori. Suppose that $\prn{M,\prn{V^j}_{j=1}^{2(\ell-1)}}$ is an $(\ell-1)$-piece, and suppose the associated $\ell$-pieces $(M_i,\prn{V_i^j}_{j=1}^{2\ell})$ are each $(2m_1,\dots,2m_\ell)$-hyperbolic and  $V^{2j -1}_i \cap V^{2j}_i$ does not intersect $H_i$ for all $j$.
    Then, $(M,\prn{V_i}_{i=1}^{2(\ell-1)})$ is $(2m_1,\dots,2m_{\ell-1})$-hyperbolic with volume satisfying
    \[   
        \vol^{(2m_1,\dots,2m_{\ell-1})}\prn{M} \geq \sum_{i=1}^{2m_\ell} \vol^{(2m_1,\dots,2m_\ell)}\prn{M_i}.
    \]
    This equality is attained if and only if each $\Sigma \in \bS$ is totally geodesic in $M$.
\end{theorem}
\begin{proof}
For brevity, we will write $D^{m_1}(V^j)$ to replace $D^{m_1}(V^j,(V^1,V^2))$.

    We prove this theorem inductively. In the case $\ell = 1$, 
    the result is given by Theorem \ref{Main theorem separating}. 
    
    Suppose inductively that Theorem \ref{Main theorem separating inductive} is satisfied for starbursts in a $k$-piece for $k \leq \ell - 2$.
    The definition
    \[
        D^{(2m_1,\dots,2m_\ell)}(P,\prn{V^1,\dots,V^{2\ell}}) = D^{(2m_2,\dots,2m_{\ell})}\prn{D^{2m_1}(P,\prn{V^{1},V^{2}}),\prn{D^{m_1}(V_3),\dots,D^{m_1}(V^{2\ell})}}
    \]
    and Proposition \ref{Replication commuting proposition} together imply that  the $(2m_1,\dots,2m_\ell)$-hyperbolicity of each $\ell$-piece is equivalent to $(2m_2,\dots,2m_\ell)$-hyperbolicity of pieces corresponding with a separating starburst $D^{2m_1}(\bS)$ in the $(\ell-2)$-piece $\prn{D^{2m_1}(M),(D^{m_1}(V_3),\dots,D^{m_1}(V^{2(\ell-1)}))}$.
    Our inductive hypothesis then implies that $\prn{D^{2m_1}(M),(D^{m_1}(V_3),\dots,D^{m_1}(V^{2\ell})}$ is $(2m_2,\dots,2m_\ell)$-hyperbolic, or equivalently, that $(M,(V^1,\dots,V^{2\ell}))$ is $\bm$-hyperbolic.
    
    A similar argument yields volume bounds:
    \begin{align*}
        2m_1 \cdot \vol^{(2m_1,\dots,2m_{\ell-1})}(M,\prn{V^j}_{j=1}^{2(\ell-1)}) &= \vol^{(2m_2,\dots,2m_{\ell-1})}\prn{D^{2m_1}(M,\prn{V^{1},V^{2}}),\prn{V^j}_{j=3}^{2(\ell-1)}}\\
        &\geq \sum_i \vol^{(2m_2,\dots,2m_{\ell})}\prn{D^{2m_1}(M_i,\prn{V_i^{1},V_i^{2}}),\prn{D^{m_1}(V_i^j)}_{j=3}^{2\ell}}\\
        &= 2m_1 \cdot \sum_i  \vol^{(2m_1,\dots,2m_\ell)}(M_i,\prn{V_i^j}_{j=1}^{2\ell})).
    \end{align*}
    
    In fact, this equality is attained if and only if equality is attained in the middle inequality above.
    By the inductive hypothesis, this is true if and only if each doubled surface in $\bS$ is totally geodesic in $D^{2m_1}(M,(V^1,V^2))$.
    This is true if and only if each surface in $\bS$ is totally geodesic in $M$, as desired.
\end{proof}




\section{The main theorem in the general case}\label{Main theorem nonseparating section}
\subsection{Definitions}
In Section \ref{Main theorem separating section}, we defined suitable collections of surfaces, called \emph{$m$-sheeted separating starbursts}, which separate a 3-manifold $M$ into a disjoint union of $2m$ pieces such that hyperbolicity of the pieces implies hyperbolicity of $M$.
In fact, we may summarize this with one datum rather than $2m$:
defining the disjoint union of $\ell$-pieces by
\[
    (P,(V^j)) \sqcup (P',(V'^j)) := (P \sqcup P', (V^j \sqcup V'^j)).
\]
For any $\bm \in 2\ZZ^{\ell}_{>0}$, there is a diffeomorphism
\[
    D^{\bm}\prn{(P,(V^j)) \sqcup (P',(V'_j))} \simeq D^{\bm}\prn{P,(V^j)} \sqcup D^{\bm}\prn{P',(V^j)}.
\]
In particular, since we have $(M \dbs \bS,(V^1,V^2)) = \coprod_{i=1}^{2m} (M_i,(V_i^1,V_i^2))$ as in Construction \ref{Piece construction}, we may summarize Theorem \ref{Main theorem separating} by stating that $2m$-hyperbolicity of $(M \dbs \bS,(V^1,V^2))$ implies hyperbolicity of $M$, with volume satisfying
\[
        \vol(M) \geq \vol^{2m}(M \dbs \bS).
\]
In the case of a starburst in an $(\ell-1)$-piece, we further have that $(2m_1,\dots,2m_\ell)$-hyperbolicity of the $\ell$-piece $\prn{M \dbs \bS, (V^j)_{i=1}^{2\ell}}$ implies $(2m_1,\dots,2m_{\ell-1})$-hyperbolicity of $\prn{M,(V^j)_{i1}^{2(\ell-1)}}$, with volumes satisfying
\[
    \vol^{(2m_1,\dots,2m_{\ell-1})}\prn{M} \geq \vol^{(2m_1,\dots,2m_\ell)}\prn{M \dbs \bS}.
\]

This formulation is suggestive;
the contruction of $(M \dbs \bS, (V^j))$ doesn't appear to depend on the fact that $M \dbs \bS$ separates into $M_i$ in an essential way.
To expand on this, we make the following definition, which relaxes orientability of $M$ and the separation of $M$ into a cycle of pieces.
\begin{definition}\label{Starburst definition}
    Let $M$ be a 3-manifold and let $\bS = \cbr{\Sigma_1,\dots,\Sigma_m}$ be a set of properly embedded two-sided surfaces in $M$.
    Let $E := \bigcup_{i\neq j} \Sigma_i \cap \Sigma_j$ be the locus at which at least two surfaces intersect, and suppose that $E = \Sigma_i \cap \Sigma_j$ for each $i \neq j$.
    Let $\rho:\bigcup_i \Sigma_i \dbs E \rightarrow \cbr{+,-}$ be a ``two-coloration'' of the connected components of $\Sigma_i \dbs E$ for each $i$.
    We say that the data $(M,\bS,\rho)$ constitute an \emph{n-sheeted starburst in $M$} (often, simply a \emph{starburst}) if the following conditions are met:
    \begin{enumerate}
        \item the intersection locus $E$ is a compact one-manifold (possibly disconnected and possibly with boundary) properly embedded in $M$;
        \item \label{Diffeomorphism condition} there is a neighborhood $E \subset U \subset M$ and a diffeomorphism of pairs $(U,U \cap\prn{\bigcup_{i=1}^m \Sigma_i}) \simeq (D \times E, \bigcup_{i=1}^m L_i \times E)$ which restricts to a diffeomorphism $\Sigma_i \cap U \simeq L_i \times E$ for each $i$, with notation defined in Figure \ref{newstarburstfigure} and \ref{newstarburstfigure 2}
        ;
        \item \label{Coloration condition} any two connected components of $\Sigma_i \dbs E$ who neighbor the same points of $E$ in $M$ have opposing colorations via $\rho$;
        \item \label{Compatible direction condition} for each connected component of $E$, given the induced orientations of normal bundles and labels of the lines $L_i \subset S_m$, all orientations of normal bundles on a $+$-labelled component of $L_i \setminus E$ share a direction (either clockwise or counterclockwise).
    \end{enumerate}
\end{definition}
\begin{remark}
    We may formally define "all $+$-labelled components of $L_i \setminus E$ share direction" several equivalent ways;
    one convenient one is to note that the subspace of points $d(\mathbf{L})$ of $\partial(D \dbs \bigcup_i L_i)$ corresponding with $\bigcup_i L_i$ have $2n$ connected components, and each surface together with an orientation of its normal bundle determines two connected components of $d(\mathbf{L})$ having either positive orientation and $+$ label or negative orientation and $-$ label, as in Construction \ref{Piece construction}.
    One may replace condition \eqref{Compatible direction condition} with the requirement that this correspondence identifies all $2n$ components of $d(\mathbf{L})$, or equivalently, than no $+$ and $-$ labelled components of half-lines identify the same component.
\end{remark}

\begin{figure}
    \centering
    \includegraphics[scale=0.6]{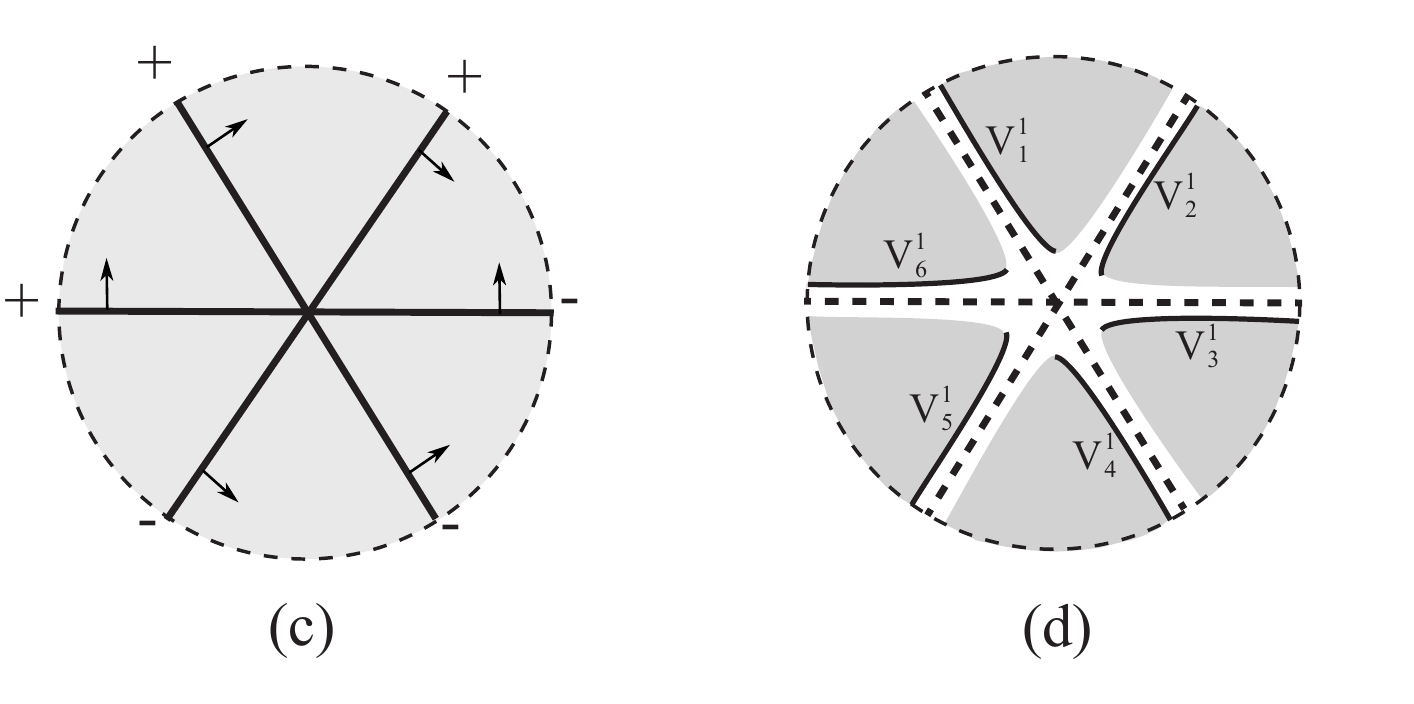}
    \caption{Continuation of Figure \ref{newstarburstfigure}, now using the orientation and labelling data.
    Part (c) includes arrows denoting the orientation, and labels denoting $\rho$.
    Part (d) shows the parts of surfaces near $E$ comprising $V^1_i$.
    }
    \label{newstarburstfigure 2}
\end{figure}

For this paragraph, let $(M,\bS)$ be a separating starburst. 
Then, by definition, $(M,\bS)$ satisfies the first two conditions for a starburst.
The portion $d(\Sigma_i) \setminus d(E) \subset \partial M$ corresponding with $\Sigma_i$ can be expressed as the union of $V^1_i$, $V^2_{i-1}$, $V^1_{i+n}$, and $V^2_{i + n - 1}$, with the first pair and last pair of surfaces corresponding with a common component of $\Sigma_i \dbs E$.
The parts $V^2_i$ all point \emph{clockwise} local to $E$, so we may choose the labeling $\rho|_{\Sigma_i}:\Sigma_i \dbs E \rightarrow \cbr{+,-}$ so that all of $V^2_i$ corresponds with either positive orientation and $+$-labelling or negative orientation and $-$-labelling;
we may further choose labelling so that $V^1_i$ corresponds with either positive orientation and $-$-labeling or negative orientation and $+$-labelling.
Conditions \eqref{Coloration condition} and \eqref{Compatible direction condition} then follow from the fact that $\bS$ separates $M$ into $M_i$.
Hence separating starbursts are themselves starbursts, and we may reverse-engineer the notion of \emph{clockwise} or \emph{clockwise side} from orientations and labelling.
We will use this to extend the piece structure from separating starbursts to all starbursts.


Henceforth let $(M,\bS)$ be an $m$-sheeted starburst.
We now explicitly construct surfaces rendering $M \dbs \bS$ to be a piece, generalizing the piece structure on each $M_i$ in the separating case.
In the same sense that the outer boundary of the saucer containing a tangle is separated by copies of $E$ into a \emph{clockwise} and \emph{counterclockwise side}, we define the piece structure by using the 2-coloration and orientation of the normal bundles of $\Sigma_i$ to choose surfaces which, local to the points corresponding with $E$, appear to distinguish the \emph{clockwise} and \emph{counterclockwise side} as well.

\begin{construction}\label{Piece construction}
For $\Sigma_i \in \bS$, let $d(\Sigma_i)$ be the union of the points in $\partial (M \dbs \bS)$ which correspond with $\Sigma_i$, and let $d(\bS)$ be the union of the $d(\Sigma_i)$. Let $d(E)$ be defined similarly for $E$. 

Every point of $d(\Sigma_i) \setminus d(E)$ corresponds with a point in $\Sigma_i \setminus E$ and for a point in $\Sigma_i \setminus E$, there are two points in $d(\Sigma_i) \setminus d(E)$ which correspond to it.
One comes about as a boundary point on the positively oriented side of a normal bundle on $\Sigma_i$ and one comes about from the negatively oriented side, so we may accordingly call a point in $d(\Sigma_i) \setminus d(E)$ positive or negative. 
By nature of the orientation, if two points in $d(\Sigma_i) \setminus d(E)$ reside in the same connected component, they must either both be positive or both be negative. 
Hence we may refer to a positively-oriented component of $d(\Sigma_i) \setminus d(E)$ as one whose points are positive away from $d(E)$, and a negatively-oriented component otherwise.

Similarly, every component of $d(\Sigma) \setminus d(E)$ has points corresponding with a unique component of $\Sigma_i \dbs E$, so there is a two-coloration $\bigcup_i d(\Sigma_i) \setminus d(E) \rightarrow \cbr{+,-}$ induced by $\rho$, and we will also call this $\rho$.
We say that a component of $d(\bS)\setminus d(E)$ (that is, a component of some $d(\Sigma_i) \setminus d(E)$) as \emph{counter-clockwise-pointing} if it is either $+$-labeled and positively-oriented or $-$-labelled and negatively-oriented, and \emph{clockwise-pointing} otherwise.
Let $V_i^2$ be the union of $E$ and the clockwise-pointing components of $d(\Sigma_i) \setminus d(E)$, and let $V_i^1$ be the union of $E$ and the counterclockwise-pointing components of $d(\Sigma_i) \setminus d(E)$.
Let $V^j := \bigcup_i V_i^j$ be the union of the clockwise-pointing or counterclockwise-pointing components.
We visualize this via Figure \ref{newstarburstfigure 2}.

By condition \eqref{Compatible direction condition}, each component of $d(E)$ separates a clockwise-pointing component and a counterclockwise-pointing component;
hence $V^1 \cap V^2 = d(E) = \bigsqcup_{i=1}^{2m} E$, which is a 1-manifold.
\end{construction}

\begin{remark}
    In the case that every connected component of $\Sigma_i$ intersects $E$, note that every allowable choice of orientations of normal bundles and labellings identifies the same piece structure on $M \dbs \bS$;
    in general, varying orientations and labellings allows one to construct up to $2^C$ many isomorphism classes of piece structures, where $C$ is the number of connected components of $\bigcup_i \Sigma_i$ not intersecting $E$.
\end{remark}

This construction yields a piece structure $(M \dbs \bS, (V^1, V^2))$.
Hence there is a notion of replication and $2m$-hyperbolicity of a starburst, and we may ask whether generalizations of Theorem \ref{Main theorem separating} applies in this case. But we would like to allow the intersection set $E$ of the surfaces to intersect the boundary of the manifold. This is the topic of the next subsection.

\subsection{Boundary compatibility and the torus boundary case}

In Theorem \ref{Main theorem separating}, we assumed that the manifold $M$ had only torus boundary and $E$ did not intersect the boundary. We shortly generalize these conditions.  In the case that $E$ intersects a torus boundary of $M$, let $G$ be the graph on $T$ formed by the intersection of $\bigcup_i \partial \Sigma_i$ with $T$. We prove the following result.

\begin{lemma}\label{Eintersectstorusboundary}
    If $(M, \bS)$ is an $m$-sheeted separating starburst, $E$ intersects a torus boundary $T$ of $M$, and the $2m$-replicants of the corresponding pieces are all hyperbolic, then $G$ is connected, $m$ = 2, and all of the complementary regions to $G$ on $T$ are squares. 
\end{lemma}

\begin{proof}
Note that every surface $\Sigma_i$ must touch every vertex of the graph.  Also, hyperbolicity of the replicants means that the replicant cannot have spherical boundaries. This would, for instance occur, if any complementary region of $G$ on $T$ is a bigon. 

We first consider the simple closed curves that make up $\Sigma_i  \cap T$ for a fixed value of $i$. None can be trivial curves on $T$ as follows. If there were trivial curves, take an innermost such. Let $D$ be the disk it bounds on $T$. So $D$ has no other curves from $\Sigma _i$ inside it. Therefore it has no vertices in its interior, since every vertex must touch  a curve from $\Sigma_i$. If $D$ has vertices on its boundary, there are curves from each of the other surfaces that cross it. But that there are no vertices inside $D$ implies that we can take an outermost arc from the other surfaces inside $D$ and it forms a bigon with an arc on the boundary of $D$, which will generate a sphere in the boundary of the corresponding replicant, a contradiction to its hyperbolicity. 

On the other hand, suppose there are no vertices on $\partial D$. Then when we form $D^{2m} (M_i)$ for the  piece containing $D$ on its boundary, we again obtain a sphere boundary, contradicting hyperbolicity of the replicants. 

Hence,  for all $i$, every curve in $\Sigma _i \cap T$ must be nontrivial on $T$. So for each fixed $i$, the curves in $\Sigma_i \cap T$ must all be parallel $(p_i, q_i)$-curves. Further, the only options for complementary regions to $G$ are disks and annuli. 

If there is an annulus in $M_i \cap T$, take its boundary to be meridians for convenience. Then for all  $i$ and all of the curves in $\Sigma_i \cap T$, $q_i = 0$. That is to say, for all $i$, all simple closed curves in $\Sigma_i\cap T$ are meridians. 

Consider the first two surfaces $\Sigma_1$ and $\Sigma_2$.  The fact that all curves in their intersections with $T$ are meridians implies that any two that cross must do so an even number of times.  Further, for some pair of curves $\gamma_1$ in $\partial \Sigma_1$ and $\gamma_2$ in $\partial \Sigma_2$, there must be an  arc $\alpha_1$ on $\gamma_1$ and an arc $\alpha_2$ on $\gamma_2$ such that they form a bigon with sides that are not crossed by any other curve from $\partial \Sigma_1$ or $\partial \Sigma_2$. Additional arcs from other surfaces could enter the bigon through one of the vertices. But since there can be no vertices inside the bigon as there are no curves from $\partial \Sigma_1$ or $\partial \Sigma_2$ inside it, any such arc must just exit out the opposite vertex. Thus, they decompose the bigon into thinner bigons, which again lift to spheres in the corresponding replicants, contradicting their hyperbolicity.  
So there are no annuli.

We now show that $(p_i, q_i) \neq (p_j, q_j)$ for $i \neq j$. The proof is the same argument as above. Namely if curves from $\partial \Sigma_1$ and $\partial \Sigma_2$ on $T$ are parallel, they must cross, and then we can find an innermost bigon between the two types of curves and there can be no other boundary curves of $\Sigma_1$ or $\Sigma_2$ inside the bigon formed. So only curves from one vertex to the other can appear in the bigon, which makes a thinner bigon, a contradiction. 

So now if we consider a $(p_1, q_1)$-curve $\gamma_1$ and a $(p_2, q_2)$-curve $\gamma_2$ from two different surfaces,  they must cross. But in fact they must cross a minimal amount of times, namely $|p_1 q_2 - p_2 q_1|$ times, as if not, there will be a bigon formed between them with no other curves from these two surfaces in the bigon and the same argument shows  there is a bigon. 

So we have collections of parallel curves for each surface. We can always choose one set of them to be thought of as meridians. Then if there are only two surfaces, they decompose the  torus into squares and $m = 2$. Then all replicants have torus boundaries corresponding to $T$.

If there are three surfaces, the first two surface boundaries decompose the torus into squares. The third surface boundaries must intersect the first two only at their intersections. Since they cannot generate bigons, they must cut diagonally across the squares. 

Note that if there were a fourth surface, it would also have to have its boundary curves on $T$ intersecting only at the existing vertices and this would force bigons. Thus, $m = 2$ or 3.

In the case $m =3$, the torus is decomposed into triangles by the graph $G$. However, there can be no odd cycles in $G$. This is because the cyclic order of pieces around a vertex must be preserved. Hence, for any edge in $G$, the cyclic orders must be clockwise around one and counterclockwise around the other. In particular, this means $G$ must be a bipartite graph, with all even-length cycles.   
\end{proof}

In particular, this lemma ensures that if $E$ intersects a torus boundary of $M$, then if the replicants are hyperbolic, as we assume in Theorem \ref{Main theorem separating},  the boundaries of the replicants of the corresponding pieces are also tori. 

We would like to generalize Theorems \ref{Main theorem separating} and \ref{Main theorem separating inductive} to manifolds with higher genus boundary. Such a manifold is said to be tg-hyperbolic if when all torus boundaries are removed, it has a hyperbolic metric such that its higher genus boundaries are totally geodesic. This is equivalent to the fact that if the manifold is doubled across its higher genus boundaries, and the torus boundaries are again removed, the resulting manifold is hyperbolic.  A tg-hyperbolic manifold has a well-defined finite hyperbolic volume. 

But in order to generalize to this situation, in the case that $E$ intersects the higher genus boundary components, we need to know that the $2m$-replicants of the corresponding pieces that intersect those boundaries also have corresponding genus greater than 1. Hence we have the following definition.

\begin{definition}Any connected sub-surface $Q_i$ of the boundary of $M_i$ that lies in the boundary of $M$ will correspond to boundary-components of $D^{2m}(M_i,(V^1_i, V^2_i))$. We define a piece $(M_i,(V^1_i, V^2_i))$ to be {\bf boundary-compatible} with the manifold $M$ from which it comes if for any connected components of $\partial M_i \subset \partial M$, the corresponding boundary components of $D^{2m}(M_i,(V^1_i, V^2_i))$ have genus one if and only if the corresponding boundary components of $M$ have genus one. Note that all of the boundary-components of $D^{2m}(M_i,(V^1_i, V^2_i))$ coming from a particular connected component $Q_i$ of $\partial M_i$ will have the same genus. 
\end{definition}

\subsection{Statement of the General Theorems}

We are now ready to state the generalization of Theorem \ref{Main theorem separating}.
We prove the following.
\begin{theorem}\label{Main theorem}
    Suppose $(M,\bS,\rho)$ is an $m$-sheeted starburst in a 3-manifold $M$, and $(M \dbs \bS,(V^1,V^2))$ is the associated piece.
    If $(M \dbs \bS, (V^1,V^2))$ are $2m$-hyperbolic and boundary-compatible with $M$, then $M$ is hyperbolic, with volume satisfying
    \[
        \vol(M) \geq \vol^{2m}(M \dbs \bS).
    \]  
    Equality is attained if and only if each surface $\Sigma_i \subset M$ is totally geodesic.
\end{theorem}

Note that Lemma \ref{Eintersectstorusboundary} implies that if $m > 2$, it must be the case that $E$ intersects only higher genus components of $\partial M$ to obtain hyperbolicity of the pieces.

Similarly to the separating case, a starburst in an $(\ell-1)$-piece $(M,(V^j)_{i=1}^{2(\ell-1)})$ yields an $\ell$-piece $(M \dbs \bS, (V^j)_{i=1}^{2\ell})$, and there is an associated notion of $\bm$-hyperbolicity.
We will prove the following generalization of Theorem \ref{Main theorem separating inductive} in this case.
\begin{theorem}\label{Main theorem inductive}
    Suppose that $\prn{M,\prn{V^j}_{j=1}^{2(\ell-1)}}$ is an $(\ell-1)$-piece, suppose $(M,\bS,\rho)$ is an $m$-sheeted, boundary-compatible starburst, and suppose the associated $\ell$-piece $(M \dbs \bS,\prn{V^j}_{j=1}^{2\ell})$ is $(2m_1,\dots,2m_\ell)$-hyperbolic.
    Then, $(M,\prn{V^j}_{j=1}^{2(\ell-1)})$ is $(2m_1,\dots,2m_{\ell-1})$-hyperbolic with volume satisfying
    \[   
        \vol^{(2m_1,\dots,2m_{\ell-1})}\prn{M} \geq \vol^{(2m_1,\dots,2m_\ell)}\prn{M \dbs \bS}.
    \]
    This equality is attained if and only if each $\Sigma \in \bS$ is totally geodesic in $M$.
\end{theorem}

In Subsection \ref{Non-orientable subsection}, we reduce Theorem \ref{Main theorem} to the case that $M$ is orientable.
Subsection \ref{Non-orientable subsection} can be skipped by the reader only interested in working with orientable 3-manifolds.

Then, in Subsection \ref{secondproofsubsection} we give an alternative proof of Theorem \ref{Main theorem separating} which directly uses the piece structure on $M \dbs \bS$, rather than individual $M_i$.
We use this as motivation to prove Theorems \ref{Main theorem} and \ref{Main theorem inductive} in Subsection \ref{Nonseparating subsection}.

\subsection{Reduction of Theorem \ref{Main theorem} to the orientable case}\label{Non-orientable subsection}
In general, we denote the orientable double cover of $M$ as $M^*$, with map $\pi_M:M^* \rightarrow M$. In the case $M$ is already orientable, $M^*$ represents the double cover obtained by taking two copies of 
$M$.

\begin{proposition}\label{Orientable double cover proposition}
    Suppose $(M,\bS)$ is an $m$-sheeted starburst.
    \begin{enumerate}
        \item Let $\Sigma_i^* := \pi_{M}^{-1}(\Sigma_i)$, let $E^* := \pi_M^{-1}(E)$, and let $\bS^* := \cbr{\Sigma^*_i}$.
        Define labels $\rho^*:\bigcup_i \Sigma^*_i \dbs E^* \rightarrow \cbr{+,-}$ by the composition
        \[
            \bigcup_i \Sigma^*_i \dbs E^* \xrightarrow \pi \bigcup_i \Sigma_i \dbs E \xrightarrow{\rho} \cbr{+,-}.
        \]
        Then, $(M^*,\bS^*,\cbr{\rho^*})$ constitute an $m$-sheeted starburst.
        \item Cutting and $2m$-replication commute with taking the orientable double cover;
        in particular, letting $(M^* \dbs \bS^*,(V^{1*},V^{2*}))$ be the piece on the orientable double cover, we have
        \[
            D^{2m}(M \dbs \bS,(V^1,V^2))^* = D^{2m}(M^* \dbs \bS^*,(V^{1*},V^{2*})).
        \]
    \end{enumerate}
\end{proposition}

For part (1), our argument will extend naturally to say that, for an $m$-sheeted starburst $(M,\bS,\cbr{\rho})$ and a finite covering map $\pi:M' \rightarrow M$, the set of preimages $\pi^{-1}(\bS) := \cbr{\pi^{-1}(\Sigma_i) \mid \Sigma_i \in \bS}$ endowed with labels given by composition 
\[
    \begin{tikzcd}
    \rho \pi : \bigcup_i \pi^{-1}(\Sigma_i) \dbs \pi^{-1}(E) \arrow[r,"\pi"]
    & \bigcup_i \Sigma_i \dbs E \arrow[r,"\rho"] 
    & \cbr{+,-}
    \end{tikzcd}
\]
constitute an $m$-sheeted starburst in $M'$.
Verifying this will consist mostly of noting that each condition in Definition \ref{Starburst definition} can be verified locally, sometimes verifying compatibility of local data between certain pairs of points. 

\begin{proof}[Proof of Proposition \ref{Orientable double cover proposition}]
    \textit{Part (1).}
    Each $\Sigma_i^*$ is a properly embedded surface in $M^*$.
    A trivialization of the normal bundle of $\Sigma_i$ lifts to a trivialization of the normal bundle of $\Sigma_i^*$, so each $\Sigma_i^*$ is a two-sided surface.
    
    The intersection locus $E^* = \bigcup_{i \neq j} \Sigma_i^* \cap \Sigma_j^*$ satisfies $E^* = \pi^{-1}(\Sigma_i \cap \Sigma_j) = \Sigma_i^* \cap \Sigma_j^*$ for any $i \neq j$.
    Furthermore, $E^*$ is the preimage of a 1-manifold $E$ under a covering map, so it is a 1-manifold.
    
    The existence of a diffeomorphism $(U,U \cap (\bigcup_i \Sigma_i)) \simeq (D \times E, S_m \times E)$ respecting the cyclic order is equivalent to existence, for each point $p \in E$, of a neighborhood $U_p$ of $p$ and a diffeomorphism $(U_p,U_p \cap (\bigcup_i \Sigma_i)) \simeq (D \times I, S_m \times I)$ respecting the cyclic order of $\bS$ and $S_m$ for each $p$.
    For each $p^* \in E^*$, composing with a local diffeomorphism $U_{p^*} \xrightarrow{\sim} \pi U_{p^*}$ of a neighborhood $U_{p^*}$ of $p^*$ yields such a diffeomorphism.
    
    Similarly, compatibility of the direction induced by orientations of the normal bundles of $\Sigma_i$ and by $\rho$ can be verified local to each point of $E$;
    since the orientation and labeling are compatible with a local diffeomorphism $\pi$, this compatibility on $(M^*,\bS^*,\cbr{\rho^*})$ follows from compatibility on $(M,\bS,\cbr{\rho^*})$.
    Hence we have defined an $m$-sheeted starburst in $M^*$.
    
    \textit{Part (2).}
    We may characterize the cover $(M^* \dbs \bS^*) \rightarrow M \dbs \bS$ on an open set $(M^* \setminus \bS^*) \rightarrow M \setminus \bS$.
    The latter is the restriction of the orientable double cover of $M$ onto an open set, and hence an immersed submanifold of codimension 0.
    Hence it is the orientable double cover, and we have $(M^* \dbs \bS^*) \simeq (M \dbs \bS)^*$.
    
    The orientable double cover is compatible with disjoint unions, so the induced cover on $\coprod_{i=1}^{2m} (M \dbs \bS)$ is the orientable double cover.
    The covering translations permute fibers, and the surfaces $V^{j*}$ are given by the preimages of $V^{j*}$, with the identifications defining $D^{2m}(M^* \dbs \bS^*,(V^{1*},V^{2*}))$ defined fiberwise;
    hence the orientable double cover on $\coprod_{i=1}^{2m}(M \dbs \bS)$ descends to the orientable double cover on $D^{2m}(M \dbs \bS, (V^{1*},V^{2*}))$, i.e. the cover $D^{2m}(M^* \dbs \bS^*,(V^{1*},V^{2*})) \rightarrow D^{2m}(M \dbs \bS, (V^1,V^2))$ is isomorphic to the orientable double cover, and in particular we have an diffeomorphism of their total spaces, as we wanted to show.
    \end{proof}
By this proposition, we are able to reduce to the orientable case by passing to the orientable double cover.  Hyperbolicity of $D^{2m}(M \dbs \bS,(V^1,V^2))$ implies hyperbolicity of $D^{2m}(M^* \dbs \bS^*,(V^{1*},V^{2*}))$, which implies hyperbolicity of $M^*$, which implies hyperbolicity of $M$.
Additionally, we have
\[
    \vol(M) = \frac{1}{2}\vol(M^*) \geq \frac{1}{2}\vol^{2m}(M^* \dbs \bS^*,(V^{1*},V^{2*})) = \vol^{2m}(M \dbs \bS, (V^1,V^2)).
\]
with equality attained if and only if each element of $\bS^*$ is totally geodesic in $M^*$, which occurs if and only if each element of $\bS$ is totally geodesic in $M$.

\subsection{A second proof of Theorem \ref{Main theorem separating} (the main theorem for a separating starburst)}\label{secondproofsubsection}

For the duration of this subsection and Subsection \ref{Nonseparating subsection} we fix the notation $\bm = (2m_1,\dots,2m_\ell)$, $(M,\bS)$, and $(V^j)_{i=1}^{2\ell}$ as in Theorem \ref{Main theorem}, and we assume that $\prn{M \dbs \bS,(V^j)_{i=1}^{2\ell}}$ is $\bm$-hyperbolic.

This subsection focuses mainly on giving an analog of the proof given for Theorem \ref{Main theorem separating} without constructions that essentially rely on the separating condition.
This will provide motivation for the proof of Theorem \ref{Main theorem} in Subsection \ref{Nonseparating subsection}. However, the proof of Theorem \ref{Main theorem} is independent of this subsection.

Recall the set $W$ of cyclic words on alphabet $\cbr{T_i}$ parameterizing the bracelet links formed by the tangles $\cT_i$ of a generic bracelet link on $2m$ tangles, as defined in Construction \ref{Cyclic word construction}.
We had previously defined a set of manifolds $\cbr{\cS(w) \mid w \in W}$ as the topological foundation of our strategy.

However, the manifold $\cS(w)$ consisted of individual pieces $M_i$, and hence it cannot directly generalize to $M \dbs \bS$.
Intuitively, we need to \emph{fill in} the remaining $(M_j)_{i \neq j}$ corresponding with each copy of $M_i$ in $\cS(w)$;
in the separating case, this can be done simply by defining the manifolds
\begin{equation}\label{Symmetric manifold equation}
    \cSc\prn{T_{i_1} \cdots T_{i_{2m}}} := \coprod_{j=1}^{2m} \cS\prn{T_{i_1 + j} \cdots T_{i_{2m} + j}}
\end{equation}
and extending this to $\ZZ_{\geq 0}[W]$ by $\cSc(\sum_i w_i) := \coprod_i \cSc(w_i)$.

Our manifolds are grounded to $M$ and $D^{2m}(M \dbs \bS,(V^1,V^2))$ by the following identifications:
\begin{align*}
    \cSc\prn{T_1 \cdots T_{2m}} &= 2m \cdot \cS\prn{T_1 \cdots T_{2m}}\\
    &= \coprod_{i=1}^{2m} M\\
    \cSc(T_1^{2m})
    &= \coprod_{i=1}^{2m} \cS(T_i^{2m})\\
    &= \coprod_{i=1}^{2m} D^{2m}(M_i,(V^1_i,V^2_i))\\
    &= D^{2m}(M \dbs \bS, (V^1,V^2)).
\end{align*}
Furthermore, there is a natural notion of volume
\[
    \volc(w) := \vol \cSc(w)
\]
with the volume of non-hyperbolic manifolds equal to 0.
It suffices to prove that $\volc\prn{T_1 \cdots T_{2m}} \geq \volc\prn{T_1^{2m}}$.

In fact, the relation
\[
    2\volc(a_1a_2) \geq \volc(a_1a_1^R) + \volc(a_2a_2^R)
\]
follows from the same relation for the volume of each summand in \eqref{Symmetric manifold equation}.
This together with the algebra of Subsection \ref{Separating subsection} immediately imply Theorem \ref{Main theorem separating}.

\subsection{Proof of the main theorem in the general case}\label{Nonseparating subsection}
Using the results of Subsection \ref{Non-orientable subsection}, we henceforth assume that $M$ is orientable.
By boundary compatibility, we may double across higher genus boundary, which takes the original starburst to another starburst in the resulting manifold, and henceforth assume that $M$ has torus boundary.

As exemplified in Subsection \ref{secondproofsubsection}, we may prove Theorem \ref{Main theorem} via the following strategy:
\begin{enumerate}
    \item First, we define a set of manifolds $\cbr{\cSc(w) \mid w \in W}$ such that $\cSc(T_i^{2m}) = D^{2m}(M \dbs \bS, (V^1, V^2))$ and $\cSc(T_1 \cdots T_{2m})$ is the total space of a $2m$-fold cover of $M$.
    \item Next, we prove the relation $2\vol \cSc(a_1a_2) \geq \vol \cSc(a_1a_1^R) + \vol \cSc(a_2a_2^R)$ holds in this case.
    \item Last, we use the algebraic results given in Subsection \ref{Separating subsection} to prove the theorem. 
\end{enumerate}

We recall some notation for the purpose of defining $\widetilde S(w)$.
Recall that the set $d(\Sigma_i) \subset \partial M \dbs \bS$ consists of the points corresponding with $\Sigma_i$;
for each $i$ we write $V_i^2$ for the clockwise side of the points of $d(\Sigma_i)$ and $V^{i}_1$ for the counter-clockwise side of the points of $d(\Sigma_{i+1})$.
We have $V^j = \bigsqcup_i V^j_i.$

In this language, we can describe the $2m$-replicant as
\[
    D^{2m}(M \dbs \bS, (V^1,V^2)) = \frac{\coprod_{k=1}^{2m} (M \dbs \bS)}{\prn{V_i^{1},2k} \sim \prn{V_i^{2},2k + 1} \hspace{10pt} \text{and} \hspace{10pt} \prn{V_i^{2},2k - 1} \sim \prn{V_i^{1},2k}}
\]
We will extend this to arbitrary elements of $\widetilde W$ in the following construction.
\begin{construction}
Setting $\tilde w = t_1 \cdots t_{2m}$, and fixing (arbitrarily) a choice of first letter, define 
\begin{align*}
    R^1(k,\tilde w) &:= \begin{cases}
        1 & \text{if}\hspace{.1cm} t_k = T_j \text{ for some } j\\
        2 & \text{if}\hspace{.1cm} t_k = T_j^R \text{ for some } j\\
    \end{cases}\\
     R^2(k,\tilde w) &:= \begin{cases}
        2 & \text{if}\hspace{.1cm} t_k = T_j \text{ for some } j\\
        1 & \text{if}\hspace{.1cm} t_k = T_j^R \text{ for some } j\\
    \end{cases}\\
    \text{if} \hspace{10pt} t_k \in \cbr{T_j,T_j^R} \hspace{10pt} \text{and} \hspace{10pt} t_{k+1} \in \cbr{T_{j'},T_{j'}^R} \hspace{60pt} \off(k,\tilde w) &:= j' - j
\end{align*} 
Using this language, we may define the following:
\[
    \cSc(\tilde w) := \frac{\coprod_{k=1}^{2m} (M \dbs \bS)}{\prn{V_i^{R^1(2k,\tilde w)},2k} \sim \prn{V_{i + \off(2k,\tilde w)}^{R^2(2k+1,\tilde w)},2k + 1} \hspace{10pt} \text{and} \hspace{10pt} \prn{V_i^{R^2(2k - 1,\tilde w)},2k - 1} \sim \prn{V_{i + \off(2k - 1,\tilde w)}^{R^1(2k,\tilde w)},2k}}
\]
Note that there exists a diffeomorphism $\cSc(T_1T_2) \simeq \cSc(T_1^RT_2^R)$, and hence this descends to a construction of a $W$-indexed collection of manifolds.
\end{construction}
This generalizes the construction of $\cSc(w)$ in the separating case;
for instance, in that case the identification 
\[
\prn{V_i^{R_1(2k,w)},2k} \sim \prn{V_{i+\off(2k,w)}^{R_2(2k+1,w)},2k+1}
\]
represents the gluing of the piece at index $2k$ to the piece at index $2k + 1$ within the submanifold $\cS(t_{j_1 + i} \cdots t_{j_{2m} + i}) \subset \cSc(w)$.
The reader may verify that, for each tuple $(i,j,k)$, there is a unique tuple $(i',j',k')$ such that there exists an identification $(V_i^j,k) \sim (V_{i'}^{j'},k')$ in $\cSc(w)$.

Now that we've defined $\widetilde S(w)$, we verify its relationship with $(M,\bS)$ in the following lemma.
\begin{lemma}\label{Pieces and manifold lemma}
    Let $(M,\bS,\cbr{\rho})$ be an $m$-sheeted starburst, and let $\cSc(w)$ be the corresponding word for each $w \in W$.
    \begin{enumerate}
        \item For each $j$ there is a diffeomorphism
        \[
            \cSc\prn{T_j^{2m}} \simeq D^{2m}(M \dbs \bS)
        \]
            so that $\cSc\prn{T_j^{2m}}$ is hyperbolic if and only if $M \dbs \bS$ is $2m$-hyperbolic, in which case we have
        \[
            \volc\prn{T_j^{2m}} = 2m \cdot \vol^{2m}(M \dbs \bS).
        \]
        \item There is a $2m$-fold covering map
        \[
            \cSc\prn{T_1 \cdots T_{2m}} \rightarrow M
        \]
            so that $\cSc\prn{T_1 \cdots T_{2m}}$ is hyperbolic if and only if $M$ is hyperbolic, in which case we have
        \[
            \volc\prn{T_1 \cdots T_{2m}} = 2m \cdot \vol(M).
        \]
    \end{enumerate}
\end{lemma}
\begin{proof}
    Part (1) follows by noting that
    \[
    \cSc(T_j^{2m}) = \frac{\coprod_{k=1}^{2m} (M \dbs \bS)}{\prn{V_i^{1},2k} \sim \prn{V_i^{2},2k + 1} \hspace{10pt} \text{and} \hspace{10pt} \prn{V_i^{2},2k -1} \sim \prn{V_i^{1},2k}} = D^{2m}(M \dbs \bS,(V^1,V^2))
    \]
    
    For part (2), note that
    \[
    \cSc(T_1 \cdots T_{2m}) = \frac{\coprod_{k=1}^{2m} (M \dbs \bS)}{\prn{V_i^{2},k} \sim \prn{V^{i + 1}_{1},k+1}}
    \]
    has points either given by a point in the interior of $M \dbs \bS$, by an equivalence class of two points $(p,k)$ and $(p',k')$ where $p$ and $p'$ are points of $V_i^2$ and $V_{i+1}^1$ who correspond with the same point of $\Sigma_{i+1}$, or by an equivalence class of $2m$ points who correspond with the same point of $E$.
    This correspondence outlines a map
    \[
        \cSc(T_1 \cdots T_{2m}) \rightarrow M.
    \]
    
    We will exhibit this map as a covering map, with local trivialization at points in the interior of $M \dbs \bS$ by an $M$-neighborhood which doesn't intersect any element of $\bS$.
    For points on some $\Sigma_i \setminus E$, we may choose the trivializing neighborhood to be the union of two ``half-balls'' in two different copies of $M \dbs \bS$ in $\cSc(T_1 \cdots T_{2m})$.
    For points on $E$, we may choose the trivializing neighborhood to be the union of $2m$ ``slices'' in the $2m$ copies of $M \dbs \bS$ in $\cSc(T_1 \cdots T_{2m})$.
    
    Note that the fiber of each point $p \in M \dbs \bS$ is given by $\cbr{(p,k) \mid 1 \leq k \leq 2m}$.
    Hence the fibers of this covering map have size $2m$, as desired.
\end{proof}

As outlined in our strategy, we now prove the following proposition, implementing the relations used in the proof of Theorem \ref{Main theorem separating}.
\begin{lemma}\label{Volume relation lemma}
    Suppose that $w = a_1a_2 \in \Wknot$ where $\operatorname{len}(a_1) = m$, and define $w_j := a_ja_j^R$.
    If $\cSc(w_1)$ an $\cSc(w_2)$ are hyperbolic, then $\cSc(w)$ is hyperbolic, with volume satisfying
    \[
        2\volc(w) \geq \volc\prn{w_1 + w_2}.
    \]
\end{lemma}
\begin{proof}
    Assume without loss of generality that $\cSc(w_1)$ was constructed fixing the first letter of $a_1$ as the first letter of the word.
    Then, there is a diffeomorphism of $\coprod_{k=1}^{2m} M \dbs \bS$ sending $\prn{M \dbs \bS,k} \rightarrow \prn{M \dbs \bS, 2m - k}$.
    For a general word $v$ this descends to a diffeormophism $\cSc(v) \rightarrow \cSc(v^R)$;
    in the case of a palindromic word, this descends to a diffeomorphism of $\cSc(v)$.
    
    In particular, the fixed point set of this diffeomorphism is the surface
    \[
        F_j := \bigcup_k \prn{V_1^{R_1(1,w)} \cup V_m^{R_2(m,w)},k} \subset \cSc(w_j).
    \]
    By Theorem \ref{fixed surface lemma}, $F_j$ is totally geodesic;
    since $a_1a_2 \in W$, the surfaces $F_1$ and $F_2$ are diffeomorphic.
    Hence we may form the 3-manifold $X$ by gluing together the $\cSc(a_ja_j^R) \dbs F_j$ along a diffeomorphism of the points corresponding to $F_j$ induced by the diffeomorphism $F_1 \simeq F_2$.
    Then, by Theorem \ref{Cutting and pasting theorem}, $X$ is hyperbolic, with volume satisfying
    \[
        \vol(X) \geq \volc(w_1 + w_2).
    \]
    
    To prove the lemma, it suffices to supply a double cover $p:X \rightarrow \cSc(w)$.
    This may be defined by noting that there exists a map $\cSc(w_1) \dbs F_1 \sqcup \cSc(w_2) \dbs F_2 \rightarrow \cSc(w)$ identifying the points corresponding with word $a_j$ in the domain with the same points corresponding with $a_j$ in the codomain;
    two points in the domain are identified by the diffeomorphism of $F_1$ and $F_2$ only if they are mapped to the same point of $\cSc(w)$, so this descends to a map $p:X \rightarrow \cSc(w)$.
    
    We verify that $p$ is a covering map.
    For a point contained within the points corresponding with subword $a_j$ and away from the image of $F_j$ in $\cSc(w)$, we may choose a trivializing neighborhood which is contained in the points corresponding with $a_j$.
    For a point on the image of $F_j$, we may choose the neighborhood composed of the union of ``half-balls'' in the points corresponding with each $a_1$ and $a_2$. 
    Hence this is a covering map.
    
    It is clear that this map has fibers of size 2, and hence $\cSc(w)$ is hyperbolic with volume satisfying
    \[
        2 \cdot \volc(w) = \vol(X) \geq \volc(w_1 + w_2),
    \]
    as desired.
\end{proof}

\begin{proof}[Proof of Theorem \ref{Main theorem}]
    Given Lemmas \ref{Pieces and manifold lemma} and \ref{Volume relation lemma}, the proof of Proposition \ref{Reducability proposition} yields hyperbolicity of $M$ and the inequality
    \begin{align*}
        2m \cdot \vol(M)
        &= \volc(T_1 \cdots T_{2m})\\ 
        &\geq \frac{1}{2m} \sum_{i=1}^m \volc(T_i)\\
        &= 2m \cdot \vol^{2m}(M \dbs \bS).
    \end{align*}
    
    We now characterize when this equality is attained.
    Let $a_1 := T_1 \cdots T_{m}$ and let $a_2 := T_{m+1} \cdots T_{2m}$.
    Then, in the notation of Lemma \ref{Volume relation equation}, the image of the surfaces $F^1_j \subset \cSc(a_ja_j^R)$ in $\cSc(T_1 \cdots T_{2m})$ is given by the preimage of $\Sigma_1$ under the covering map $\cSc(T_1 \cdots T_{2m}) \rightarrow M$.
    Call this surface $F^1$, and note that $F^1$ is totally geodesic if and only if $\Sigma_1$ is totally geodesic;
    this generalizes immediately to $\Sigma_j$, and let the associated surface be $F^j$.
    
    Suppose first that some $\Sigma_j$ is not totally geodesic.
    Without loss of generality, we may fix $j = 1$.
    Then, note that $F^1 \subset \cSc(T_1 \cdots T_{2m})$ is not totally geodesic, so its preimage in the total space of the double cover $X \rightarrow \cSc(T_1 \cdots T_{2m})$ is not totally geodesic either.
    Then, since $X$ is formed by cutting and pasting $\cSc(a_ja_j^R)$ along diffeomorphisms of totally geodesic surfaces, Theorem \ref{Cutting and pasting theorem} implies that $\vol(X) > \volc(a_1a_1^R) + \volc(a_2a_2^R)$.
    This in turn implies
    \begin{align*}
        4m \cdot \vol(M)
        &= 2 \cdot \volc(T_1 \cdots T_{2m})\\
        &= \vol(X)\\
        &> \volc(a_1a_1^R) + \volc(a_2a_2^R)\\
        &\geq 4m \cdot \vol^{\cdot} (M \dbs \bS),
    \end{align*}
    as desired.

    Suppose conversely that each $\Sigma_i$ is totally geodesic;
    then, each of the surfaces $F^j \subset \cSc(T_1 \cdots T_{2m})$ are totally geodesic, and hence each of the preimages of these surfaces in $X$ are totally geodesic.
    We may repeatedly cut along totally geodesic surfaces to yield $\coprod_{k=1}^{2m} (M \dbs \bS)$, and glue along isometric totally geodesic surfaces to yield $\cSc(T_i) = D^{2m}(M \dbs \bS, (V^1,V^2))$, directly yielding the equality
    \[
        2m \cdot \vol(M) 
        = \volc(T_1 \cdots T_{2m})
        = \vol^{2m}(M \dbs \bS). 
    \]
\end{proof}

Having proven our theorem in the case of a starburst in a manifold, we may extend our theorem to starbursts in $\ell$-pieces identically to our extension from Theorem \ref{Main theorem separating} to Theorem \ref{Main theorem separating inductive}, with references to Theorem \ref{Main theorem separating} replaced by Theorem \ref{Main theorem}, and omitting the word ``separating.''

\subsection{Composition of pieces}
In the following short section, we state and prove a corollary to Theorem \ref{Main theorem} about \emph{combining hyperbolic pieces}.
We will use this to prove hyperbolicity of composition of tangles in Sections \ref{Knot theory section} and \ref{hyperbolicitytangles}.
\begin{corollary}\label{General composition corollary}
    Let $\cbr{(P_i, (V_i^1,V_i^2))}$ be a set of $n$ pieces, each $2mn$-hyperbolic, equipped with diffeomorphisms $\phi_i:V_i^2 \xrightarrow\sim V_{i+1}^1$ for each $1 \leq i \leq {n-1}$.
    Then, the piece
    \[
        (P,(V_1^1, V_n^2)) \hspace{50pt} \text{where} \hspace{50pt} P = \frac{\coprod_i P_i}{V_i^2 \sim_{\phi_i} V_{i+1}^2}
    \]
    is $2m$-hyperbolic, with volume satisfying
    \[
        \vol^{2m}(P) \geq \frac{1}{n} \sum_{i=1}^n \vol^{2mn}(P_i).
    \]
\end{corollary}
\begin{proof}
    Note that $D^{2m}(P,(V_1^1,V_n^2))$ contains, by tracking the images of $V_i^j$, an $mn$-sheeted starburst $\bS$ which separates the manifold into $m$-namy copies of $P_i$ for each $i$.
    Hence Theorem \ref{Main theorem} applied to the replicant $D^{2m}(P,(V_1^1,V_n^2))$ yields that the manifold is hyperbolic, with volume bound as described in the corollary.
\end{proof}

\section{Applications to knot theory} \label{Knot theory section}
Given a manifold $M$ and a surface $\Sigma \subset \partial M$, we define a \emph{tangle in $(M,\Sigma)$} to be an embedded compact 1-manifold in $M$ with boundary contained in $\Sigma$.
Given an $\ell$-piece $(P,(V^j))$, we define a \emph{tangle in $(P,(V^j))$} to be a tangle in $(P,\bigcup_n V_{2n-1} \cap V_{2n})$.

A link may be viewed as a tangle with no arcs. 
Given a tangle in an $\ell$-piece $(M,\prn{V_i})$ and a starburst $(M,\bS,\rho)$ in the $\ell$-piece, there is an associated tangle in the $(\ell+1)$-piece on $M \dbs \bS$.
Further, if we have an $\ell$-piece that is a  disjoint union  $(M \sqcup M',\cbr{V_i \sqcup V'_i})$ containing a tangle $\cT$, then $\cT$ decomposes into the disjoint union of tangles in $(M,\cbr{V_i})$ and $(M',\cbr{V'_i})$. 

We explore various applications of Theorem \ref{Main theorem} to the theory of links that are constructed out of tangles.
In Subsection \ref{thickenedtorussection} we explore hyperbolic links within the thickened torus, which can be thought of as the complement of the Hopf link in $S^3$. We decompose the links into thickened-cylinder tangles, and further decompose those into cubical tangles.
In Subsection \ref{solidtorussection}, we consider links in an open solid torus, which can be thought of as the complement of a trivial component in $S^3$. There, we decompose links into wedge tangles.
Following this, in Subsection \ref{Bracelet subsection} we explore hyperbolic links in the 3-sphere which decompose into a cycle of saucer tangles.

Later, in Theorem \ref{Saucer hyperbolicity theorem}, we show that $2n$-hyperbolicity of saucer tangles implies $2m$-hyperbolicity for any $m \geq n$, and hence a cycle of $n \geq 2m$ saucer tangles which are each $2m$-hyperbolic glue together to yield a hyperbolic link.
We decompose saucer tangles further into tetrahedral tangles, and prove the corresponding results there.

In Subsection \ref{spheretimescirclesubsection}, we consider torus lattice links in $S^2 \times S^1$ and decompose them into thickened bigon tangles, proving similar results in this context as well.



\subsection{Links within the thickened torus}\label{thickenedtorussection}
Let $T := S^1 \times S^1$ be the torus.
We define two different classes of tangles arising from the study of links in $T \times I$.
\begin{definition}\label{Thickened torus tangle definition}
    Let $C := S^1 \times I \times I$ be the thickened cylinder and let $R_0 := S^1 \times I \times \cbr{0}$ and $R_1 := S^1 \times I \times \cbr{1}$ be the annuli at heights $\{ 0 \}$ and $\{ 1 \}$.
    A \emph{thickened-cylinder tangle} $\cT_i$ is a tangle in $(C,R_0 \cup R_1)$, with endpoints in the two annuli at the ends of the thickened cylinder.
    By reflecting across $R_0$ and $R_1$,  a thickened-cylinder  tangle $\cT_i$ determines a link $\Dthick{2}{\cT_i} \subset S^1 \times S^1 \times I$ in a thickened torus.
    We say that $\cT_i$ is \emph{2-hyperbolic} if $\Dthick{2}{\cT_i}$ is a hyperbolic link, in which case we define
    \[
        \volthick{2}{\cT_i} := \frac{\vol\prn{\Dthick{2}{\cT_i}}}{2}.
    \]
    
    Let $Q = I^3$ be the cube.
    Define the \emph{top and bottom faces} $R_{0i} := I \times \cbr{i} \times I$ and the \emph{left and right faces} $R_{1i} := I \times I \times \cbr{i}$.
    A \emph{cubical tangle} $\cT_{ij}$ is a tangle in $(Q,I \times \partial(I^2) - I \times \partial I \times \partial I)$.
    By reflecting in two opposite faces of the cube to obtain a second copy glued to the first,  a cubical tangle $\cT_{ij}$ determines  a thickened-cylinder  tangle.  Via another  reflection operation, this in turn determines a link $\Dthick{(2,2)}{\cT_{ij}} \subset S^1 \times S^1 \times I$ in a thickened torus.
    We say that $\cT_{ij}$ is \emph{$(2,2)$-hyperbolic} if $\Dthick{(2,2)}{\cT_{ij}}$ is a hyperbolic link, in which case we define
    \[
        \volthick{(2,2)}{\cT_{ij}} := \frac{\vol\prn{\Dthick{(2,2)}{\cT_{ij}}}}{4}.
    \]
\end{definition}

See Figure \ref{cubicaltangle} for cubical tangle $\cT$ that when doubled, yields a thickened-cylinder tangle, which when doubled again generates $D^{(2,2)}(\cT)$.

\begin{figure}[htpb]
    \centering
    \includegraphics[width=.9\textwidth]{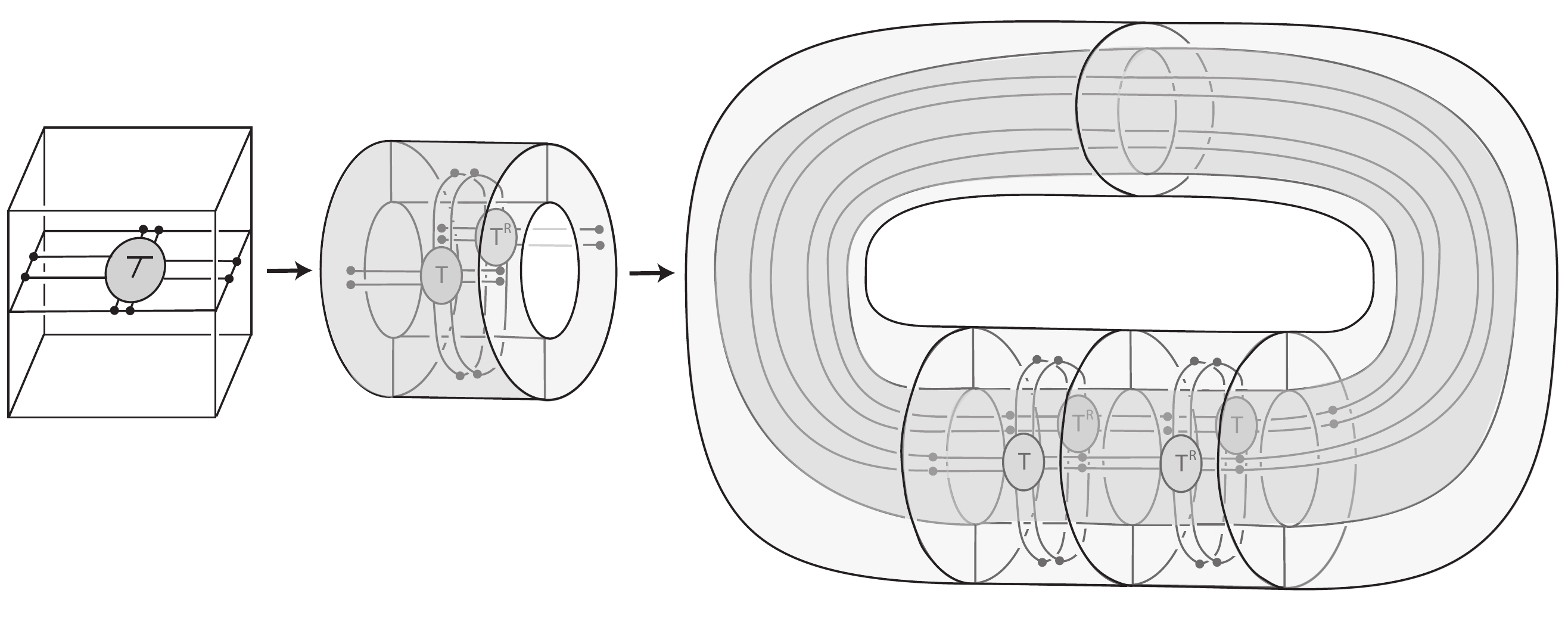}
    \caption{ A cubical tangle generates a thickened-cylinder tangle which in turn generates a link in a thickened torus.
    }
    \label{cubicaltangle}
\end{figure}

The following theorems summarize the situations of Figures \ref{Torus cycle figure} and \ref{thickened torus example}, and is an easy corollary of Theorem \ref{Main theorem}.
\begin{theorem}\label{Torus theorem} 
    Let $L \subset T \times I$ be a link in a thickened torus such that there exists disjoint surfaces separating $L$ into thickened-cylinder tangles $\cT_i$ and each $\cT_i$ is 2-hyperbolic.
    Then $L$ is hyperbolic, and its volume satisfies
    \[
        \vol(L) \geq \sum_i \volthick{2}{\cT_i}.
    \]
    \qed
\end{theorem}

\begin{theorem} \label{Square tangles theorem}
    Let $\cT_i$ be  a thickened-cylinder  tangle, and suppose there are disjoint surfaces separating $\cT_i$ into a cycle $\cT_{ij}$ of cubical tangles and each $\cT_{ij}$ is $(2,2)$-hyperbolic.
    Then, $\cT_i$ is $2$-hyperbolic within, with volume satisfying
    \[
        \volthick{2}{\cT_i} \geq \sum_j \volthick{(2,2)}{\cT_{ij}}.
    \]
    If a link $L$ in a thickened torus decomposes into cubical tangles such that each $\cT_{ij}$ is $(2,2)$-hyperbolic, then $L$ is hyperbolic with volume satisfying
    \[
        \vol(L) \geq \sum_{ij} \volthick{(2,2)}{\cT_{ij}}.
    \]
    \qed
\end{theorem}
Such surfaces as in Theorem \ref{Torus theorem} can be found by taking a sequence of disjoint punctured annuli tracing out a $(p,q)$-curve on each boundary torus of $T \times I$.
Since there is a homeomorphism of $T \times I$ taking such a punctured annulus to one bounding meridians, Theorem \ref{Torus theorem} may be pictured as \emph{cutting the thickened torus along meridians} as in Figure \ref{Torus cycle figure}.
Further, Theorem \ref{Square tangles theorem} may be pictured by cutting a thickened cylinder into cubes, as in Figure \ref{thickened torus example}.

Note that diffeomorphisms of pieces lift to diffeomorphisms of their $2n$-replicants.
Hence Mostow-Prasad rigidity implies that we may replace any piece with a diffeomorphic piece while maintaining the same information regarding hyperbolicity and volume.
Using this, we may replace each thickened-cylinder tangle with one whose endpoints are ``standardized'' such that whenever thickened-cylinder tangle $\cT_i$ has the same number of endpoints at its top as $\cT_j$ has at its bottom, there is an evident gluing of the tangles 
\[
    (C,\cT_i \circ \cT_j) := (C,\cT_i) \cup_{R_1 \sim R_0} (C,\cT_j)
\]
which attaches the first copy of $R_1$ to the second copy of $R_0$.

Through an analogous process, we may define two ways of gluing cubical tangles:
\emph{top-to-bottom}, or \emph{left-to-right}, written
\[
    (Q,\cT_{ij} \circ_a \cT_{kl}) := (Q,\cT_{ij}) \cup_{R_{a0} \sim R_{a1}} (Q,\cT_{kl}))
\]
whenever the numbers of endpoints are compatible.

The following establishes a relationship between hyperbolicity and composition, and it follows as a clear application of Theorem \ref{Main theorem inductive}.
\begin{proposition}\label{Tangle composition proposition}
    Suppose $\cT_i,\cT_j$ are $2$-hyperbolic thickened-cylinder tangles, and suppose that that $\cT_i \circ \cT_j$ is defined.
    Then, $\cT_i \circ \cT_j$ is $2$-hyperbolic with volume satisfying
    \[
        \volthick{2}{\cT_i \circ \cT_j} \geq \volthick{2}{\cT_i} + \volthick{2}{\cT_j}
    \]
    
    Suppose instead that $\cT_{ij},\cT_{kl}$ are $(2,2)$-hyperbolic cubical  tangles and suppose that $\cT_{ij} \circ_a \cT_{kl}$ is defined.
    Then, $\cT_{ij} \circ_a \cT_{kl}$ is $(2,2)$-hyperbolic in a thickened torus, with volume satisfying
    \[
        \volthick{(2,2)}{\cT_{ij} \circ_a \cT_{kl}} \geq
        \volthick{(2,2)}{\cT_{ij}}
        + \volthick{(2,2)}{\cT_{kl}}
    \]
    \qed
\end{proposition}

In a subsequent paper \cite{selectalternating}, we will provide many examples of 2-hyperbolic thickened-cylinder tangles. 

\subsection{Links within a solid torus} \label{solidtorussection}
Let $V := D \times S^1$ be the solid torus, where $D$ is the 2-dimensional disk.
Unlike the thickened torus case, cutting links $L \subset V$ via starbursts that intersect the boundary $\partial V = T$ in meridians and longitudes yields tangles in different ambient spaces.
In particular, cutting along a starburst with longitude boundaries yields a cycle of tangles living in \emph{cyclical wedges}, and cutting along a surface bounded by disjoint meridians yields a cycle of tangles living in solid cylinders.
Each of these may be further cut into square tangles that live in wedges.

We give analogous definitions to the thickened torus case in order to find lower bounds on volumes of links $L \subset V$.

\begin{definition}\label{Solid torus tangle definition}
    A \emph{solid-cylinder tangle} $\cT_i$ is a tangle in $(D \times I, D \times \partial I)$.
    In a similar manner to thickened-cylinder tangles, doubling a solid-cylinder tangle $\cT_i$ over $D \times \partial I$ determines a link in a solid torus, written
    \[
        (V,\Dsolid{2}{\cT_i}) := (D \times I,\cT_i) \cup_{D \times \partial I} (D \times \cT_i).
    \]
    We say that $\cT_i$ is \emph{$2$-hyperbolic} if $\Dsolid{2}{\cT_i}$ is hyperbolic, in which case we define
    \[
        \volsolid{2}{\cT_i} := \frac{\vol\prn{\Dsolid{2}{\cT_i}}}{2}.
    \]
    
    Let $P$ be the solid triangle, and let $F \subset P$ be the union of two edges of $P$.
    A \emph{cyclic wedge tangle} $\cT_i$ is a tangle in $(S^1 \times P, S^1 \times F)$.
    Via the reflection $r:P \rightarrow P$ sending one edge of $F$ to another,  a cyclic wedge tangle $\cT_i$ determines a reflected cyclic wedge tangle $\cT_i^R$.
    Then, gluing adjacent tangles together along the length $2n$ cycle sequence $(\cT_i,\cT_i^R,\dots,\cT_i^R)$ determines a link $\Dsolid{2n}{\cT_i} \subset V$, called the \emph{$2n$-replicant}.
    We say that $\cT_i$ is \emph{$2n$-hyperbolic} if $\Dsolid{2n}{\cT_i}$ is hyperbolic, in which case we define the volume
    \[
        \volsolid{2n}{\cT_i} := \frac{\vol\prn{\Dsolid{2n}{\cT_i}}}{2n}.
    \]
    
    A \emph{wedge tangle} $\cT_{ij}$ is a tangle in $(I \times P, I \times F \cup \partial I \times P)$.
    Similarly, $r$ determines a reflected tangle $\cT^R_{ij}$, the sequence $(\cT_{ij},\cT_{ij}^R,\dots,\cT_{ij}^R)$ determines a solid cylindrical tangle, and doubling determines a cyclic wedge tangle.
    Fixing the number $2n$, the $2n$-replicant of the cyclic wedge tangle is the same as the $2$-replicant of the solid cylindrical tangle, so we may define the \emph{$(2n,2)$-replicant} $\Dsolid{(2n,2)}{\cT_i}$ to be the resulting link from either replication.
    We say that $\cT_{ij}$ is \emph{$(2n,2)$-hyperbolic} if $\Dsolid{(2n,2)}{\cT_{ij}}$ is hyperbolic, in which case we define
    \[
        \volsolid{(2n,2)}{\cT_{ij}} := \frac{\vol\prn{\Dsolid{(2n,2)}{\cT_{ij}}}}{4n}.
    \]
\end{definition}

See Figure \ref{wedgetangle} for a wedge tangle $\cT$ which, when doubled one way, yields a solid-cylinder tangle, and when doubled another way, yields a cyclic wedge tangle. Each of those, when doubled, both generate $D^{(2,2)}(\cT)$.
The fact that doubling over the two sets of endpoints commute with each other is a special case of Proposition \ref{Replication commuting proposition}.

\begin{figure}[htpb]
    \centering
    \includegraphics[width=.7\textwidth]{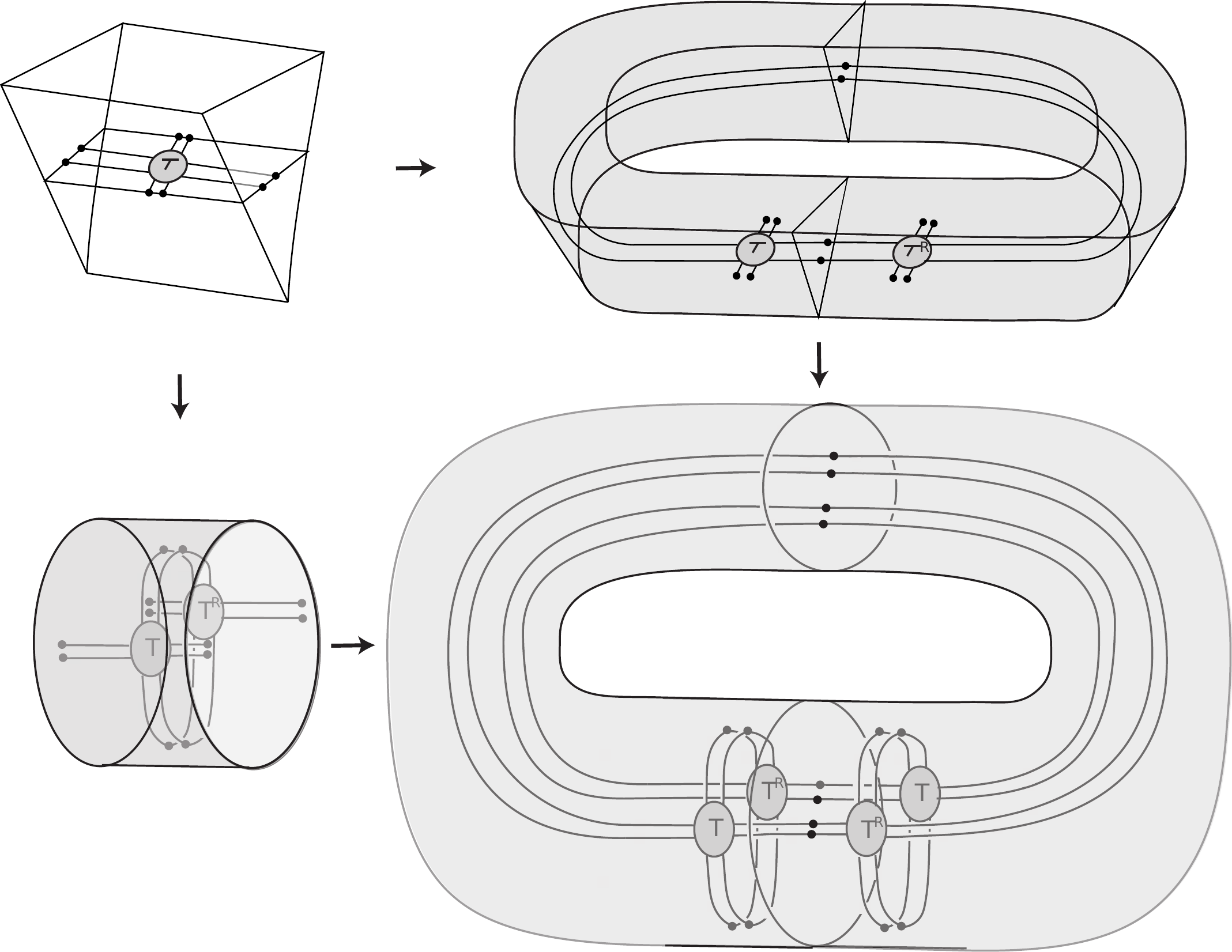}
    \caption{A wedge tangle generates either a solid-cylinder tangle or a cyclic wedge tangle, each of which generates the same link in the solid torus.}
    \label{wedgetangle}
\end{figure}

Using this notation, we have a similar result to Theorems \ref{Torus theorem} and \ref{Square tangles theorem}, which we condense into a single theorem which follows immediately from Theorem \ref{Main theorem inductive}.
\begin{theorem}\label{Solid torus theorem}
    Let $L \subset V$ be a link in a solid torus.
    \begin{enumerate}[label=\upshape{(\roman*)},ref={\thetheorem(\roman*)}]
        \item
            Suppose there exists a surface separating $L$ into solid-cylinder tangles $\cT_i$ such that each $\cT_i$ is $2$-hyperbolic.
            Then $L$ is hyperbolic, and its volume satisfies
            \[
                \vol(M) \geq \sum_i \volsolid{2}{\cT_i}.
            \]
        \item Suppose there exists a starburst separating $L$ into $2n$ cyclic wedge tangles $\cT_i$ such that each $\cT_i$ is $2n$-hyperbolic.
            Then $L$ is hyperbolic, and its volume satisfies
            \[
                \vol(M) \geq \sum_i \volsolid{2n}{\cT_i}.
            \]
        \item
            Suppose there is a starburst separating a solid-cylinder tangle $\cT_{i}$ into a cycle of $2n$ wedge tangles $\cT_{ij}$ such that each wedge tangle is $(2n,2)$-hyperbolic.
            Then, $\cT_i$ is $2$-hyperbolic in the solid torus, with volume satisfying
            \[
                \volsolid{2}{\cT_i} \geq \sum_j \volsolid{(2n,2)}{\cT_{ij}}.
            \]
        \item
            Suppose there is a starburst separating a cyclic wedge tangle $\cT_{i}$ into a cycle of wedge tangles $\cT_{ij}$ such that each wedge tangle is $(2n,2)$-hyperbolic in a solid torus.
            Then, $\cT_i$ is $2n$-hyperbolic in the solid torus, with volume satisfying
            \[                   \volsolid{2n}{\cT_i} \geq \sum_j \volsolid{(2n,2)}{\cT_{ij}}.
            \]
    \end{enumerate}
    \qed
\end{theorem}

We describe an analog to Proposition \ref{Tangle composition proposition};
the class of 2-hyperbolic solid-cylinder tangles is closed under composition by Theorem \ref{Main theorem}.
Further, Corollary \ref{General composition corollary} implies that the composition of $n$ cyclic-wedge tangles which are each $2mn$-hyperbolic is $2m$-hyperbolic in the solid torus.
Similar statements hold for wedge tangles, depending on which direction the composition follows.

There is a strong relationship between solid-cylinder tangles and thickened-cylinder tangles.
If a solid-cylinder tangle $\cT$ has an unknotted strand with endpoints on both components of $D \times \partial I$, then the complement of that strand determines a thickened-cylinder tangle.
The solid-cylinder tangle may be reconstructed from this thickened-cylinder tangle.
The complement of the double of the solid-cylinder tangle is homeomorphic to the complement of the associated thickened-cylinder tangle, so \cite{selectalternating} will yield a large class of examples of $2$-hyperbolic solid-cylinder tangles.

\subsection{Bracelet links within the sphere} \label{Bracelet subsection}


\begin{definition}\label{Bracelet tangle definition}
A \emph{saucer tangle} is a tangle in the piece defined in Figure \ref{saucer} which has at least two endpoints on each face and composition of saucer tangles is defined by attaching their endpoints, as in the thickened-cylinder case. As defined in the introduction, if $\cT_i$ is a saucer tangle, then the length-$2n$ sequence $(\cT_i,\cT_i^R,\dots,\cT_i^R)$ determines a bracelet link $\Dsphere{2n}{\cT_i} \subset S^3$ in the 3-sphere, called it $2n$-replicant.
   We say that $\cT_i$ is \emph{$2n$-hyperbolic} if $\Dsphere{2n}{\cT_i}$ is a hyperbolic link, in which case we define the volume
    \[
        \volsphere{2n}{\cT_i} := \frac{\vol\prn{\Dsphere{2n}{\cT_i}}}{2n}.
    \]
\end{definition}

\begin{remark}
    We make the restriction that saucer tangles have two endpoints on each face in service of Theorem \ref{Saucer hyperbolicity theorem}.
    This does not fundamentally change the class of $2m$-hyperbolic saucer tangles relative to the class of tangles without such a restriction;
    in the case $m \geq 2$, any tangle with $\leq 1$ endpoints on one face will have $2m$-replicant which is not a prime link, and hence is not hyperbolic, and in the case of a $2m=2$ hyperbolic tangle, we can choose a piece structure which has the same $2$-replicant and has at least 2 endpoints on each face.  
\end{remark}

Unlike Theorems \ref{Torus theorem}, \ref{Square tangles theorem}, and \ref{Solid torus theorem}, the intersection of the surfaces making up the evident starburst is closed in this case;
hence we may directly conclude the following theorem from Theorem \ref{Main theorem separating}.
\begin{theorem}\label{Bracelet theorem}
If $L$ is a bracelet link made of a cycle $(\cT_i)_{i=1}^{2n}$ of $2n$ saucer tangles such that each $\cT_i$ is $2n$-hyperbolic, then $L$ is hyperbolic and the volumes satisfy
    \[
        \vol(L) \geq \sum_i \volsphere{2n}{\cT_i}.
    \]
    \qed
\end{theorem}

Hyperbolicity of saucer tangles will be discussed further in Section \ref{hyperbolicitytangles}, including composition of tangles and the relationship between $2m$-hyperbolicity and $2m'$-hyperbolicity for $m,m'$ distinct.

    
\subsection{Links in \texorpdfstring{$S^2 \times S^1$}{S2 x S3}} \label{spheretimescirclesubsection}

When we glue two solid tori together along their boundary, meridian  to meridian, we obtain $S^2 \times S^1$. Given a torus lattice link $L$ projected onto the separating torus, we can glue a triangular prism from each solid torus together along a square face from each on the separating torus to obtain a thickened bigon. See Figure \ref{thickened bigon square tangle}. Thus we can decompose $L$ into thickened bigon tangles.By Theorem \ref{Main theorem separating}, a similar theorem to our previous cases holds.  

\begin{figure}[htbp]
    \centering
    \includegraphics[scale=.4]{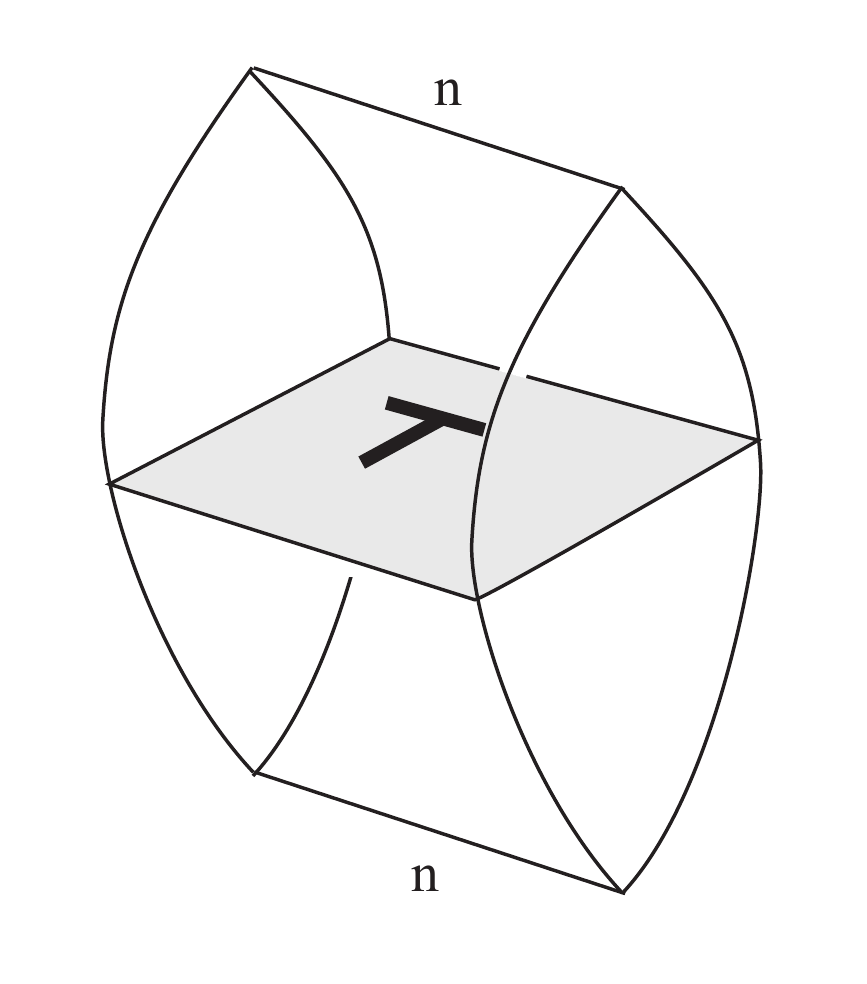}
    \caption{A square tangle in a thickened bigon.}
    \label{thickened bigon square tangle}
\end{figure}

\begin{theorem}\label{S^2 times S^1 theorem}
    Let $L$  be an $(2m, 2n)$-torus lattice link in $S^2 \times S^1$, consisting of $4mn$ thickened bigon tangles $\{ \cT_{ij}\}$. If each tangle $\cT_{ij}$ is $(2,2n)$-hyperbolic, then $L$ is hyperbolic, and its volume satisfies
            \[
                \vol(M) \geq \sum_{i,j} \volsolid{(2,2n)}{\cT_i}.
            \]
\end{theorem}

\section{Various hyperbolicities for tangles} \label{hyperbolicitytangles}

Define a saucer tangle to be a {\it braid} if all strings of the tangle move monotonically left to right from the left face to the right face, intersecting each disk in the saucer parallel to the faces exactly once. We note that if a tangle is $2n$-hyperbolic, then its composition with a braid is still $2n$-hyperbolic with the same $2n$-volume. This is immediate because the composition of an $\ell$-string braid with its reflection yields a trivial braid of $\ell$ strings. Hence, the $2n$-replicant of a tangle $T$ composed with a braid yields the $2n$-replicant of $T$ alone.  

\begin{theorem}\label{Saucer hyperbolicity theorem}
    Suppose $\cT$ is a $2n$-hyperbolic tangle in a saucer for some $n \geq 1$.
    Then, $\cT$ is $2m$-hyperbolic for all $m \geq n$.
\end{theorem}
\begin{proof}
    It suffices to prove that $\cT$ is $2(n+1)$-hyperbolic, which is by definition showing  $D^{2(n+1)}(\cT)$ is hyperbolic. Let $L'$ be the link  that is given by the $2(n+1)$-replicant of  $\cT$. So we want to show $S^3 \setminus L'$ is hyperbolic.
    
    We are assuming $D^{2n}(\cT)$ is  hyperbolic, which means that the link $L$ in $S^3$ that is the $2n$-replicant of $\cT$ is hyperbolic. Define $\bS$ to be the collection of surfaces that decompose $D^{2n}(\cT)$ into homeomorphic copies of $\cT$, and let $E$ be the intersection set of the surfaces. Note that in the case $n=1$, $E = \emptyset$.
    Because $D^{2n}(\cT)$ is hyperbolic and reflection in any of the surfaces is a symmetry of the manifold, Theorem \ref{fixed surface lemma} implies that each surface is totally geodesic, which immediately implies it is incompressible in $D^{2n}(\cT)$. 
    
    Let $\bS'$ be the collection of surfaces that decompose $D^{2(n+1)}(\cT)$ into homeomorphic copies of the saucer tangle $\cT$. Let $E'$ be their intersection set. Note that $E' \neq \emptyset$ for all $n$. 
    
    We begin the proof by showing that each surface in $\bS'$ is incompressible.
    Then, we will use this to prove that $L'$ is nontrivial, and use this to prove that essential disks induce essential spheres.
    Then, we again use incompressibility to prove that essential spheres induce essential spheres in $D^{2n}(\cT)$, so $2n$-hyperbolicity precludes essential disks and spheres in $D^{2(n+1)}(\cT)$.
    Then, we prove that $L'$ is a prime link, precluding certain essential annuli.
    Given a torus in $D^{2(n+1)}(\cT)$, we use this together with a strategy of restricting to pieces and replicating in order to prove that there are no essential tori in $D^{2(n+1)}(\cT)$.
    Lastly, we use the properties of $L'$ already proved to show that $L'$ is not Seifert fibered, and use \cite{Hatcher, BM} to conclude that $D^{2(n+1)}(\cT)$ has no essential annuli, implying that it is hyperbolic.
    
    \sauceremph{Each surface in $\bS'$ is incompressible.}
    Suppose for contradiction that $\Sigma \in \bS'$ is compressible.
    Choose a compressing disk $D$ for $\Sigma$ that has the least number of intersections with $E'$, and then the least number of simple closed curve intersections with $\bS'$ among such disks. 
    The intersection graph of $D$ with $\bS'$ on $D$ will be a collection of arcs with both endpoints on the boundary and a collection of simple closed curves. 
    If there is a simple closed intersection curve with a surface $\Sigma'$, we replace the disk $D$ with an innermost such disk on $D$, which does not intersect any other surfaces if it exists.  This contradicts the minimality of the intersection set for $D$, so there are no simple closed curves of intersection. 
    
    If there are arcs of intersection on $D$,  choose an outermost arc of intersection on $D$ cutting a bigonal disk $D'$ from $D$. Then $D'$ is a bigonal disk in a particular saucer $S_i$  with its edges in the  faces of $S_i$. If both edges are in the same face of $S_i$, then when replicated in $D^{2n}(\cT)$, it yields a set of spheres, each of which must bound balls in $S^3 \setminus L$. Thus we could have isotoped $D$ to eliminate intersections in $D^{2(n+1)}(\cT)$, a contradiction to minimality.
    
    If the two edges of $D'$ are on distinct faces of $S_i$, replicating $D'$ yields a sphere in $D^{2n}(\cT)$, which must be inessential. 
    Hence it bounds a ball. Thus $D'$ must cut a ball from the saucer tangle in $S_i$. This means that $D$ could have been isotoped to eliminate the intersection curve in $D^{2(n+1)}(\cT)$, contradicting minimality. This proves that the surfaces in $\bS'$ are all incompressible. 
    
    \sauceremph{The link $L'$ is nontrivial.}
    Suppose first that the link $L'$ is trivial. Then $L'$ consists of a single component bounding a disk $D$. By incompressibility of the surfaces in $\bS'$, we can remove all simple closed curves of intersection with any of the surfaces with $D$. For any surface $\Sigma$ in $\bS'$, consider an outermost intersection arc in $D \cap \Sigma$ on $D$.
    It cuts a disk $D'$ from $D$, the boundary of a regular neighborhood of which is a compressing disk for $\Sigma$, since $\Sigma$ has at least four punctures by $L'$.
    Thus, $D$ and therefore the link $L'$ cannot intersect any of the surfaces in $\bS'$, again a contradiction to our construction. Therefore $L'$ cannot be a trivial knot.
    
    \sauceremph{The manifold $D^{2(n+1)}(\cT)$ has no essential disks or spheres.}
    Note that for a nontrivial link, the boundary of a regular neighborhood of an essential disk in its complement is an essential sphere.
    Hence we can prove that there are no essential disks in $D^{2(n+1)}(\cT)$ by proving that there are no essential spheres.
    
    Suppose now that there is an essential sphere 
    $S$ in $D^{2(n+1)}(\cT)$. Because of the incompressibility of the surfaces in $\bS'$, we can isotope the sphere to miss each in turn until it misses all of the surfaces. Hence, it now resides in a single saucer. Thus, there is a copy $S'$ that exists in $D^{2n}(\cT)$ (in fact $2n$ copies). But in $D^{2n}(\cT)$, $S'$ must bound a ball. Hence it bounds a ball in the single saucer. Thus, $S$ bounds a ball in $D^{2(n+1)}(\cT)$, a contradiction to the essentiality of the sphere.
    
    \sauceremph{The link $L'$ is prime.}
    Before moving on to  essential tori, we first prove that the link $L'$ in $S^3$ is prime. Suppose not, and let $S$ be the essential (incompressible, boundary-incompressible)  twice punctured sphere in the complement of $L'$. Consider its intersections with the surfaces in $\bS'$, and assume we have chosen $S$ to have a minimum number of intersection curves. If it intersects one of the surfaces $\Sigma$ in a closed curve that is trivial on the punctured sphere, then we can isotope to remove the intersection. If it intersects one of the surfaces in a closed curve that is nontrivial on the punctured sphere and that avoids the axis of the starburst on $\Sigma$, and all intersection curves avoid the axis, then taking an innermost such on $S$  yields a disk $D$ that is once-punctured on $S$ and that is properly embedded in a single saucer. This reflects to a twice-punctured sphere $S'$ in $D^{2n}(\cT)$, which must then bound an unknotted arc to one side in $L$ since $L$ must be prime. 
    If $n \geq 2$, then there are multiple such twice-punctured spheres so it must be that they contain a trivial arc to the inside.
    However, then $D$ cuts a ball from the saucer in which it lies which contains an unknotted arc in $D^{2(n+1)}(\cT)$, and we can isotope $D$ to remove the intersection, contradicting minimality of the number of intersection curves.
    
    If $n = 1$, it could be that the unknotted arc is to the outside of $S'$. But then the opposite face of the saucer contains only one endpoint of the tangle, a contradiction to our assumption that the number of endpoints of the tangles in a face is at least two. 
    
    Finally, suppose that $S$ intersects a surface $\Sigma$ in a curve nontrivial on $S$ that intersects the axis $E'$ of the starburst on $\Sigma$.  Then $S$ intersects all of the surfaces in $\bS'$. The set of intersection arcs forms a $2(n+1)$-valent graph on $S$ with an even number of vertices. Choose a component of this graph that is innermost on $S$, in the sense that all of its complementary regions but one do not contain further components of the graph. Dropping the other components of the graph, we now have a connected $2(n+1)$-valent graph on $S$ with an even number of vertices. 
    
    In fact, there can be no odd cycles in $G$. This is because the cyclic order of pieces around a vertex must be preserved. Hence, for any edge in $G$, the cyclic orders must be clockwise around one endpoint and counterclockwise around the other. In particular, this means $G$ must be a bipartite graph, with all even-length cycles.
    
    For such a graph on a sphere, define $f_{2n}$ to be the number of complementary faces with $2n$ edges. Then if $v$, $e$ and $f$ are the number of vertices, edges and faces, by Euler characteristic, we know $v-e+f = 2$, where $e = (n+1) v$, so $-\frac{n}{n+1}e + f = 2$. However, $$f = f_2 + f_4 + f_6 +\dots$$ and $$e = \frac{2f_2 + 4f_4 + 6f_6 + \dots}{2}$$ Plugging these into the Euler equation and algebra yields:
    
    $$\frac{f_2}{n+1} +\frac{1-n}{n+1}f_4 + \frac{1-2n}{n+1}f_6 +\dots = 2$$
    
    Thus, $$f_2 \geq 2(n+1) \geq 4$$
    
    
    At most two of the at least four bigonal faces can be punctured, since $S$ has only two punctures, so there is a bigonal disk face $D'$ that is not punctured. Then by the same argument as before (replicating in $D^{2n}(\cT)$ to obtain a sphere which must bound a ball), we can isotope $S$ in $D^{2(n+1)}(\cT)$ to lower the number of intersections, a contradiction to our choice of $S$.  Thus $L'$ is prime in $S^3$. 
    
    

    \sauceremph{The manifold $D^{2(n+1)}(\cT)$ has no essential tori.}
    Suppose now that there is an essential torus $T$ in $D^{2(n+1)}(\cT)$. By the incompressibility of the surfaces in $\bS'$, we can isotope $T$ to eliminate any trivial intersection curves on the torus with any of the surfaces in $\bS'$.
    
    We first consider the case that the torus does not intersect the axis of the starburst. Then $T$ intersects the collection of surfaces $\bS'$ in a finite collection of parallel curves on $T$. Assume we have chosen $T$ to have a minimal number of these intersections. These curves divide $T$ into a finite collection of annuli $\{A_i\}$ which are properly embedded in $D^{2(n+1)}(\cT) \dbs \bS'$. If there is an annulus $A_i$  in a given saucer $S_i$ with its boundaries on the same copy of a surface $\Sigma$ from $\bS'$, then it replicates in 
    $D^{2n}(\cT)$ to a collection of tori, each made up of two copies of the annulus, and each of which cannot be essential. Let $T'$ be one of these tori. Then it is either boundary-parallel or it is compressible. If it is boundary-parallel, the annulus $A_i$ is boundary-parallel to a segment of a link component in the saucer $S_i$. Then meridianally compressing $T$ in $D^{2(n+1)}(\cT)$ either yields an essential annulus, contradicting primeness of $L'$ or it shows that $T$ was in fact boundary-parallel, also a contradiction. 
    
    If $T'$ is compressible, then there is a disk $D$ that realizes the compression with boundary a longitude on $T'$. The disk $D$ can be isotoped to intersect a single saucer in a disk with boundary one arc a nontrivial arc on the annulus $A_i$ and the other arc on the boundary of the saucer. Hence, we can isotope a nontrivial arc on the annulus $A_i$ into the neighboring saucer, leaving a disk from $T$ still in $S_i$. By incompressibility, we can eliminate this intersection, a contradiction to minimality of the intersection curves.
    
    Thus, all annuli $A_i$ have their two boundaries on two different faces of each saucer. Since a pair of adjacent  saucers and tangles are homeomorphic via reflection, we can consider $A_i$ and $A_{i+1}^R$ in the same saucer $S_i$. Note that both $A_i$ and $A_{i+1}^R$ share one boundary component on the face $F_1$ of $\partial S_i$. 
    
    We now isotope $A_i$ and $A_{i+1}^R$ into general position and consider intersection curves between them. These can take the form of arcs or simple closed curves. If there is a simple closed curve intersection, then if it is trivial on both annuli, by taking innermost such curves, the disks cut from them form a sphere in $S_i$. But as $S_i$ appears in $D^{2n}(\cT)$, this sphere must bound a ball. Hence we can isotope to remove the intersection curve. If one of the two curves is trivial, this contradicts the incompressibility of the torus. 
    We consider the case that both are nontrivial shortly.
    
    \begin{figure}[htbp]
    \centering
    \includegraphics[scale=0.7]{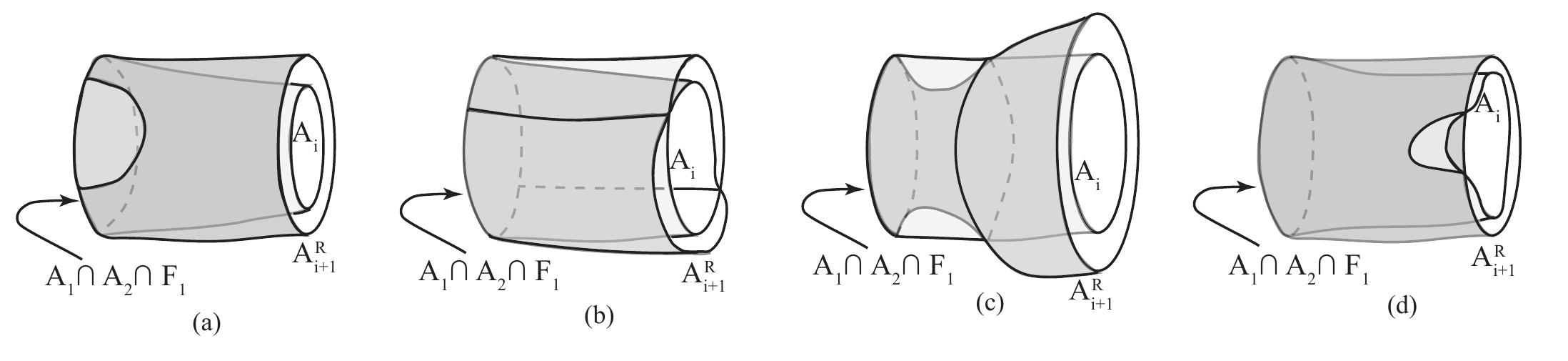}
    \caption{The two annuli $A_i$ and $A_{i+1}^R$ can intersect in various ways.}
    \label{intersectingannuli}
\end{figure}
    
    Suppose there is an intersection arc beginning and ending on their shared boundary, as in Figure \ref{intersectingannuli}(a). Then it must cut a disk from each and again, we can form a sphere that allows us to isotope the intersection away.
    
    
    Suppose there is an intersection curve that runs from the shared boundary on $F_1$ to the opposite boundary on each of $A_i$ and $A_{i+1}^R$ on $F_2$, as in Figure \ref{intersectingannuli}(b). Then the opposite boundaries must cross each other on $F_2$, so there must be an even number of such curves. Then a pair of disks on the annuli cut from two adjacent vertical arcs on each share their bottom edge and their two vertical edges. Together they form a disk with boundary in $F_2$. When we take the neighboring saucer, this disk doubles to a sphere, which appears in these two copies of the saucer and hence in $D^{2n}(\cT)$. 
    But it must bound a ball in $D^{2n}(\cT)$ and hence in $D^{2(n+1)}(\cT)$. Thus, we can isotope the annuli to eliminate these two intersection curves.
    
    So there are now no intersection arcs that end on the shared boundary of the two annuli. Suppose that there is a simple closed curve of intersection. As shown above, it must be nontrivial on both annuli, as in Figure \ref{intersectingannuli}(c). Taking the two annuli it cuts from $A_i$ and $A_{i+1}^R$ that share their shared boundary, we can construct a torus living in $S_i$, with  a longitude on the boundary of $S_i$. Since $S_i$ appears in $D^{2n}(\cT)$, this torus $T'$ must either compress or be boundary parallel there. Because $A_i$ and $A_{i+1}^R$ are incompressible, $T'$ cannot compress to the outside. If it compresses to the inside, then it compresses to a sphere $S' \subset S_i$, which must bound a ball; hence $T'$ bounds a solid torus in $S_i$, and we can isotope $A_{i+1}^R$ to $A_i$, eliminating the intersection curve.  
    
    If $T'$ is boundary parallel, then there is a link component that forms the core curve for the solid torus bounded by $T'$. But this link component wraps once around both annuli. When we reflect $S_i$ in $D^{2n}(\cT)$, this link component is reflected to a second copy that is parallel to the first. That is, they will form the boundary of an annulus that follows either $A_i$ and its reflection or $A_{i+1}$ and its reflection, contradicting the hyperbolicity of $D^{2n}(\cT)$. Thus there are no simple closed curves of intersection.
    
    The final possibility for intersection curves are arcs of intersection that end on the other unshared boundaries of the two annuli, as in Figure \ref{intersectingannuli}(d). These must cut disks from the two annuli. Together these two disks form a disk  with boundary on $F_2$. By incompressibilty, we know that this disk together with the disk it cuts from $F_2$ bounds a ball, and therefore we can isotope the two annuli to eliminate the intersection.
    
    Thus, we now have no intersections curves. The union of the two annuli forms a new annulus, the reflected double over $F_2$ of which is a torus $T''$. That torus exists in $D^{2n}(\cT)$ and therefore it must either be boundary parallel or compressible there. If it is boundary parallel there, there is a single link component inside it that forms the core curve for the solid torus bounded by $T''$. But this contradicts the fact $T''$ was obtained by reflection doubling. If $T''$ is compressible then it must bound a solid torus with no components inside it. Therefore $A_{i+1}^R$ is isotopic to $A_i$.
    
    Thus we see that the original torus $T$ can be isotoped so that all of the $A_i$ are the same in each $S_i$. That is, the torus is made up of $A_1, A_1^R, A_1, \dots, A_1^R$, in the respective copies of $S_1$ and $S_1^R$, making up $D^{2(n+1)}({\cT})$.
    
    Consider the corresponding torus $T'''$ in $D^{2n}(\cT)$. It consists of $2n$ copies of $A_1$. It must be either boundary-parallel or compressible. if it is boundary-parallel the solid cylinder bounded by $A_1$ contains a core curve that is an arc in a link component. But then $T$ is also boundary-parallel in $D^{2(n+1)}(\cT)$.
   
    If $T'''$ is compressible in $D^{2n}(\cT)$, the compression must be to the outside. Thus, there is a compressing disk $D$ which we can choose to minimize the intersection curves with $\bS$ and such that its boundary is a longitude on $T'''$. By incompressibility of the surfaces, we can isotope it so that it has a single arc of intersection with each surface $\Sigma_i$ and all of those arcs meet at the intersection of the surfaces. Then take $D'$ to be its intersection with saucer $S_1$. If we then take the union of the copies of $D'$ and its reflections, we obtain a compressing disk for $T$ in $D^{2(n+1)}(\cT)$. Thus there are no essential tori that avoid $E$ in $D^{2(n+1)}(\cT)$.
    
    
    Suppose now that there is an essential torus in $D^{2(n+1)}(\cT)$ that intersects the axis $E$ of the starburst. We take the essential torus with the least number of intersections with $E$, and then with the least number of curves of intersection with $\bS'$.  Since $E$ is the only place that any two of the surfaces intersect and all the surfaces intersect there, the collection of intersection curves of $T$ with $\bS'$ forms a graph $G$ on $T$ with an even number of vertices, each of valency $2(n+1) \geq 4$. Since $T$ intersects each surface in a disjoint set of curves, each surface $\Sigma_i$ intersects $T$ in a parallel set of $(p_i, q_i)$-curves.
    
    If there is a component of $G$ that lies in either a disk or annulus on $T$, then as in the case of an essential sphere that intersects $E'$, there is a complementary region to $G$ that is a bigon, and this generates a sphere in $D^{2n}(\cT)$ that bounds a ball, and that then allows us to isotope $T$ to reduce the number of intersections with $E'$ in $D^{2(n+1)}(\cT)$, a contradiction. 
    
    So $G$ must consist of a single component such that all of its complementary regions are disks. 
    
    We use a similar argument in the proof of Lemma \ref{Eintersectstorusboundary}. As mentioned previously, $G$ cannot have any bigonal complementary regions. This forces $n= 1$ or $n=2$. In the second case, there are three surfaces in $\bS'$ and  and all complementary regions on $T$ are triangles.  However, as mentioned previously, there can be no odd cycles in $G$, as cyclic order must be preserved between different vertices.
    
    So $n = 1$ and there are two surfaces in $\bS'$, each a punctured sphere. $T$ intersects each one in a collection of curves. Taking an innermost such curve on the sphere, we know it bounds a disk in $S^3$ intersecting $T$ only in its boundary. So it must either be a meridian or a longitude. Since the two sets of curves cannot be parallel, as otherwise they would generate a bigon, one set consists of meridians and the other of longitudes, and $T$ is an unknotted torus in $S^3$. If we take any square $Q$ complementary to $G$ on $T$, it is properly embedded in $\cT$ and $\partial Q$ crosses $E$ four times. It doubles to a sphere living in $D^{2n}(\cT) = D(\cT)$, which by hyperbolicity, bounds a ball. So to one side of $Q$ in $\cT$, there are no link arcs. Therefore, we can isotope a properly embedded arc in $Q$ to $E$. We can use this fact to isotope $T$ in $D^4(\cT)$ to lower the number of intersections with $E$, a contradiction to our choice of $T$.

    \sauceremph{The manifold $D^{2(n+1)}(\cT)$ has no essential annuli.}
    We now consider an essential annulus $A$ in $D^{2(n+1)}(\cT)$. 
    Lemma 1.16 of \cite{Hatcher}, shows that if a manifold is connected, orientable, irreducible and atoroidal, which we have shown our link complement to be, then the existence of an incompressible boundary-incompressible annulus implies the manifold is Seifert fibered. In \cite{BM}, a complete determination of link complements in $S^3$ that are Seifert fibered is given.  In particular, the only options that are nontrivial, prime, and that have no essential tori are a Hopf link, or a $(p,q)$-torus knot $K$, with annulus on the defining torus, and  possibly with one or both of the core curves of the solid tori to either side of the torus, denoted $\alpha_1$ and $\alpha_2$,  included as well. 
    
    There are no spheres in $S^3$ that become incompressible punctured spheres in the Hopf link complement.   So $L'$ is a torus knot with possibly two additional components as described.

    If neither core curve is present, so we are in a torus knot complement, then there are no incompressible punctured spheres to play the role of the surfaces in $\bS'$ (\cite {Tsau}). In the case that $\alpha_1$ or $\alpha_2$ are present, the $(p,q)$-torus knot $K$ lies on a torus $T$ defined by the annulus $A$, with a solid torus or a cored solid torus to each side.
    
   
     Since the fundamental groups of a solid torus and cored solid torus are $\mathbf{Z}$ or $\mathbf{Z} \times \mathbf{Z}$, and a planar surface has fundamental group a free product of $\mathbf{Z}$'s, the only incompressible sub-surfaces of our punctured sphere properly embedded in either one are a disk, a once-punctured disk whose puncture corresponds with a puncture of $\Sigma$, and an annulus whose boundaries are not punctures of $\Sigma$. But note that the boundaries of any of these surfaces can intersect the punctures on $\Sigma$ since they may cross the $(p,q)$-torus knot on $T$. Thus, an incompressible punctured sphere $\Sigma \subset D^{2(n+1)}(\cT)$ can only intersect each side in a disjoint union of disks, once-punctured disks, or annuli. 
     In the case of the annulus, both boundaries are on $T$, and the annulus is boundary-parallel to $T$.
    
    If there is a disk $D$ that has boundary trivial on $T$, then it cuts a ball $B$ from that side of $T$. If there are no further components of $\Sigma \setminus (\Sigma \cap T)$ inside the ball, then either $\partial D$ intersects $L'$ or it does not. If not, we can isotope $\Sigma$ to decrease the number of intersections with $T$. If it does, then we can isotope an arc of $L'$ to $D$ while fixing it endpoints, implying $\Sigma$ is boundary-compressible and hence compressible.
    
    If there  is another component of $\Sigma \setminus (\Sigma \cap T)$ in $B$, we can take an innermost such. It cannot be a punctured disk. If it is a disk, the argument just given eliminates it. If it is an annulus $A'$, that annulus must have boundaries trivial on $T$, and that annulus cuts a solid torus $V$ off of that side of $T$. If $\partial A'$ does not intersect $L'$, we can isotope $A'$ to eliminate the intersections with $T$. If it does intersect $L'$ then there is an arc of $L'$ that can be isotoped fixing its endpoints into $A'$, once again contradicting incompressibility of $\Sigma$. 

    A similar argument eliminates all annuli in $\Sigma \setminus (\Sigma \cap T)$ with boundary trivial in $T$ by taking an innermost such and applying the argument already given.

    By taking an innermost curve of intersection of $\Sigma \cap T$ on $\Sigma$, there is a disk or punctured disk to one side of $T$, and it must have meridianal boundary in the solid torus or cored solid torus to that side of $T$.
    
    Hence, to the other side, there must be an annulus $A'$ with boundaries that are longitudinal in the solid torus or cored solid torus to that side. Such an annulus splits the solid torus or cored solid tori into two solid tori, at most one of which can contain the possibly missing core. To that side, take an innermost such annulus $A''$. That annulus is boundary-parallel to $T$, and if its boundary misses $L'$, we can eliminate the intersections. If its boundaries intersect $L'$, then we can isotope an arc of $L'$ into $A'$, contradicting incompressibility of $\Sigma$.

\end{proof}
\begin{remark}
  Though Theorem \ref{Saucer hyperbolicity theorem} may extend to other special classes of pieces, it is not true for all pieces.
  For instance, the restriction that saucer tangles have at least two endpoints on each face is necessary. If we allowed a saucer tangle that had only one endpoint on a face, we could still choose the tangle so that it is 2-hyperbolic. But it will never be $2m$-hyperbolic for any $m > 1$ as it will not yield a prime knot or link.
  
  
  Another example where Theorem \ref{Saucer hyperbolicity theorem} fails to extend comes from Lemma \ref{Eintersectstorusboundary}.
  Given a $4$-hyperbolic piece $P$ whose $4$-replicant has an intersection locus intersecting a torus boundary component, its $2m$-replicant also has an intersection locus intersecting a torus boundary component.
  Hence Lemma \ref{Eintersectstorusboundary} implies that $P$ is not $2m$-hyperbolic for any $m > 2$.

\end{remark}

We gain the following useful corollary.
\begin{corollary}\label{Saucer composition corollary}
    Suppose $\{\cT_j\}$ is a collection of $n$ $2m$-hyperbolic saucer tangles such that the composition $\cT_i \circ \cT_{i+1}$ is defined for all $i$.
    Then, $\cT_1 \circ \cT_2 \circ \cdots \circ \cT_n$ is $2m$-hyperbolic.
\end{corollary}
\begin{proof}
    Theorem \ref{Saucer hyperbolicity theorem} implies that each $\cT_j$ is $2mn$-hyperbolic, hence Corollary \ref{General composition corollary} implies that the composition is $2m$-hyperbolic.
\end{proof}

The analogy to tangles in $T \times I$ carries to the following theorems.
\begin{theorem}\label{Bracelet theorem2}
    Suppose $L$ is a bracelet link made of a cycle $(\cT_i)_{i=1}^{n}$ of $n \geq 2m$ saucer tangles such that each $\cT_i$ is $2m$-hyperbolic.
    Then, $L$ is hyperbolic.
    If $n = 2m$, then the volume satisfies
    \[
        \vol(L) \geq \sum_i \volsphere{2n}{\cT_i}.
    \]\qed
\end{theorem}

\begin{theorem}\label{tetrahedral hyperbolicity theorem}
    Suppose $\cT$ is a tetrahedral tangle that is $(2m, 2n)$-hyperbolic for some $m, n \geq 1$.
    Then, $\cT$ is $(2r, 2s)$-hyperbolic for all $r \geq m$ and $s \geq n$.
\end{theorem}

\begin{proof} The $(2m, 2n)$-replicant of $\cT$ can be decomposed into $2m$ saucer tangles, each consisting of $n$ copies of $\cT$ and $n$ copies of its reflection. But then the fact the $(2m, 2n)$-replicant of $\cT$ is hyperbolic is equivalent to the fact the $2n$-replicant of the saucer tangle is hyperbolic. But then we know the $2s$-replicant of the saucer tangle is hyperbolic by the theorem, so $\cT$ is  $(2m, 2s)$-hyperbolic. We can repeat the process, only this time decomposing into $2s$ saucer tangles, each consisting of $m$ copies of $\cT$ and $m$ copies of its reflection. Again the theorem implies that this saucer tangle is $2r$-hyperbolic, implying that $\cT$ is $(2r, 2s)$-hyperbolic.  
 
\end{proof}

\section{Explicit Tangles} \label{explicittangles}
In the case of bracelet links, $2$-hyperbolicity of some classes of saucer tangles have been previously studied. We consider one such class here.

\begin{figure}[htbp]
    \centering
    \includegraphics[width=.8\textwidth]{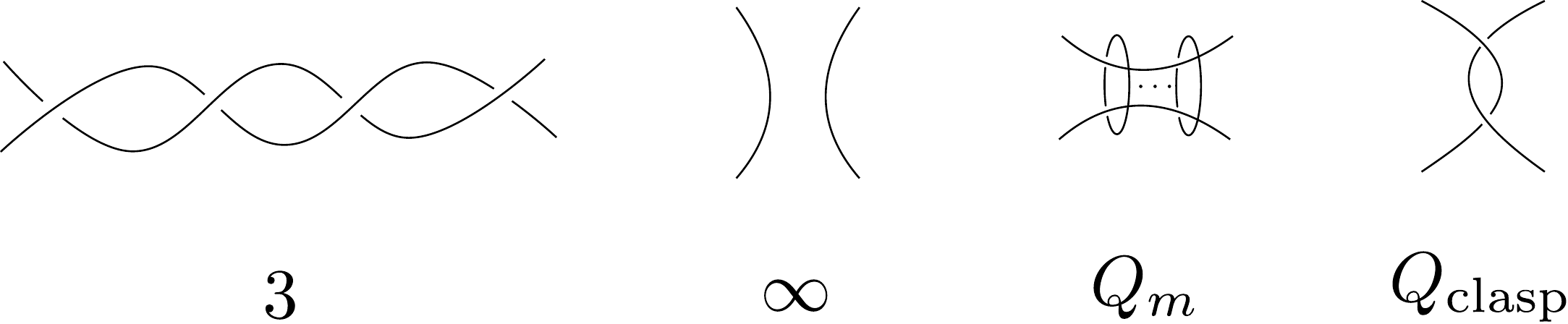}
    \caption{Pictured are some particular examples of arborescent tangles.
    The integer tangles are defined analogously to 3, with tangle $m$ having $m$ crossings.
    The tangle $Q_m$ has $m$ ``loops.''
    }
    \label{Arborescent examples figure}
\end{figure}
\begin{figure}[htbp]
    \centering
    \includegraphics[width=.8\textwidth]{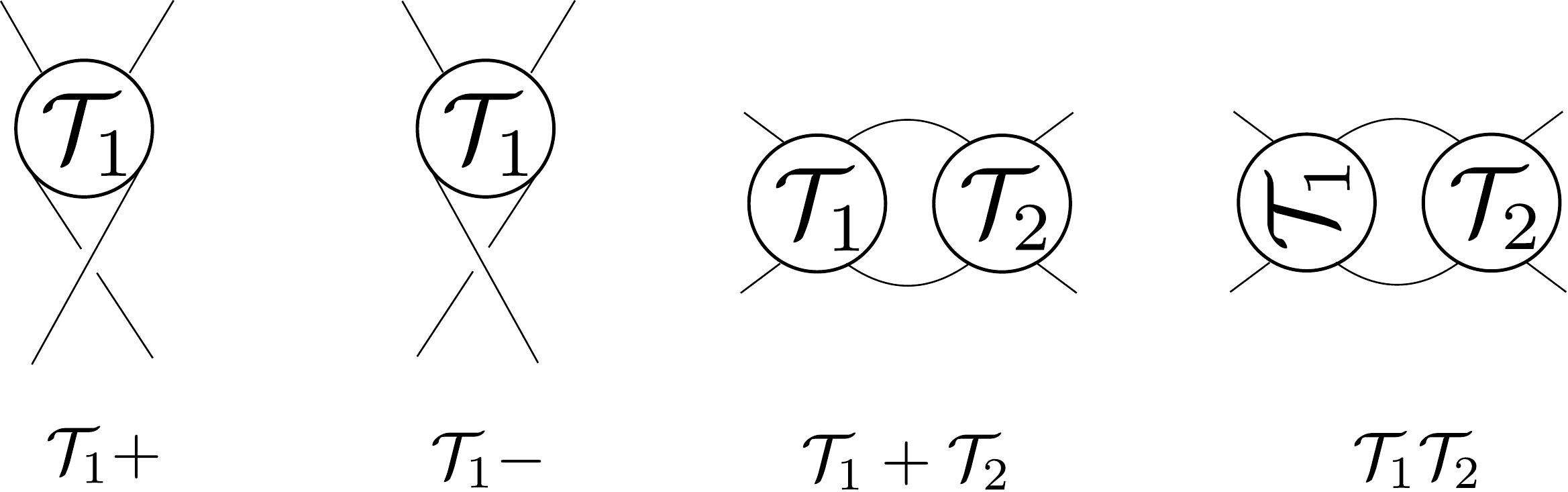}
    \caption{Pictured are the four operations defined by \cite{Conway} to generate arborescent tangles.
    We realize arborescent tangles as saucer tangles with the endpoints on the far left or far right, and hence composition of arborescent tangles corresponds with the sum operation.}
    \label{Arborescent operations figure}
\end{figure}

\subsection{Arborescent tangles}\label{Arborescent tangles subsection}
Rational tangles, as defined by Conway in \cite{Conway} are examples of saucer tangles. We can think of a rational tangle as constructed by starting with either the 0 tangle corresponding to two horizontal strands or the $\infty$ tangle corresponding to two vertical strands.Then we alternately twist the right two endpoints and the  the bottom two endpoints to create a link in the 3-ball. Note that when we start, there is a disk $D$ known as the \emph{defining disk} of the rational tangle, that is properly embedded in the ball and that separates the two strands. It continues to exist as we twist the endpoints, but its boundary becomes more complicated on the boundary of the ball. 

Given a rational tangle in a ball, we can choose  a disk $D'$ on the boundary of the ball that contains two of the endpoints of the tangle, called a \emph{composition disk}. Then given two such rational tangles with choices of disks on the boundary, we can compose the tangles by gluing the one disk to the other.

 A Montesinos tangle is obtained by composing rational tangles in a row, left to right, our composition disks either being the leftside disk or the rightside disk. Arborescent tangle are defined to include rational tangles and the compositions of  arborescent tangles along any composition disks.

We use previous results to determine the $2n$-hyperbolicity of all arborescent tangles.

Let $Q_m$ be the arborescent tangles defined in Figure \ref{Arborescent examples figure}.
We say that an arborescent tangle $\cT$ \emph{contains $Q_m$} if there is a 4-times punctured sphere in $D \times I \setminus \cT$, called a \emph{Conway sphere}, cutting $\cT$ into two tangles, one of which is $Q_m$.
We say that a link contains $Q_m$ if it possesses such a Conway sphere.

The following Theorem was proved in \cite[Thm. 3.22]{Volz}.
It is related to a theorem originally in an unpublished manuscript by Bonahon and Siebenmann (see \cite{Bonahon}), also proved in \cite{Volz} and \cite{Futer} by essentially different means.
\begin{theorem}[\cite{Volz}]\label{Arborescent theorem}
    Suppose $\cT$ is an arborescent tangle.
    Then, $\cT$ is $2$-hyperbolic if and only if none of the following hold:
    \begin{enumerate}
        \item $\cT$ is a rational tangle.
        \item $\cT = Q_m \circ \cT'$ or $\cT = \cT' \circ Q_m$ for some $m \geq 1$ and some arborescent tangle $\cT'$.
        \item $\cT$ contains $Q_m$ for some $m \geq 2$.
    \end{enumerate}\qed
\end{theorem}
 
 Note that a rational tangle $\cT$ can never be 2-hyperbolic since the defining disk doubles to a separating sphere for the resulting link.

We will use this theorem and Corollary \ref{General composition corollary} to determine $2n$-hyperbolicity.
In order to state the result, we define a saucer tangle $\cT$ to be \emph{principally $2n$-hyperbolic} if it is $2n$ hyperbolic and either $n=1$ or $\cT$ is not $2(n-1)$-hyperbolic.
We say that a saucer tangle is \emph{entirely non-hyperbolic} if it is not $2n$-hyperbolic for any $n$.
\begin{corollary}\label{Arborescent corollary}
    Suppose $\cT$ is an arborescent tangle.
    Then, $\cT$ is entirely non-hyperbolic if and only if one of the following hold:
    \begin{enumerate}
        \item $\cT$ is an integer tangle.
        \item $\cT = Q_m \circ \cT'$ or $\cT = \cT'  \circ Q_m$ for some $m \geq 1$.
        \item $\cT$ contains $Q_m$ for some $m \geq 2$.
    \end{enumerate}
    Further, $\cT$ is principally 2-hyperbolic if and only if none of these hold and $\cT$ is non-rational.
    Moreover, $\cT$ is principally 4-hyperbolic if and only if it is a non-integer rational tangle other than $Q_{\text{clasp}}.$
    Last, $\cT$ is principally 6-hyperbolic if and only if it is equal to $Q_{\text{clasp}}.$
\end{corollary}
We use the fact (see  \cite{Kauffman_classification}), that a sum of two rational tangles is rational if and only if one of the two is an integer tangle.
\begin{proof}[Proof of Corollary \ref{Arborescent corollary}]
        We first assume that $\cT$ is an arborescent tangle satisfying one of (1)-(3), and we show that $\cT \circ \cT^R \cdots \widetilde \cT$ satisfies one of the conditions of Theorem \eqref{Arborescent theorem}, where $\widetilde \cT$ stands for either $\cT$ or $\cT^R$ depending on the parity.
    
    Suppose that $\cT$ is an integer tangle.
    Then, by the preceding discussion, we see inductively that 
    \[
        \cT \circ \cT^R \cdots \widetilde \cT =  \cT \circ \prn{\cT^R \cdots \widetilde \cT} 
    \]
    is a rational tangle (in fact an integer tangle).
    Hence Theorem \ref{Arborescent theorem} implies that no such tangles are 2-hyperbolic, so \ref{General composition corollary} implies that $\cT$ is entirely non-hyperbolic.
    
    Suppose instead that $\cT = Q_m \circ \cT'$  or $  \cT' \circ Q_m$ for some $m \geq 1$.
    Then, $\cT^R \circ \cT \cdots \widetilde \cT$ contains $Q_{2m}$, so we may assume that condition (3) holds.
    In fact, if condition (3) holds for $\cT$, then it holds for $\cT \circ \cT^R \cdots \widetilde \cT$, so Theorem \ref{Arborescent theorem} and Corollary \ref{General composition corollary} together imply that $\cT$ is entirely non-hyperbolic.
    
    The corollary for principal 2-hyperbolicity is a restatement of Theorem \ref{Arborescent theorem}.
    For 4-hyperbolicity, note that $\cT$ is not an integer tangle, so $\cT^R$ is also not an integer tangle.
    Hence $\cT \circ \cT^R$ is not a rational tangle.
    
    Since $Q_m$ contains more than two components, any rational tangle satisfies conditions (2) and (3).
    Further, conditions (2), (3)  together imply that $\cT$ neither contains $Q_m$ for $m \geq 2$ nor attains equality with $Q_m \circ T$ or $T \circ Q_m$ for $m \geq 1$.
    Hence $\cT \circ \cT^R$ is 2-hyperbolic and $\cT$ is not 2-hyperbolic.
    Hence $\cT$ is principally 4-hyperbolic.
    
    The corollary is well-known for $Q_{\text{clasp}}$ (see for instance \cite{NR}).
\end{proof}


\subsection{Alternating tangles} \label{alternatingtangles} 
We now explore a class of 2-hyperbolic solid cylindrical tangles related to alternating saucer tangles.


\begin{definition}
A saucer tangle $\cT$ is an \emph{alternating tangle} if it possesses a projection to a cross-sectional bigon of the saucer  such that the crossings alternate along every strand.
\end{definition}


Since they are topologically indistinguishable, a saucer tangle $\cT$ determines a canonical solid-cylinder tangle $\cT_{C}$.
We show the following:
\begin{proposition}\label{Hyperbolic alternating tangle proposition}
    Suppose $\cT$ is a $2m$-hyperbolic alternating saucer tangle which intersects each end at two points.
    Then, $\cT_{C}$ is a 2-hyperbolic solid-cylinder tangle with volume satisfying
    \[
        \volsolid{2}{\cT_C} >  \volsphere{2m}{\cT}.
    \]
\end{proposition}
We use the following construction.
\begin{definition}
    A saucer tangle $\cT$ determines another saucer tangle $\widehat \cT$ with a projection identical to that of $\cT$ except that all crossings are switched between under-crossings and over-crossings.
    Using this, the length $2m$ cycle $\prn{\cT,\widehat \cT^R,\dots,\widehat \cT^R}$ determines an alternating bracelet link $D^{2m}_{\text{mod}}(\cT)$, called the \emph{modified $2m$-replicant}.
\end{definition}
 If $\cT$ is $2m$-hyperbolic, then $\widehat \cT$ is as well, and by Theorem \ref{Bracelet theorem}, the link$D^{2m}_{\text{mod}}\prn{\cT}$ is also hyperbolic, with volume satisfying 
    \[
        \vol\prn{D^{2m}_{\text{mod}}{\cT}} > 2m \volsphere{2m}{\cT}
    \]

\begin{proof}[Proof of Proposition \ref{Hyperbolic alternating tangle proposition}]
    The link $D^{2m}_{\text{mod}}(\cT)$ is alternating by construction. Since it is hyperbolic, it must be nonsplit, prime  and it must not be a torus link.
    
    Using this, we may construct a \emph{generalized augmented alternating link} by adding in a single unknotted component along the central axis of the bracelet link $D^{2m}_{\text{mod}}(\cT)$.
    By \cite{Adams_generalized_augmented}, this link is hyperbolic, and by \cite[Theorem 3.4(c)]{Thurston_kleinian} , its volume is greater than the volume of $D^{2m}_{\text{mod}}(\cT)$.
    In fact, the complement of this link is homeomorphic to that of the link in the open solid torus determined by $\prn{\cT_C,\widehat \cT^R_C,\dots,\widehat \cT^R_C}$.
    Call this link $D^{2m}_{\text{mod}}(\cT_C)$.
    
    There is a collection of disjoint twice-punctured meridianal disks in the open solid torus which separate the copies of $\cT_C$ and $\widehat \cT^R_C$ in $D^{2m}_{\text{mod}}(\cT_C)$ from the rest of the link.
    By \cite{Adams_thrice_punctured}, these twice-punctured disks are totally geodesic, and cutting and then regluing along any subset of them will yield another hyperbolic manifold. In particular,  $\cT_C$ is 2-hyperbolic.
    Since $\volsolid{2}{\cT_C} = \volsolid{2}{\widehat \cT_C^R}$, Theorem \ref{Solid torus theorem} implies that $\vol\prn {D^{2m}_{\text{mod}}(\cT_C)} = 2m \volsolid{2}{\cT_C}$.
    We may summarize these inequalities as follows:
    \begin{align*}
        \volsphere{2m}{\cT}
        &= \frac{\vol\prn{D^{2m}(\cT)}}{2m} \\
        &\leq \frac{\vol\prn{D^{2m}_{\text{mod}}(\cT)}}{2m}\\
        & < \frac{\vol\prn{D^{2m}_{\text{mod}}(\cT_C)}}{2m}\\
        &= \volsolid{2}{\cT_C}.
    \end{align*}
\end{proof}

Note that because the twice-punctured meridianal disks are totally geodesic, the link $J$ obtained by taking a single copy of $\cT$, and gluing its left endpoints to its right endpoints (generating a ``bracelet link'' of length 1) and then removing a single unknotted component corresponding to the axis of the bracelet link is hyperbolic and its volume is equal to $\volsolid{2}{\cT_C}$. A  corollary applies to Montesinos tangles. 

\begin{corollary}
    Let $\cT_C$ be a non-integer Montesinos tangle in a solid cylinder.
    Then $\cT_C$ is 2-hyperbolic.
\end{corollary}

\begin{proof}
    Let $\cT$ be the corresponding Montesinos tangle in a saucer.
    Then, $\cT$ decomposes into a composition of non-integer rational tangles (by including an integer tangle on the tail end of the preceding rational tangle if necessary).
    By Corollary \ref{Arborescent corollary} these are $2m$-hyperbolic for some $m$, and by \cite{Kauffman_classification} they are alternating.
    Since $\cT_C \circ \cT'_C = (\cT \circ \cT')_C$, this implies that $\cT_C$ decomposes into $2$-hyperbolic solid cylindrical tangles.
    Then, Theorem \ref{Main theorem inductive} implies that $\cT_C$ is 2-hyperbolic, as desired.
\end{proof}


\section{Application to Other Tangle Reflection Graphs}\label{tanglereflectiongraphs}



Up to now, we have used the cyclic graph on $2n$ vertices or the $(2n, 2m)$-torus lattice graph to create links. We defined the replicants of a given tangle $\cT$ in terms of the  link obtained by alternately reflecting the tangle $\cT$ around the graph.

We now generalize the construction. Let $G$ be a finite connected $r$-valent bipartite graph with bipartition of the vertices into two sets $U$ and $V$ such that $G$ is embedded in an orientable 3-manifold $M$. Let $\bS$ a collection of surfaces in $M$, each intersecting $G$ transversely at a subset of the edges such that for each such surface $\Sigma$, there  is an orientation-reversing involution of $M$, which we call a reflection,  with fixed point set $\Sigma$. Further, assume that for each pair of vertices sharing an edge, there is a distinguished such reflection that switches the two vertices, sends $G$ back to itself, switches $U$ and $V$ and sends the collection $\bS$ back to itself. 
Further assume that in the finite group $\mathcal{G}$ generated by the collection of reflections, the only element that fixes a vertex is the trivial element.
Then we call $(M, G,\bS)$ a {\it reflection graph} of $M$.
The group $\mathcal{G}$ is fully determined as the group of isometries of $M$ generated by reflections about $\bS$.

Let $L$ be a link formed by inserting tangles $\cT_1, \cT_2, \dots, \cT_{2n}$ into the $2n$ vertices of $G$ such that the nonzero number of outgoing strands along each edge match up from the two vertices. For each such tangle $\cT$,  we can form the $G$-replicant $D^G(\cT)$ of the tangle by removing all tangles except for it and then using the reflections to reflect $\cT$ across each outgoing edge. Then we repeat the process with all subsequent edges. 

To see that the definition of the replicant is well-defined, we need to show that there are $r$ distinct edge classes at each vertex under the action of $\mathcal{G}$.  Suppose not, and that there is a nontrivial  element $\mu$ of $\mathcal{G}$ that identified two edges that share a vertex $y$. Then since $\mu$ cannot fix $y$, it must send the other endpoint of the first edge, call it $x$ to $y$ and $y$ to the other endpoint of the second edge, call it $z$. 
    
But because $y$ and $z$ share an edge, there must be a reflection $\lambda$ that sends $y$ to $z$ and $z$ to $y$. But then $\lambda \circ \mu(y) = y$. hence, $\lambda \circ \mu$ must be the identity so $\mu = \lambda^{-1} = \lambda$. But $\lambda$ sends $z$ to $y$ whereas $\mu$ sends $x$ to $y$, a contradiction.  
    
The fact there are exactly $r$ equivalence classes of edges follows by reflecting to get from one vertex to an adjacent one. Then the $r$ equivalence classes with representatives at the first vertex will generate representatives of the same classes at the next vertex. Repeating for adjacent vertices will eventually make this true at all vertices. 
    
Since the graph is bipartite, each cycle is even and this construction is well-defined. If the complement of the $G$-replicant of $\cT$ is hyperbolic, we say $\cT$ is a $G$-hyperbolic tangle with volume $\vol(D^G(\cT))/ |G|$.

If we have a combinatorial proof that a collection of $2n$ copies of the $G$-replicants for the tangles can be cut and reglued, always along totally geodesic surfaces, in order to obtain $2n$ copies of $M \setminus L$, then we know that the analog of Theorem \ref{Bracelet theorem} applies. That is to say that under these circumstances, if each tangle we insert at the vertices of the graph is $G$-hyperbolic, then the resulting link is hyperbolic with volume greater than the sum of $1/2n$ times the volumes of the $G$-replicas. Under these circumstances, we say that $(M, G, \bS)$ is a {\it tangle reflection graph} for $M$.    

\begin{theorem} \label{tangle reflection graph} Let $(M, G, \bS)$ be a tangle reflection graph with $2n$ vertices in a compact orientable 3-manifold $M$ with possible torus boundaries.
If there is a orientation-reversing symmetry $\phi$ of $M$ with fixed point set an orientable  surface $S$, such that $G$ lies in $S$ and $\phi$ preserves each surface in $\bS$,  then there exists a new tangle reflection graph $(M, G', \bS')$ for $M$, where $G'$ is the product of $G$ with the graph $P_1$ consisting of one edge and two vertices and $\bS'$ is the collection $\bS$ together with the surface $S$.
\end{theorem}

\begin{proof}
Push $G$ off $S$ in a choice of normal direction, and call that copy of $G$ the upper copy. Take a second copy of $G$ that is the reflection of $G$ by $\phi$, and call that the lower copy. Then we connect the corresponding vertices of the two copies of $G$ by edges that are perpendicular to $S$, which we call {\it vertical edges}. This yields an embedding of $G'$ in $M$ such that $(M, G', \bS')$ is a reflection graph.

It remains to show that $(M, G', \bS')$ is a tangle reflection graph. Let $L'$ be a link formed from $(G', \bS')$, with tangles $\cT_1, \dots \cT_{2n}$ in the upper copy of $G$ and tangles $\cT_{2n+1} , \dots, \cT_{4n}$ in the lower copy of $G$. 

We form two new links $L''$ and $L'''$ from $(G',\bS')$ such that the first consists of $\cT_1, \dots \cT_{2n}$ in the upper copy of $G$, and then their reflections across the vertical edges in the lower copy of $G$. The second consists of $\cT_{2n+1} , \dots, \cT_{4n}$ in the lower copy of $G$ and then their reflections across the vertical edges in the upper copy of $G$. Treating each pair of tangles along a vertical edge as a single tangle, we know that because $(G, \bS)$ is a tangle reflection graph, we can form the $G$-replicants of each of the now $2n$ tangles in $L''$ and then by cutting along totally geodesic surfaces and regluing, we can create $2n $ copies of $L''$ in $M$. Similarly, from the $2n$-replicants of the tangles in $L'''$ we can create $2n$ copies of $L'''$ in $M$. In each case, the sum of the volumes of the $2n$ copies of the complements of the links have volumes at least as great as the sum of the volumes of the $G$-replicants. But each of these links has a reflection surface such that reflecting across it switches the two copies of each $\cT_i$ within it. So we can cut the manifolds open along the totally geodesic surfaces and reglue to obtain $4n$ copies of $L$ in $M$. Hence the theorem follows.
\end{proof}

The following examples are immediate from the theorem. 
Our discussion of bracelet links showed that a cycle graph $C_{2n}$ with $2n$ vertices is a tangle reflection graph in $S^3$.
That graph can be isotoped to live on a sphere $S$ that is the fixed-point set for a reflection symmetry of the 3-sphere. Therefore Theorem \ref{tangle reflection graph} implies that we obtain a prism graph with $4n$ vertices that is a tangle reflection graph.  Note that each vertex in the graph is trivalent so our tangles have three bundles of strands exiting. 
We can reapply Theorem  \ref{tangle reflection graph} to the prism graph, since it also can be isotoped to live on a sphere. 
Thus, we generate a hyperprism graph of $8n$ vertices, as in Figure \ref{hyperprism}, which must also be a tangle reflection graph. Of course, this last example is also a $(2n, 4)$-torus lattice graph, so we already knew that it was a tangle reflection graph. 

\begin{figure}[htbp]
    \centering
    \includegraphics[width=.7\textwidth]{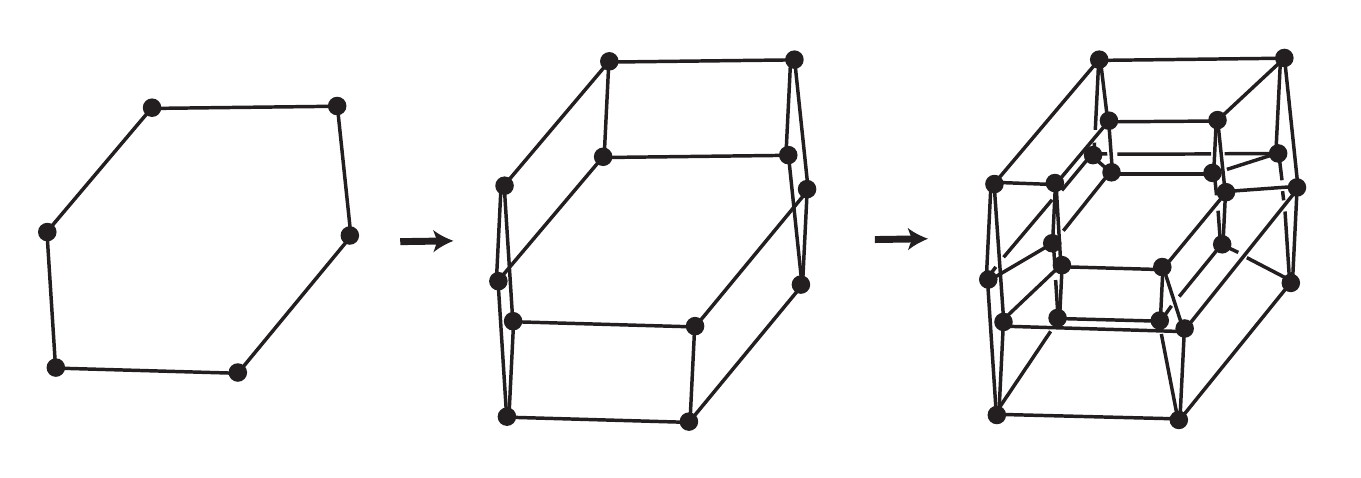}
    \caption{That the cycle graph is a tangle reflection graph in $S^3$ implies the prism graph is also, which implies the hyperprism graph is as well.}
    \label{hyperprism}
\end{figure}

Unlike what happened with the cycle graph, in these graphs, we can have any positive number of strands running along an edge, including a single strand.

Another example is generated by the $(2m, 2n)$-torus lattice graph, denoted $G$, which we know to be a tangle reflection graph. If we take that graph to lie in either the intermediate torus in $T \times (0,1)$  or to lie in the torus splitting $S^2 \times S^1$ into two solid tori, then there is a reflection of the manifold that preserves the torus and therefore the graph. Hence, Theorem \ref{tangle reflection graph} implies that the graph obtained as a product of $G$ with $P_1$ is also a tangle reflection graph in those manifolds.

As  further potential tangle reflection graphs, we can look to the 1-skeleta of the Archimedean solids.  
 Three of these solids are relevant to this construction since the only cycles in the 1-skeleta have even length, namely the truncated octahedron, truncated cuboctahedron, and truncated icosidodecahedron. 

We observe that in $S^3$, the graphs determined by the $1$-skeleta of these three solids are reflection graphs. We expect that volume bounds similar to those proven for simpler configurations would hold for links in $S^3$ with tangles at the vertices of these Archimedean solids, and connecting strands along the edges. However, the combinatorics quickly becomes difficult. 

But note that if these can be shown to be tangle reflection graphs, then Theorem \ref{tangle reflection graph} implies that their products with $P_1$ would also be tangle reflection graphs.

\section{Volume Calculations} \label{volumecalculations}

In this section we give tables of volumes of tangles living in various objects, all calculated using the SnapPy computer program \cite{SnapPy}.

A \emph{rational square tangle} is defined to be a rational tangle projected onto a square, having one endpoint at the midpoint of each edge of the square. Table \ref{rational square tangle} gives the volumes of some such tangles in a variety of contexts. The first column specifies the object in which the rational square tangle is embedded. The second column specifies the manifold $M$. The third column specifies the $m$ for the $(2, 2m)$-replicant we perform on the tangle when finding its volume.  The subsequent columns correspond to particular rational tangles given in Conway notation, where the first four are integer tangles.

Note that on the standardly embedded torus, $(2, 2m)$-replicant means $2$ columns and $2m$ rows.  We do not give the $(2m,2)$-replicant volumes as in $T \times I$, $\mbox{vol}_{T \times I}^{(2m, 2n)}(\cT)= \mbox{vol}_{T \times I}^{(2n, 2m)}(\cT)$. In  $D^2 \times S^1$, $\mbox{vol}_{D^2 \times S^1}^{(2m, 2)}(\cT) = m \mbox{vol}_{D^2 \times S^1}^{(2, 2)}(\cT)$. And in $S^3$, $\mbox{vol}_{S^3}^{(2m, 2n)}(\cT)$ gives the same volume as  $\mbox{vol}_{S^3}^{(2n, 2m)}(\cT')$, where $\cT'$ is obtained from $\cT$ by rotating $90^\circ$ and then changing all crossings. In the case of the tangles listed, this generates a tangle that is the reflection of the initial tangle, and hence the volumes are equal. 

As a note, caution should be exercised. For instance, in $S^3$ and $D^2 \times S^1$, both the integer 2-tangle and the integer 3-tangle are not $(2m,2)$-hyperbolic for any positive integer $m$ for one choice of crossings in the tangle and are $(2m,2)$-hyperbolic for all positive $m$ for the other choice of crossings. See Example \ref{tettangleexample} for the case where the 2-tangle is $(2m,2)$-hyperbolic for all positive $m$. In $S^2 \times S^1$, the integer 2-tangle and 3-tangle are not $(2m,2)$-hyperbolic for any choice of crossings.



\begin{table}[htbp]
  \centering
  \caption{$\mbox{vol}_M^{(2, 2m)}$ of rational square tangles}
    \begin{tabular}{|c|c|c||r|r|r|r|r|}
    \hline
    Embedded in a & Link is in $M$ & $m$ & \multicolumn{1}{|c|}{2} & \multicolumn{1}{|c|}{3} & \multicolumn{1}{|c|}{4} & \multicolumn{1}{|c|}{5} & \multicolumn{1}{|c|}{2 1} \\ \hline \hline
    
    cube & $T \times I$ & 1 & 4.3692852 & 6.26446133 & 7.37277416 & 8.05714498 & 7.48167784 \\ \hline
    
    \multirow{3}{*}{triangular prism} & \multirow{3}{*}{$D^2 \times S^1$} & 1 & 3.66386238 & 5.07470803 & 5.82605812 & 6.26446133 & 6.57223395 \\ 
    & & 2 & 4.20148819 & 5.97374285 & 6.98517306 & 7.59952104 & 7.26467498 \\
    & & 3 & 4.2953372 & 6.92143228 & 7.20028164 & 7.85266168 & 7.38599241 \\ \hline
    
    \multirow{3}{*}{tetrahedron} & \multirow{3}{*}{$S^3$} & 1 & 3.13223067 & 4.3062076 & 4.90649786 & 5.24442274 & 5.81283664 \\
    & & 2 & 3.538345513 & 4.89348601 & 5.60919297 & 6.02345588 & 6.39075686 \\
    & & 3 & 3.60861463 & 5.53828532 & 5.73062363 & 6.1583781 & 6.4921938 \\ \hline
    
    \end{tabular}
  \label{rational square tangle}
\end{table}

 We may define a \emph{rational thickened-cylinder tangle} to be obtained by having one endpoint of a rational tangle at the midpoint of each side of a square, thickening the square and then gluing the north and south  pair of opposite sides of the thickened square. See Table \ref{integer cylindrical tangles} for $\mbox{vol}_{T \times I}^2$ of example integer thickened-cylinder tangles (c.f. Definition \ref{Thickened torus tangle definition}). Note that when we have rational thickened-cylinder tangles, since the tangles only intersect the left and right boundaries of the thickened cylinder once, these left and right boundaries become once-punctured annuli. Therefore they are totally geodesic in the hyperbolic structure of the resulting link complement in the thickened torus. Hence, the 2-volumes of the constituent tangles add to give the volume of the resulting link exactly.

\begin{table}[htbp]
  \centering
  \caption{$\mbox{vol}_{T \times I}^2$ of integer cylindrical tangles}
    \begin{tabular}{|c|c|c|c|}
    \hline
    2 & 3 & 4 & 5 \\ \hline \hline
    5.33348957 & 7.32772475 & 8.35550215 & 8.92931782 \\ \hline
    \end{tabular}
  \label{integer cylindrical tangles}
\end{table}

We may also define a \emph{reciprocal integer saucer tangle} to be an integer tangle with NE and NW strands going out one side of a square cross section of a saucer, and SE and SW strands going out the other. 
Note that $D^{2m}(p/q) = D^{2m}(-p/q)$, so it suffices to consider these for nonnegative integers.
The reciprocal saucer tangle $1/0 = \infty$ and $1/1 = 1$ are entirely non-hyperbolic.
Further, Corollary \ref{Arborescent corollary} implies that $1/2 = Q_{\text{clasp}}$ is principally $6$-hyperbolic and $1/n$ is principally $4$-hyperbolic for all $n > 2$.

Table \ref{reciprocal integer saucer tangle} gives the beginning of a table of $\mbox{vol}^{2m}_{S^3}$ for integer reciprocal saucer tangles. 
Note that $\mbox{vol}^{2m}_{S^3}(\cT)$ of the  reciprocal integer saucer tangle $1/2$ corresponds to $\frac{1}{2n}$ the volume of the minimally twisted $2m$-chain link.

\begin{table}[htbp]
  \centering
  \caption{$\mbox{vol}^{2m}_{S^3}$ of reciprocal integer saucer tangles}
    \begin{tabular}{|c||r|r|r|r|}
    \hline
    $m$ & \multicolumn{1}{|c|}{1/2} & \multicolumn{1}{|c|}{1/3} & \multicolumn{1}{|c|}{1/4} & \multicolumn{1}{|c|}{1/5} \\ \hline \hline
   
    2 & 0 & 3.13223067 & 4.30620760 & 4.90649786 \\
    3 & 2.44257492 & 4.45025692 & 5.38411452 & 5.86524434 \\ 
    4 &  3.01152301 &  4.85098130 &   5.72360375 &  6.17292141\\
    5 &  3.25515403 & 5.28582982 &   5.87549250 & 6.31134186\\
    \hline
    \end{tabular}
  \label{reciprocal integer saucer tangle}
\end{table}

We note that for a reciprocal integer saucer tangle,  by Proposition \ref{Hyperbolic alternating tangle proposition},  $\vol\prn{D^{2m}(\cT)}$ is less than the volume $\volsolid{2}{\cT_C}$. In fact, as $m$ increases, the volumes will approach the upper bound from below. This follows because $\volsolid{2}{\cT_C} = \volsolid{2m}{\cT_C}$ since $D^{2m}(\cT_C)$ is an $m$-fold cyclic cover of $D^2(\cT_C)$. And $D^{2m}(\cT)$ comes from surgery on $D^{2m}(\cT_C)$, with surgery coefficients increasing with $m$.

Hence, in Table \ref{limitvolumes}, we include some of these limit volumes as well. So each column in Table \ref{reciprocal integer saucer tangle} has values approaching the corresponding entry in Table \ref{limitvolumes}.  Note that the limit volumes in Table \ref{limitvolumes} themselves limit to 7.32772474 from below, which is the volume of the Borromean rings. So all $2m$-replicants of reciprocal integer tangles have volume less than 7.32772474.

\begin{table}[htbp] 
  \centering
  \caption{$\volsolid{2}{\cT_C}$ for integer reciprocal tangles}
    \begin{tabular}{|c|c|c|c|}
    \hline
   1/2 & 1/3 & 1/4 & 1/5 \\ \hline
    3.66386237 & 5.33489567 & 6.13813879  & 6.5517432897 \\ \hline
    \end{tabular}
  \label{limitvolumes}
\end{table}

\printbibliography


\end{document}